\documentclass[11pt]{amsart}

\usepackage[top=1in, bottom=1in, left=1in, right=1in]{geometry}
\usepackage[dvipsnames]{xcolor}
\usepackage[colorlinks=true, pdfstartview=FitV, linkcolor=RoyalBlue,citecolor=ForestGreen, urlcolor=blue]{hyperref}
\usepackage{xfrac}

\usepackage[english]{babel}
\usepackage{graphicx} 
\usepackage{amsthm} 
\usepackage{amsmath}
\usepackage{amssymb}
\usepackage{amsthm} 
\usepackage{mathrsfs} 
\usepackage{bbm} 
\usepackage{enumerate} 
\usepackage[shortlabels]{enumitem}
\usepackage{xcolor}
\usepackage{nicefrac}
\usepackage{subcaption}

\usepackage{hyperref}

\newcommand{\N}{\mathbb{N}}
\newcommand{\Z}{\mathbb{Z}}

\newcommand{\R}{\mathbb{R}}
\newcommand{\C}{\mathbb{C}}
\newcommand{\T}{\mathbb{T}}

\newcommand{\dif}{\,\mathrm{d}}

\theoremstyle{theorem}
\newtheorem{thm}{Theorem}[section]
\newtheorem{prop}{Proposition}[section]
\newtheorem{cor}{Corollary}[section]
\newtheorem{lemma}{Lemma}[section]
\theoremstyle{definition}
\newtheorem{defi}{Definition}[section]
\theoremstyle{remark}
\newtheorem{Rem}{Remark}[section]

\numberwithin{equation}{section}

\title[Instability, non-uniqueness and global smooth solutions for SQG]{Unstable vortices, sharp non-uniqueness with forcing, and global smooth solutions for the SQG equation}
\author{\'Angel Castro, Daniel Faraco, Francisco Mengual, and Marcos Solera}
\date{\today}

\begin{document}
	
\maketitle

\begin{abstract}
We prove non-uniqueness of weak solutions to the forced $\alpha$-SQG equation with Sobolev regularity $W^{s,p}$ in the supercritical regime $s < \alpha + \frac{2}{p}$, covering the 2D Euler equation ($\alpha = 0$), the Surface Quasi-Geostrophic equation ($\alpha = 1$), and the intermediate cases. 
A key step is the construction of smooth, compactly supported vortices that exhibit non-linear instability. As a by-product, we show existence of global smooth solutions to the (unforced) $\alpha$-SQG equation that are neither rotating nor traveling.
\end{abstract}

\setcounter{tocdepth}{1} 

\tableofcontents

\section{Introduction and main results}\label{sec:intro}

We address the non-uniqueness of weak solutions to the generalized Surface Quasi-Geostrophic ($\alpha$-SQG) equation\begin{subequations}\label{eq:SQG}
\begin{align}
\partial_t\theta + v\cdot\nabla\theta=f,\\
v=-\nabla^\perp\Lambda^{\alpha-2}\theta,\label{eq:SQG:2}\\
\theta|_{t=0}
=\theta^\circ,
\end{align}
\end{subequations}
posed on $\R^2$, 
for some given external force $f$ and initial datum $\theta^\circ$, 
where $\Lambda=(-\Delta)^{\nicefrac{1}{2}}$.
The velocity $v(t,x)$ is recovered from the scalar $\theta(t,x)$ through \eqref{eq:SQG:2}, which we find convenient to refer to here as the \textit{$\alpha$-Biot-Savart law}
\begin{equation}\label{eq:BiotSavart}
v(t,x)
=C_\alpha\int_{\R^2}\frac{(x-y)^\perp}{|x-y|^{2+\alpha}}\theta(t,y)\dif y,
\end{equation}
where $C_\alpha=\frac{2^{\alpha}}{2\pi}\frac{\Gamma(1+\frac{\alpha}{2})}{\Gamma(1-\frac{\alpha}{2})}>0$.

The $\alpha$-SQG model
was proposed by C\'ordoba, Fontelos, Mancho and Rodrigo \cite{CFMR05} as an interpolation between the 2D Euler equation ($\alpha=0$) and the SQG equation ($\alpha=1$). The Euler equation describes the motion of an ideal, incompressible fluid, capturing vorticity evolution (see e.g.~\cite{MajdaBertozzi02}). 
The SQG equation extends this concept to geophysical fluids, modeling temperature evolution in quasi-geostrophic atmospheric and oceanic flows (see e.g.~\cite{pedlosky}).
This family of active scalar equations has attracted significant interest since the seminal work of Constantin, Majda, and Tabak \cite{CMT}, which highlighted analogies between SQG and the 3D Euler equation.
We will generally refer to $\theta$ as the temperature, unless specifically discussing 2D Euler, where $\theta=\nabla^\perp\cdot v$ is the vorticity.

The potential lack of uniqueness is a central issue in fluid mechanics, relevant to both the question of determinism and its relationship with chaos and turbulence (see e.g.~\cite{Hopf51,Ladyzenskaja69,Lorenz69,Yudovich63}).
From a mathematical point of view, given a parametrized family of function spaces $X^s$,
where $s$ typically quantifies the regularity of the solution $\theta$,
a natural question is to determine
\begin{equation}\label{eq:sU}
s_U\equiv
\text{ threshold for Uniqueness}.    
\end{equation}
Notice that \eqref{eq:sU} depends on $\alpha$ and the baseline space. In terms of Sobolev regularity $\theta\in W^{s,p}$, 
we have the following natural conjecture
\begin{equation}\label{conjecture:sU}
s_U(\alpha,p)=\alpha+\frac{2}{p}.
\end{equation}

The well-posedness of the $\alpha$-SQG equation has been proved in the subcritical regime $s>\alpha+\frac{2}{p}$: globally for the 2D Euler equation 
by Bourguignon and Brezis \cite{brezis},
locally for the SQG equation by Chae \cite{dongho}, and locally, at least for $p=2$, in the intermediate cases by Chae and Wu \cite{chaewu}.
Heuristically, this is because $\theta$ is transported by a velocity $v\in W^{s+1-\alpha,p}\subset C^{1+}$. 
These results, based on standard energy methods, readily extend to the case with forcing in the natural space
\begin{equation}\label{forceintegrability}
f\in L^1([0,T],W^{s,p}\cap\dot{H}^{\frac{\alpha-2}{2}}),
\end{equation}
where $\dot{H}^{\frac{\alpha-2}{2}}$ is the energy space associated with the Hamiltonian of the $\alpha$-SQG equation.

In the celebrated work \cite{Yudovich63}, Yudovich proved that the 2D Euler equation ($\alpha=0$) is also globally well-posed at the critical point $(s,p)=(0,\infty)$. In fact, Yudovich's energy method can be adapted to prove uniqueness, assuming existence, along the critical line for $\alpha=0$,
and other critical spaces such as $H^{\alpha+1}$ for $0 \leq \alpha \leq 1$.
We caution that well-posedness implies uniqueness, but the converse is not true. 
Counterexamples by Elgindi and Jeong \cite{ElgnidiJeong17}, C\'ordoba and Mart\'inez-Zoroa \cite{CordobaZoroa22}, and Jeong and Kim \cite{jeongkimsqg} show critical exponents where ill-posedness occurs, with solutions losing regularity instantaneously in spaces where uniqueness holds. 
See also \cite{Jeong21,czo24,CZOpp,CLZpp} for ill-posedness in supercritical spaces, as well as \cite{BahouriChemin94,BrueNguyen21,CDE22} for the related issue of propagation of regularity.

Overall, the main obstruction for proving Conjecture \eqref{conjecture:sU} is to establish non-uniqueness in the supercritical regime $s<\alpha+\frac{2}{p}$. In this work, we show that this holds in the case with forcing \eqref{forceintegrability}.

Conjecture \eqref{conjecture:sU} remains an outstanding problem in the case $f=0$. Remarkably, the construction of solutions that violate both uniqueness and Hamiltonian conservation
has been successfully addressed in recent years through the method of convex integration, introduced into fluid dynamics by De Lellis and Székelyhidi \cite{DeLellisSzekelyhidi09}.  
See Sections \ref{sec:Hamiltonian} and \ref{sec:renormalization} for discussions on convex integration and conservation laws, as well as the reviews \cite{DLSzREview17, DLSzReview19, DLSzReview22, BVReview21}.
Regarding the 2D Euler equation, $\theta$ is the vorticity, and the Hamiltonian equals the kinetic energy at the velocity level.
In terms of H\"older regularity, Giri and Radu \cite{GiriRadu24} recently solved the 2D Onsager conjecture by constructing non-trivial Euler flows
$v\sim\Lambda^{-1}\theta\in C^{\nicefrac{1}{3}-}$. More recently, Bru\`e, Colombo, and Kumar \cite{BCKpp2} constructed the first non-trivial Euler flows with integrable vorticity $\theta\in L^{p}$ for some small $p>1$ below the Onsager threshold.
Regarding the SQG equation, the construction of non-trivial solutions was first established by Buckmaster, Shkoller, and Vicol \cite{BSV19} and was recently improved to $ \Lambda^{-1} \theta \in C^{1-} $ by Dai, Giri, and Radu \cite{DGRpreprint}, and by Isett and Looi \cite{IsettLooiPreprint}, thereby solving the Onsager conjecture for SQG. More recently, Zhao \cite{zhao2024} extended this result to $ \Lambda^{-1} \theta \in C^{\frac{1+2\alpha}{3}-} $ for $0\leq\alpha\leq 1$.
Notice that all the above convex integration constructions lie within the regime where the Hamiltonian does not need to be conserved.
Since this is strictly contained within the broader regime where non-uniqueness is expected, it leaves a significant gap in the theory of non-uniqueness.

In the case with forcing, similar non-uniqueness results have been established using convex integration by Bulut, Huynh, and Palasek \cite{BHPpreprint2}, and by Dai and Peng \cite{DaiPengpreprint}. Indeed, for the 3D Euler equation \cite{BHPpreprint1} it is possible to surpass the Onsager threshold up to $v\in C^{\nicefrac{1}{2}-}$.
However, to the best of our knowledge, non-unique solutions with some Sobolev or Hölder regularity at the level of $\theta$ had not been constructed until now for the $\alpha$-SQG equation, even in the case with forcing.

This work is devoted to exploring another mechanism for non-uniqueness, based on the self-similar instability approach introduced by Jia and \v{S}ver\'{a}k in the context of the 3D Navier-Stokes equation \cite{JiaSverak15}. 
In a remarkable pair of papers \cite{Vishikpp1,Vishikpp2}, Vishik successfully established non-uniqueness for the forced 2D Euler equation ($\alpha=0$) below the Yudovich well-posedness class $(s,p)=(0,\infty)$. 
The present paper extends Vishik's non-uniqueness theorem to the forced $\alpha$-SQG equations for the full supercritical Sobolev regime and 
$0 \le \alpha \le 1$.

\begin{thm}[Sobolev non-uniqueness]\label{thm:main:0}
Let $0\leq\alpha\leq 1$, $s\geq 0$ and $1\leq p\leq\infty$ satisfying
\begin{equation}\label{eq:s<}
s<\alpha+\frac{2}{p}.
\end{equation}
There exist $T>0$ and a force \eqref{forceintegrability}
such that there are uncountably many solutions
\begin{equation}\label{thetaintegrability}
\theta_\epsilon\in L^\infty([0,T],W^{s,p}\cap\dot{H}^{\frac{\alpha-2}{2}}),
\end{equation}
to the $\alpha$-SQG equation \eqref{eq:SQG} starting from  $\theta^\circ=0$.
\end{thm}

As we will explain in more detail in Section \ref{sec:sketch},
Vishik's result (Theorem \ref{thm:Vishik}) corresponds to the specific case $\alpha = s = 0$. More precisely, it can be derived as a corollary of a refined 
version of Theorem~\ref{thm:main:0}, namely, our Theorem \ref{thm:main}.

\begin{figure}[h!]
    \centering
    \begin{subfigure}{0.45\textwidth}  
        \centering
    \includegraphics[width=\textwidth]{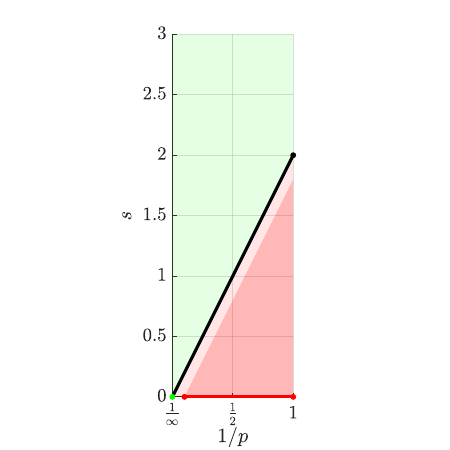}
        \caption{Case $\alpha=0$ (2D Euler).}
        \label{fig:image1}
    \end{subfigure}
    \begin{subfigure}{0.45\textwidth}  
        \centering
    \includegraphics[width=\textwidth]{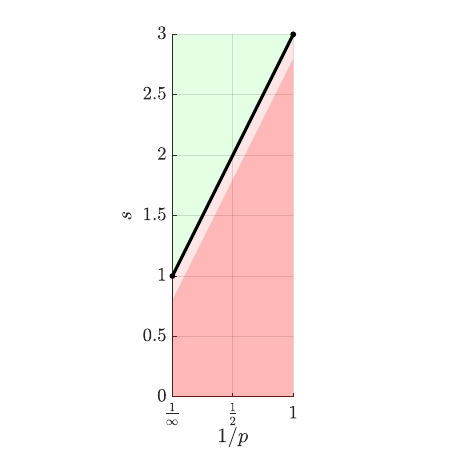}
        \caption{Case $\alpha=1$ (SQG).}
        \label{fig:image2}
    \end{subfigure}
    \caption{Representation of the Sobolev spaces $W^{s,p}$ in terms of the variables $1/p$ and $s$. The green region corresponds to the well-posedness regime $s>\alpha+2/p$. The black segment corresponds to the critical line $s=\alpha+2/p$. The light red region corresponds to the (expected) non-uniqueness regime $s<\alpha+2/p$. The red area $s\leq\alpha+2/p-\varepsilon$ corresponds to the non-uniqueness with forcing, as established in Theorem \ref{thm:main}, where $\varepsilon>0$ is arbitrarily small. The red segment corresponds to Vishik's non-uniqueness theorem ($\alpha=s=0$). In this context, the green point $(1/p,s)=(0,0)$ corresponds to the Yudovich well-posedness class.}
    \label{fig:figura_completa}
\end{figure}

\subsection{Improvements and corollaries}\label{sec:improvements}
The regularity of the force and the solutions can be described in various other function spaces. Thus, Theorem \ref{thm:main:0} allows for certain improvements and by-products, which we summarize below and will later state as theorems and corollaries in the introduction.

Firstly, let us remark that both the force $f$ and the main solution $\theta_0$ from Theorem~\ref{thm:main:0} belong to $C_c^\infty$ for positive times. Moreover, for every $k\in\N$, the other solutions $\theta_\epsilon$ can be upgraded to be in $C_c^{k}$ for positive times. These estimates inevitably deteriorate as $t\to 0$. 
Although spatial regularity cannot exceed the critical line, temporal regularity can easily be improved to Lipschitz, or even enhanced with additional derivatives, by choosing the parameter $a$ appropriately in Section \ref{sec:selfsimilarcoordinates}.

As a consequence of the smoothness of our solutions, Theorem \ref{thm:main:0} can be generalized to any space that is scaling-supercritical with respect to the $\alpha$-SQG equation and, at the same time, includes smooth, compactly supported functions, such as Besov or Triebel–Lizorkin spaces. In this work, we state the theorems in the context of classical Sobolev spaces for the sake of simplicity.

In the scale of Lebesgue spaces ($s=0$), Theorem \ref{thm:main:0} shows non-uniqueness within the DiPerna-Majda existence class for the 2D Euler equation \cite{DiPernaMajda87}, which corresponds to Vishik's non-uniqueness theorem, as well as within the Resnick-Marchand existence class for the SQG equation \cite{Resnick95, Marchand}, and the Chae-Constantin-C\'ordoba-Gancedo existence class for the generalized SQG equation \cite{CCCGW12}, all in the presence of forcing.

In the scale  of Hölder spaces, which we express as $\Lambda^{-1}\theta\in C^\gamma$ to facilitate comparison with the Onsager exponent, we have the following natural conjecture 
\begin{equation}\label{eq:gammaU}
\gamma_U(\alpha)=1+\alpha.
\end{equation}
Similarly to the Sobolev case, in the subcritical regime $\gamma > 1 + \alpha$, we would have $v \sim \Lambda^{\alpha - 1} \theta \in C^{1+}$. Well-posedness in subcritical H\"older spaces has been established for the 2D Euler equation by Wolibner \cite{Wolibner33} and H\"older \cite{Holder33}, and for the SQG equation by Wu \cite{Wu2005} and by  Ambrose, Cozzi, Erickson, and Kelliher \cite{ACEK23}. Uniqueness can be deduced, assuming existence, in some critical Hölder spaces as a consequence of the weak-strong uniqueness principle (see e.g.~\cite{Wiedemann18}). In contrast, there are known counterexamples of such critical Hölder spaces that exhibit strong ill-posedness.
This was first established by Bourgain and Li for the Euler equation \cite{BourgainLi15}, and also by Elgindi and Masmoudi \cite{ElgindiMasmoudi20}. These results were later extended to both the SQG and generalized SQG equations by C\'ordoba and Mart\'inez-Zoroa \cite{CordobaZoroa22, CordobaZoroa24}.

By the extension of Morrey's inequality to fractional Sobolev spaces \cite{DPV12}, we deduce, as a corollary of Theorem \ref{thm:main:0}, non-uniqueness for the full supercritical H\"older regime in the case with forcing.

\begin{cor}[H\"older non-uniqueness]\label{thm:nonuniqueness:Holder}
In the setting of Theorem \ref{thm:main:0}, for every fixed
$$
\gamma<1+\alpha,
$$
the solutions and the force can be chosen to satisfy
$$
\Lambda^{-1}\theta_\epsilon\in L^\infty([0,T],C^\gamma),
\quad\quad
\Lambda^{-1}f\in L^1([0,T],C^\gamma).
$$
\end{cor}

Finally, we emphasize that our solutions satisfy both the Hamiltonian identity (Corollary \ref{cor:Hidentity}) and the renormalization property (Corollary \ref{cor:renormalized}) in the presence of forcing. See Sections \ref{sec:Hamiltonian} and \ref{sec:renormalization} for the precise definitions of the Hamiltonian \eqref{eq:Hamiltonian} and the renormalization \eqref{eq:renormalized}, respectively.

Up to this point, we have discussed the improvements and corollaries regarding non-uniqueness with forcing. A key intermediate step of independent mathematical and physical interest involves the explicit construction of smooth, compactly supported vortices that exhibit non-linear instability (Theorem \ref{thm:unstablevortex}).
As a corollary, we show the existence of global smooth solutions to the (unforced) $\alpha$-SQG equation that are $n$-fold symmetric, non-stationary, and converge to a radial temperature as $t\to\infty$ (Corollary \ref{cor:global}). See Section \ref{sec:preliminaries} for the precise definition of vortices and $n$-fold symmetry.

\subsection{Main novelties}\label{sec:novelties}
The starting point of our work is the groundbreaking result of Vishik on the non-uniqueness of solutions to the forced 2D Euler equation. See \cite{Vishikpp1,Vishikpp2} for the original papers and the book \cite{ABCDGMKpp} by Albritton, Bru\`e, Colombo, De Lellis, Giri, Janisch, and Kwon for a careful and detailed exposition.
Vishik's work raised the question of the validity of analogous results for the more singular generalized SQG equations, and more importantly, for the SQG equation itself. However, it appears quite challenging to implement Vishik’s argument for these non-linear and singular counterparts of the 2D Euler equation. Motivated by this, we devised an alternative  approach for the 2D Euler equation in \cite{CFMSpp}, which not only substantially simplifies Vishik’s proof but also appears robust enough to be adapted to other situations. This is what we have accomplished in the current paper. However, essentially every step of the construction requires new ideas, with the SQG case being particularly delicate due to the non-compact relationship between the velocity and the advected scalar. We will next explain the parallels and innovations we consider noteworthy.

The proof of Theorem \ref{thm:main:0} follows the general strategy presented in \cite{CFMSpp}, which is that of \cite{Vishikpp1,Vishikpp2,ABCDGMKpp} except for the construction of the unstable vortex. This approach consists of four steps: 
\begin{enumerate}[(1)]
    \item Construct a piecewise constant unstable vortex.
    \item Regularize the unstable vortex.
    \item Prove that the vortex is self-similarly unstable.
    \item Prove that the vortex is non-linearly unstable.
\end{enumerate}
Let us summarize the main differences with respect to our proof in the 2D Euler case:

\begin{enumerate}[1.]
    \item While for $\alpha=0$ Step (1) reduces to an elementary computation, for $\alpha>0$ the non-locality introduced by the $\alpha$-Biot-Savart law necessitates dealing with a parametric integral that cannot be computed. Despite this lack of explicitness, we are able to show that there are piecewise constant unstable vortices by obtaining precise asymptotics in Section \ref{sec:In}.
    \item Step (2) requires correctly expanding the linear stability equation into a leading term and a lower-order term that behaves like $\frac{\varepsilon^{1-\alpha}}{1-\alpha}$. In the limiting case $\alpha=1$, the roles of these operators are reversed, completely changing the nature of the problem. In Section \ref{sec:regularization:a=1}, we show that even in this situation it is possible to handle the operators and apply a fixed point argument to find a regularized unstable vortex.
    \item Step (3) involves decomposing the linearization into an operator that generates a contraction semigroup and a compact operator $K$. However, this operator $K$ loses compactness for SQG. We bypass this obstacle in Section \ref{sec:selfsimilar:SQG} by decomposing $K$ into a skew-adjoint operator, which preserves the growth bound of the semigroup, along with a compact commutator.
    \item Step (4) is carried out in \cite{Vishikpp1,Vishikpp2,ABCDGMKpp,CFMSpp} through energy estimates on first-order derivatives in polar coordinates, applied in the correct order. Here, achieving the full supercritical regime $s < \alpha + \frac{2}{p}$ requires working with higher-order derivatives. The resulting energy space consists of functions that vanish at the origin. Moreover, to manage the many terms, we introduce an inductive argument in Section \ref{sec:energyestimates}, which significantly streamlines the proof. In addition, a delicate estimate involving the commutator is necessary in the SQG case.
\end{enumerate}

During the finalization of this work, a clever trick discovered by Golovkin in the 1960s \cite{Golovkin64} was highlighted by Dolce and Mescolini \cite{DolceMescolinipp}, showing that by absorbing the quadratic interaction of the linear instability into the forcing term, it is possible to bypass the non-linear instability step and construct two distinct solutions (instead of uncountably many). In Section \ref{sec:Golovkin}, we revisit this trick in the context of the $\alpha$-SQG equation.

Nonetheless, our non-linear instability approach has several remarkable by-products. First, it provides a robust framework for transitioning from linear to non-linear instability in arbitrarily high Sobolev regularity. Second, by reversing the arrow of time, it shows the existence of global smooth solutions that are neither rotating nor traveling. See the next Sections \ref{sec:unstablevortices} and \ref{sec:globalsolutions} for more details. 
Additionally, it could be useful for showing the potential lack of uniqueness in the absence of forcing, as it offers a more precise representation of the $\alpha$-SQG dynamics.

\subsection{Unstable vortices}\label{sec:unstablevortices}

The stability analysis of simple steady states
in the 2D Euler equation is a classical problem that dates back to the works of Kelvin \cite{Thomson_1880}, Rayleigh \cite{rayleigh1880stability}, Orr \cite{Orr}, Taylor \cite{taylor1923stability}, and many others. Although some explicit shear flows were known to be unstable (see e.g.~\cite{DrazinReid04}), to the best of our knowledge, the first rigorous construction of an unstable vortex was achieved by Vishik (see also \cite{ABCDGMKpp}). Vishik's approach is based on bifurcation from neutral modes, a method that traces back to Tollmien’s work on shear flows \cite{Tollmien1935}. This technique was later rigorously applied by Lin to shear flows in a bounded channel and in $\R^2$, and to rotating flows in an annulus \cite{Lin03}.

Although there was numerical evidence supporting the existence of unstable vortices for the $\alpha$-SQG equation (see e.g.~\cite{Held_Pierrehumbert_Garner_Swanson_1995}), to the best of our knowledge, they had not yet been rigorously constructed for $0<\alpha\leq 1$.
We refer to the work of Friedlander and Shvydkoy \cite{friedlandershvydkoy} for the analysis of the spectrum of SQG and the existence of unstable eigenvalues for an oscillating shear flow.

In this work, we adapt our construction of smooth, compactly supported, unstable vortices for the 2D Euler equation \cite{CFMSpp} to the full range $0\leq\alpha\leq 1$. In particular, this rigorously proves the existence of non-linearly unstable vortices for the $\alpha$-SQG equation with $f=0$. 
Moreover, we prove a stronger property, namely that the vortex exhibits non-uniqueness at $t=-\infty$.

\begin{thm}[Non-uniqueness at $t=-\infty$]\label{thm:unstablevortex}
Let $0\leq\alpha\leq 1$ and $n\geq 2$. There exists a vortex
$$
\bar{\Theta}\in C_c^\infty\cap\dot{H}^{\frac{\alpha-2}{2}},
$$
such that, for every $m>\alpha+1$, there is an $n$-fold symmetric solution to the (unforced) $\alpha$-SQG equation
$$
\theta\in C((-\infty,0],H^m\cap\dot{H}^{\frac{\alpha-2}{2}}),
$$ 
satisfying
$$
\|\theta(t)-\bar{\Theta}\|_{H^m}\sim e^{\Re\lambda t},
$$
as $t\to -\infty$, for some $\lambda\in\C$ with $\Re\lambda>0$. In particular, 
$
\|\theta(t)\|_{\dot{H}^{\frac{\alpha-2}{2}}}
=\|\bar{\Theta}\|_{\dot{H}^{\frac{\alpha-2}{2}}}
$
for all $t\leq 0$.
\end{thm}

Above, the complex number $\lambda$ corresponds to the eigenvalue of the linear instability associated with $\bar{\Theta}$.
By starting at a finite time $t_0<0$ sufficiently far back in time, the non-uniqueness at $t=-\infty$ (Theorem \ref{thm:unstablevortex}) implies the non-linear instability as it is usually stated (Corollary \ref{cor:unstablevortex}). Notice that, by shifting the time interval, we can take $0=t_0<T$.

\begin{cor}[Non-linearly unstable vortex]\label{cor:unstablevortex}
Let $0\leq\alpha\leq 1$ and $n\geq 2$. 
There exists $\delta> 0$ such that, for every $\varepsilon>0$ and $m>\alpha+1$, there exist $T>0$ and an $n$-fold symmetric solution to the (unforced) $\alpha$-SQG equation 
$$
\theta\in C([0,T],H^m\cap\dot{H}^{\frac{\alpha-2}{2}}),
$$
satisfying
\begin{equation}\label{eq:perturbation}
\|\theta(0)-\bar{\Theta}\|_{H^m}
\leq\varepsilon,
\end{equation}
and
\begin{equation}\label{eq:perturbationT}
\|\theta(T)-\bar{\Theta}\|_{H^m}\geq\delta,
\end{equation}
where $\bar{\Theta}$ is the vortex from Theorem \ref{thm:unstablevortex}.
\end{cor}

Corollary \ref{cor:unstablevortex} shows energy pump of certain solutions to the $\alpha$-SQG equation, aligning with the work of Kiselev and Nazarov \cite{kiselevnazarov} for SQG. Double exponential growth of the gradient of the vorticity  was established by Kiselev and \v{S}ver\'{a}k for the 2D Euler equation \cite{KiselevSverak14} in the unit disc. Exponential growth of the gradient of the vorticity was provided for 2D Euler in the torus by Zlato\v{s} \cite{Zlatos2015}. For SQG, He and Kiselev \cite{hekiselev}
showed exponential growth of the $C^{1,\gamma}$-norm of the temperature in $\mathbb{R}^2$.


As we introduced in Section \ref{sec:novelties}, the construction of $\bar{\Theta}$ is divided into several steps. The first step involves constructing a piecewise constant unstable vortex. While this idea was inspired by previous constructions for shear flows in the 2D Euler equation (see e.g.~\cite[Chapter 4]{DrazinReid04}), we remark that our construction operates at the level of the temperature, in contrast to the classical Rayleigh stability equation, which deals with the stream function. 
Secondly, we find a suitable regularized unstable vortex through a fixed-point argument. The idea that regularized unstable shear flows retain their instability was proposed by Grenier for the Euler equation \cite{Grenier00}. In this work, we rigorously prove that this is the case for vortices in the $\alpha$-SQG equation.
Finally, the last step consists of proving that the non-linear effects are negligible within a certain time interval.

In \cite{Grenier00} Grenier developed a program to transition from linear to non-linear instability. In particular, Grenier showed that, given a linearly unstable smooth shear flow for the Euler equation, it is possible to deduce the properties \eqref{eq:perturbation} and \eqref{eq:perturbationT}.
The transition from linear to non-linear instability has been also analyzed by Bulut and Dong \cite{BulutDong21} for the SQG equation with dissipation. 
Grenier's strategy for the Euler equation appears robust enough to be applied to other equations, such as the $\alpha$-SQG equation. Thus,  by combining this approach with our Theorem \ref{thm:L}, it is plausible that this could provide an alternative proof of Corollary \ref{cor:unstablevortex}.
In order to upgrade the non-linear instability to non-uniqueness at $t=-\infty$ (Theorem \ref{thm:unstablevortex}), we apply a delicate bootstrapping argument to obtain additional exponential decay uniformly in time. 

 Recently, Dolce and Mescolini \cite{DolceMescolinipp} presented another flexible strategy for constructing linearly unstable vortices for the 2D Euler equation. An interesting question is whether a similar variety of unstable vortices also exists for $\alpha>0$.
It is worth noting that the non-uniqueness of Leray-Hopf solutions to the forced 3D Navier-Stokes equation was established by Albritton, Bru\`e, and Colombo \cite{ABC22} by carefully adapting Vishik's unstable vortex in the axisymmetric setting. The unstable vortices constructed in \cite{CFMSpp,DolceMescolinipp}, while different from Vishik's, satisfy the conditions of \cite{ABC22}, offering an alternative example of non-unique Leray-Hopf solutions with forcing.

\subsection{Global smooth solutions}
\label{sec:globalsolutions}

A significant distinction between the 2D Euler equation ($\alpha=0$) and the SQG equation ($\alpha=1$) is the potential lack of global well-posedness in the SQG case, as well as in the intermediate cases ($0<\alpha<1$). Whether the SQG equation develops finite-time singularities remains an outstanding open problem in mathematical fluid mechanics (see e.g.~\cite{CMT,Cordoba1998,CCW11}). 

In this regard, the construction of global smooth solutions that exhibit non-trivial dynamics has attracted considerable attention from the mathematical community in recent years.
In \cite{CCGSglobal}, the first author, in collaboration with C\'ordoba and G\'omez-Serrano, constructed the first non-trivial family of global smooth solutions to the SQG equation, consisting of rotating $3$-fold symmetric solutions with $C^4$ regularity bifurcating from a smooth annular profile.
The argument required a computer-assisted proof, which could potentially be adapted to produce more regular $n$-fold symmetric solutions.
Smooth rotating solutions bifurcating from  vortex points can be found in \cite{delPino}.  
The construction of smooth traveling solutions has been carried out by Gravejat and Smets \cite{GS}, by Godard-Cadillac \cite{Godard}, as well as by Ao, D\'avila, del Pino, Musso, and Wei \cite{delPino}.


In \cite{CGSI19}, C\'ordoba, G\'omez-Serrano, and Ionescu constructed global solutions to the $\alpha$-SQG patch equation consisting of non-stationary solutions which converge to the stationary half-plane patch as $t \to \infty$.
In this line of thought, another approach to constructing families of global smooth solutions around steady states could involve the mechanism of inviscid damping.
This has been successfully applied to the 2D Euler equation for perturbations of the Couette flow $\bar{\Theta}=-1$ by Bedrossian and Masmoudi \cite{BedrossianMasmoudi15}, and by Ionescu and Jia \cite{IonescuJia20}, and later extended to general monotonic shear flows by Ionescu and Jia \cite{IonescuJia23} and by Masmoudi and Zhao \cite{MasmoudiZhao24}. 
Unfortunately, recent work by G\'omez-Serrano, Ionescu and Jia (see \cite[Section 3]{IonescuJiaEMS}) suggests that the approach of inviscid damping, with the aim of producing globally smooth solutions, fails for the Couette flow for any $\alpha > 0$, due to a forward cascade that leads to a loss of regularity in finite time. 
Therefore, for $\alpha > 0$, it remains unclear whether a smooth steady state $\bar{\Theta}$ to the $\alpha$-SQG equation can exist with a (small) open neighborhood of perturbations that converge to $\bar{\Theta}$ as $t\to\infty$, or even whether these perturbations produce global smooth solutions.

In this work, we show that by reversing the arrow of time in Theorem \ref{thm:unstablevortex}, there exist particular perturbations of the vortex $\bar{\Theta}$ that converge to it as $t \to \infty$. As a result, we establish the existence of global smooth solutions to the $\alpha$-SQG equation that are neither rotating nor traveling.

\begin{cor}[Global smooth solutions]\label{cor:global}
Let $0\leq\alpha\leq 1$, $n\geq 2$ and $m>\alpha+1$. There exists a global solution to the (unforced) $\alpha$-SQG equation
$$
\theta\in C([0,\infty),H^m\cap\dot{H}^{\frac{\alpha-2}{2}}),
$$
which is $n$-fold symmetric, non-stationary, and converges to the vortex $\bar{\Theta}\in C_c^\infty\cap\dot{H}^{\frac{\alpha-2}{2}}$ from Theorem \ref{thm:unstablevortex}
in $H^m$ as $t\to\infty$.
\end{cor}

We remark that our approach allows for the construction of global solutions in $H^m$ for any fixed $m>\alpha+1$. It would be interesting to explore whether this construction could be extended to the global $C^\infty$ case. 

In these results we have used 
$m$ (instead of $s$) to distinguish it from the supercritical case ($s<\alpha+1$) and to facilitate comparison with Section \ref{sec:non-linearinstability}, where everything is written in terms of the higher-order norms $H^m$. 
Although our non-linear analysis allows us to control any Sobolev norm, we chose to state the results in the Hilbert case for simplicity and because it is more commonly used in the $\alpha$-SQG setting. 

Recently, the existence of self-similar spirals has been established by Garc\'ia and G\'omez-Serrano for the $\alpha$-SQG equation, with $0<\alpha<1$, and by Shao, Wei, and Zhang for the 2D Euler equation \cite{SWZpp}, building on the previous work of Elling \cite{Elling13}. In these examples it holds that $\theta\in L^p$, leading to a promising scenario for non-uniqueness without forcing. Other promising scenarios based on symmetry breaking and self-similarity for the 2D Euler equation have been proposed recently by Bressan, Shen, and Murray \cite{BressanMurray20, BressanShen21}, as well as by Elgindi, Murray, and Said \cite{EMSpp}.

\subsection{Hamiltonian Identity}\label{sec:Hamiltonian}
The $\alpha$-SQG equation formally possesses a large class of conservation laws.  The Hamiltonian $\mathcal{H}=\mathcal{H}_\alpha$
\begin{equation}\label{eq:Hamiltonian}
\mathcal{H}(t)
=\frac{1}{2}\|\theta(t)\|_{\dot{H}^{\frac{\alpha-2}{2}}}^2
=\frac{1}{2}\|v(t)\|_{\dot{H}^{-\frac{\alpha}{2}}}^2,
\end{equation}
is particularly meaningful because the $\alpha$-SQG equation can be interpreted as geodesics with respect to the metric \eqref{eq:Hamiltonian} on the Lie group of volume-preserving diffeomorphisms (see e.g.~\cite{ArnoldKhesinbook,Washabaugh16,BSV19,MisiolekVu24,BHMP24}).
By a short energy estimate, it is straightforward to check that smooth solutions to the $\alpha$-SQG equation satisfy the 
\textit{Hamiltonian identity}
\begin{equation}\label{eq:Hidentity}
\frac{\mathrm{d}}{\mathrm{d}t}\mathcal{H}
=\langle f,\theta\rangle_{\dot{H}^{\frac{\alpha-2}{2}}}.
\end{equation}
In particular, $\mathcal{H}$ is conserved for classical solutions in the case $f=0$. 
Similarly to the question of uniqueness, the lack of Hamiltonian conservation is a cornerstone in the theory of turbulence (see e.g.~\cite{Frisch95}).
Analogously to \eqref{eq:sU}, a natural question is to determine
\begin{equation}\label{eq:sH}
s_H\equiv
\text{ threshold for Hamiltonian identity},    
\end{equation}
which again depends on $\alpha$ and the baseline space $X^s$. 

In terms of Hölder regularity, which is usually expressed as $\Lambda^{-1}\theta \in C^\gamma$ in the $\alpha$-SQG setting, Onsager conjectured in the celebrated work \cite{Onsager49} that $\gamma_H = \nicefrac{1}{3}$ for the Euler equation, which is dimensionless at the velocity level. The conservative part of Onsager's conjecture was proved by Constantin, E, and Titi \cite{CET94}, building on a partial result by Eyink \cite{Eyink94}. Onsager's conjecture can be extended to the $\alpha$-SQG equation as follows
\begin{equation}\label{eq:gammaH}
\gamma_H(\alpha)=\frac{1+2\alpha}{3}.
\end{equation}

As we mentioned earlier in the introduction, the non-conservative part of the Onsager-type conjecture \eqref{eq:gammaH} has been recently proven using convex integration for the 2D Euler equation \cite{GiriRadu24}, the SQG equation \cite{DGRpreprint, IsettLooiPreprint}, and the intermediate cases \cite{zhao2024}. The Onsager conjecture for the 3D Euler equation was previously solved by Isett \cite{Isett18}, by Buckmaster, De Lellis, Sz\'ekelyhidi, and Vicol \cite{BDSV19} for dissipative solutions, and by Giri, Kwon, and Novack \cite{GKN23} for locally dissipative solutions in Besov spaces. Additionally, see the work of Masmoudi, Novack, and Vicol \cite{BMNV23,NovackVicol23} on the Hilbert case.
We also refer to the works for Scheffer \cite{Scheffer93} and Shnirelman \cite{Shnirelman97} for the first constructions of non-unique Euler flows.
Notice that all these results prove non-uniqueness up to
$\gamma_H=\frac{1+2\alpha}{3}$, leaving a gap in the conjectured $\gamma_U=1+\alpha$ in the case without forcing.

In the context of Lebesgue spaces $\theta \in L^p$, Duchon and Robert were the first to prove that 2D Euler vorticities with $p > \nicefrac{3}{2}$ conserve the Hamiltonian \cite{DuchonRobert00}. This was extended to the borderline case $p = \nicefrac{3}{2}$ by Cheskidov, Lopes Filho, Nussenzveig Lopes, and Shvydkoy \cite{CLNS16}. These results lead to the natural conjecture for $p_H(\alpha)$ in the case of the 2D Euler equation ($\alpha=0$)
\begin{equation}\label{conjecture:pH0}
p_H(0)=\nicefrac{3}{2}.
\end{equation}

The first non-conservative results in weak Lebesgue spaces were obtained by Bru\`e and Colombo in Lorentz spaces \cite{BrueColombo23}, and by Modena and Buck in Hardy spaces \cite{BuckModena24}. Remarkably, Bruè, Colombo, and Kumar \cite{BCKpp2} recently constructed non-trivial 2D Euler vorticities $\theta \in L^{1 + \varepsilon}$. 
As already pointed out by the authors, this yields the first example of 2D Euler flows with uniformly integrable vorticity that are not vanishing viscosity solutions.
A natural question is whether this new convex integration scheme can be extended to reach \eqref{conjecture:pH0}.

As we discussed for the H\"older case, even if the Onsager exponent \eqref{conjecture:pH0} can be achieved with these techniques, there would still be a gap to the conjectured Yudovich exponent $p_U(0) = \infty$ in the case without forcing. In this regard, the third author constructed initial data $\theta^\circ \in L^p$, for any fixed $p < \infty$, admitting infinitely many solutions with non-increasing Hamiltonian \cite{Mengual24}, thus showcasing the sharpness of the weak-strong uniqueness principle and Yudovich's proof of uniqueness. Unfortunately, the vorticity information is lost for positive times in \cite{Mengual24}, as the weak solutions are only understood at the velocity level.

Concerning the SQG equation ($\alpha=1$) Isett and Vicol \cite{IsettVicol15} proved Hamiltonian conservation provided that $\theta\in L^3$. On the one hand, this implies that $p_H(1)\leq 3$. 
On the other hand, since $C^0\subset L^3$ in $\T^2$, this solves the conservative part of the Onsager-type conjecture \eqref{eq:gammaH} for $\alpha=1$. See also the recent work of Isett and Ma \cite{IsettMapp} on conservation laws for $\alpha$-SQG.

By the Sobolev embedding theorem, a natural extension of Conjecture \eqref{conjecture:pH0} for 2D Euler vorticities $\theta \in W^{s,p}$ would be $s_H(0, p) = \frac{2}{p} - \frac{4}{3}$. A similar exponent $s_H(\alpha,p)$ could be derived for the general case $0 \leq \alpha \leq 1$. The problem of constructing non-unique solutions with some Sobolev regularity remains widely open for the $\alpha$-SQG equation without forcing. 

The Hamiltonian structure of $\alpha$-SQG makes it a rare and significant example of active scalars. Notably, the method of convex integration has produced much more regular solutions when applied to other active scalars that lack a Hamiltonian structure. A key prototype is the IPM equation, known for its gradient flow structure \cite{Otto1999} and its role as a model to address fluid instabilities through convex integration. In this regard see e.g.~\cite{CFG11, Szekelyhidi12, CCF21,ForsterSzekelyhidi18,CFM19,  CFM22,CFGpp}. In fact, very general classes of active scalars, explicitly excluding SQG, can be treated using variants of convex integration, yielding bounded \cite{Shvydkoy11} and even continuous \cite{IsettVicol15,zhao2024} turbulent solutions.

Finally, we remark that vanishing viscosity solutions to the 2D Navier-Stokes equations with uniform integrability are known to conserve the Hamiltonian. This was proved in \cite{CLNS16} for $p > 1$, and recently extended to the borderline case $p = 1$ by De Rosa and Park.
A similar result was obtained by Constantin and Ignatova in $ L^2$ for the dissipative SQG equation in $\T^2$ \cite{CIN18}, and previously in $\R^2$ by Berselli \cite{Berselli02}.
This discrepancy in Hamiltonian conservation between ideal solutions and ideal limits has also been observed in other fluid models, such as 3D Magnetohydrodynamics (MHD) by the second author and Lindberg \cite{FaracoLindberg20}.

Therefore, the search for physically realizable solutions with uniformly integrable $\theta$ for the 2D Euler and SQG equations involves imposing that the Hamiltonian is conserved, or more generally the Hamiltonian identity in the presence of forcing. As a consequence of the regularity of the solutions in Theorem \ref{thm:main}, we deduce that our solutions satisfy this property. Thus, satisfying the Hamiltonian identity is not a criterion for uniqueness, at least in the case with forcing. 

\begin{cor}\label{cor:Hidentity}
In the setting of Theorem \ref{thm:main:0}, the solutions $\theta_\epsilon$ and the force $f$
satisfy the Hamiltonian identity \eqref{eq:Hidentity}.
\end{cor}

\subsection{Renormalization property}\label{sec:renormalization}
The Hamiltonian is not the only conserved quantity in $\alpha$-SQG. 
Since this is a transport equation, any functional of the temperature is formally preserved by the flow. Given a classical solution $\theta$ and a function $\beta$, the \textit{temperature Casimir}
\begin{equation}\label{eq:Casimirs}
C_\beta(t):=\int_{\R^2}\beta(\theta(t,x))\dif x,
\end{equation}
satisfies the identity
\begin{equation}\label{eq:Casimirsconserved}
\frac{\dif}{\dif t}
C_\beta
=\int_{\R^2}\beta'(\theta)f\dif x.
\end{equation}
In particular, $C_\beta$ is conserved in the case $f=0$. The Casimir temperature associated with the function $\beta=\frac{1}{q}|\cdot|^q$ for $1\leq q<\infty$ corresponds to the $q$-energy (or momentum)
$$
\mathcal{E}_q(t)
=\frac{1}{q}\|\theta(t)\|_{L^q}^q.
$$
In this regard, a solution to the $\alpha$-SQG equation $\theta$ satisfies the \textit{$q$-energy balance} if 
\begin{equation}\label{eq:energyconserved}
\frac{\dif}{\dif t}
\mathcal{E}_q
=\int_{\R^2}\theta|\theta|^{q-2}f\dif x.
\end{equation}
For the 2D Euler equation, $\mathcal{E}=\mathcal{E}_2$ represents the enstrophy, which plays a crucial role in 2D turbulence (see e.g.~\cite{Eyink01,LMN06,CDE22}).
Indeed, in turbulent regimes, it is plausible for the Hamiltonian $\mathcal{H}$ to be conserved, energy $\mathcal{E}_q$ to dissipate, and non-uniqueness to arise.
The existence of two quadratic invariants, both preserved by smooth solutions—one more robust than the other—forms the precise framework of the celebrated self-organization conjecture \cite{Hasegawa85}. Vorticity evolution in 2D Euler is a prime example of self-organization, and the $\alpha$-SQG equation aligns with this theory, at least at a formal level.

Recent studies have implemented convex integration schemes compatible with conservation of magnetic helicity in MHD \cite{BBV20, FLSz21, FLSz22, FLSz24}, which represents another canonical example of self-organization. Following this perspective, it is plausible that convex integration could be effective in regimes where the Hamiltonian is conserved for the $\alpha$-SQG equation but dissipation of $q$-energies is feasible.
Thus, whether convex integration can be applied in the Hamiltonian conservation regime, or whether it can be shown to be impossible, remains a compelling and challenging open problem.

The Casimirs balance can be described locally through the concept of renormalized solutions. 
A solution to the $\alpha$-SQG equation $\theta$ satisfies the \textit{renormalization property} if
\begin{equation}\label{eq:renormalized}
\partial_t\beta(\theta)
+v\cdot\nabla(\beta(\theta))
=\beta'(\theta)f
\end{equation}
holds in the sense of distributions for every $\beta\in C_c^1(\R)$, that is, 
$$
\int_0^T\int_{\R^2}(\beta(\theta)(\partial_t\varphi+v\cdot\nabla\varphi)+\beta'(\theta)f\varphi)\dif x\dif t
=-\int_{\R^2}\beta(\theta^\circ(x))\varphi(0,x)\dif x,
$$
for all test function $\varphi\in C_c^1([0,T)\times\R^2)$.
Notice that the concept of weak (or distributional) solution corresponds to $\beta\to\mathrm{id}$. By integrating \eqref{eq:renormalized} and applying that $v$ is weakly divergence-free, we recover the Casimir identity \eqref{eq:Casimirsconserved}. Formally, ignoring the singularity at zero and its growth at infinity, we can also recover the $q$-energy balance \eqref{eq:energyconserved}. 

Analogously to \eqref{eq:sH}, a natural question is to determine the threshold for the renormalization property, as well as the $q$-energy balance. 
For the 2D Euler equation, Crippa and Spirito shown that vanishing viscosity limits are renormalized \cite{CrippaSpirito15} provided that $\theta\in L^p$ with $p>1$. In fact, by the DiPerna-Lions theory \cite{DiPernaLions89} any weak solution with vorticity in $L^2$ is renormalized. However, it might be that large moments are dissipated in the vanishing viscosity limit \cite{CDE22}. 

As a consequence of the regularity of the solutions of Theorem~\ref{thm:main}, they can be shown to be renormalized, and thus being renormalized is not a criteria for uniqueness, at least in the case with forcing. Moreover, they satisfy the $q$-energy balance, provided that $\mathcal{E}_q$ is well-defined. 

\begin{cor}\label{cor:renormalized}
In the setting of Theorem \ref{thm:main:0}, the solutions $\theta_\epsilon$ and the force $f$ satisfy the renormalization property, and also the $q$-energy balance for $q$ moments as long as they are well defined.
\end{cor} 

As we already mentioned, both Corollaries \ref{cor:Hidentity} and \ref{cor:renormalized} follow from the regularity obtained in Theorem \ref{thm:main}, by taking $\varepsilon > 0$ small enough if necessary, along with the conservative results discussed in Sections \ref{sec:Hamiltonian} and \ref{sec:renormalization}.
Alternatively, they can be deduced more directly as follows. On the one hand, the force and the solutions are classical for positive times, so they trivially satisfy all the conservation laws for $t>0$. On the other hand, the solutions $\theta_\epsilon$ converge to zero as $t \to 0$ in the corresponding Sobolev spaces (Proposition \ref{prop:scalingHs}), which is sufficient to verify that the conservation laws are also satisfied at $t=0$.

Finally, we remark that most of the literature mentioned in this section refers to the periodic box $\T^2$, while our theorems are stated in $\R^2$. Although our solutions can be shown to have compact support, the corresponding velocity does not necessarily. A natural problem is to extend our construction to the periodic box, as well as to bounded domains.

\subsection{Organization of the paper}
We start Section \ref{sec:preliminaries} by deriving the stability equation in self-similar coordinates. Next, we outline the proof of the main theorems in Section \ref{sec:sketch}. In Section \ref{sec:piecewise} we construct the piecewise constant unstable vortex, and we find a regularization in Section \ref{sec:regularization}. 
Finally, in Sections \ref{sec:selfsimilarinstability} and \ref{sec:non-linearinstability} we prove that the vortex is also self-similarly and non-linearly
unstable, respectively.

\section{Stability equation in self-similar coordinates}\label{sec:preliminaries}

In this section, we provide some preliminary definitions and results that will be used throughout the paper. Firstly, we rewrite the $\alpha$-SQG equation in self-similar coordinates. Secondly, we derive both the linear and non-linear stability equations around vortices. 

\subsection{Self-similar coordinates}\label{sec:selfsimilarcoordinates}
In this section we write the $\alpha$-SQG equation \eqref{eq:SQG} in \textit{self-similar coordinates}
\begin{equation}\label{eq:selfsimilarcoordinates}
\tau=\frac{1}{ab}\log t,
\quad\quad
X=\frac{x}{(abt)^{\nicefrac{1}{a}}},
\end{equation}
in terms of two parameters $0<a,b\leq 1$, to be determined.
Notice that the logarithmic time interval is $\R$ (instead of $[0,\infty)$).
In particular, the physical initial time $t=0$ corresponds to
the logarithmic time $\tau=-\infty$.
In the next proposition we derive the \textit{self-similar $\alpha$-SQG equation}.

\begin{prop}
The pair $(\theta,f)$ given by the change of variables
\begin{subequations}\label{eq:oOfF}
\begin{align}
\theta(t,x)&=(abt)^{\frac{\alpha}{a}-1}\Theta(\tau,X),\label{eq:oO}\\
f(t,x)&=(abt)^{\frac{\alpha}{a}-2}F(\tau,X),\label{eq:fF}
\end{align}
\end{subequations}
is a solution to the $\alpha$-SQG equation if and only if $(\Theta,F)$ solves the self-similar $\alpha$-SQG equation
\begin{equation}\label{eq:SQG:SS}
\partial_\tau\Theta
+V\cdot\nabla\Theta
-b\left((a-\alpha)+X\cdot\nabla \right)\Theta
=F.
\end{equation}
The corresponding velocities 
are linked by
\begin{equation}\label{eq:vV}
v(t,x)=(abt)^{\frac{1}{a}-1}V(\tau,X),
\end{equation}
that is,
$(v,V)$ are recovered from $(\theta,\Theta)$ through the $\alpha$-Biot-Savart law \eqref{eq:BiotSavart}, respectively.
\end{prop}

\begin{proof}
Firstly, we deduce the relation \eqref{eq:vV}
\begin{align*}
v(t,x)
&=C_\alpha\int_{\R^2}\frac{(x-y)^\perp}{|x-y|^{2+\alpha}}\theta(t,y)\dif y
=(abt)^{\frac{\alpha}{a}-1}C_\alpha
\int_{\R^2}\frac{(x-y)^\perp}{|x-y|^{2+\alpha}}\Theta(\tau,Y)\dif y\\
&=(abt)^{\frac{1}{a}-1}C_\alpha
\int_{\R^2}\frac{(X-Y)^\perp}{|X-Y|^{2+\alpha}}\Theta(\tau,Y)\dif Y
=(abt)^{\frac{1}{a}-1}
V(\tau,X),
\end{align*}
where we made the change of variables $y=(abt)^{\nicefrac{1}{a}}Y$ with $\dif y=(abt)^{\nicefrac{2}{a}}\dif Y$. On the one hand,
$$
\partial_t\theta
=(abt)^{\frac{\alpha}{a}-2}
\left(\partial_\tau\Theta
+b(\alpha-a)\Theta-bX\cdot\nabla\Theta\right).
$$
On the other hand,
$$
v\cdot\nabla\theta=
(abt)^{\frac{\alpha}{a}-2}V\cdot\nabla\Theta.
$$
The last two identities imply the equivalence between \eqref{eq:SQG} and \eqref{eq:SQG:SS} through the self-similar change of variables \eqref{eq:oOfF}.
\end{proof}

\begin{Rem}
Following the notation in \cite{ABC22,CFMSpp}, given an object related to the $\alpha$-SQG equation, the corresponding object in self-similar coordinates will be denoted by the same letter in uppercase. The only exception is the vortex \eqref{eq:oOfF:0}. 
Nonetheless, 
notice that the case $b=0$ in \eqref{eq:SQG:SS}  agrees with $\alpha$-SQG in the original coordinates, while $b>0$ represents the equation in self-similar coordinates.
\end{Rem}

In the next proposition we show how the $W^{s,p}$ norm behaves under the self-similar change of variables \eqref{eq:oOfF}. We recall that, for any $s\in\R$ and $1\leq p\leq\infty$, the fractional Sobolev space $W^{s,p}=W^{s,p}(\R^2)$ is defined by the norm
$$
\|f\|_{W^{s,p}}
=
\left\|\left[(1+|k|^2)^{\frac{s}{2}}\hat{f}\right]^{\check{\:}}\right\|_{L^p},
$$
where $\hat{f}$ denotes the Fourier transform
$$
\hat{f}(k)=\int_{\R^2}f(x)e^{-ix\cdot k}\dif x,
$$
and $\check{f}$ denotes the inverse Fourier transform.
We also consider the seminorm
$$
\|f\|_{\dot{W}^{s,p}}
=
\left\|\left[|k|^{s}\hat{f}\right]^{\check{\:}}\right\|_{L^p}.
$$
In the Hilbert case ($p=2$) we will denote $H^s=W^{s,2}$ and $\dot{H}^s=\dot{W}^{s,2}$ as usual. In general, we will consider $s\geq 0$, except when dealing with the Hamiltonian.


\begin{prop}\label{prop:scalingHs}
Let $(\theta,f)$ and $(\Theta,F)$ be linked by \eqref{eq:oOfF}.
Then,  
\begin{subequations}
\begin{align}
\|\theta(t)\|_{\dot{W}^{s,p}}
&=(abt)^{\frac{\alpha+\frac{2}{p}-s}{a}-1}
\|\Theta(\tau)\|_{\dot{W}^{s,p}},\label{eq:scalingHs:theta}\\
\|f(t)\|_{\dot{W}^{s,p}}
&=(abt)^{\frac{\alpha+\frac{2}{p}-s}{a}-2}
\|F(\tau)\|_{\dot{W}^{s,p}}.\label{eq:scalingHs:theta}
\end{align}
\end{subequations}
\end{prop}
\begin{proof}
Firstly, we write $\theta$ in the Fourier side
\begin{align*}
\hat{\theta}(t,k)
&=\int_{\R^2}\theta(t,x)e^{-ix\cdot k}\dif x
=(abt)^{\frac{\alpha}{a}-1}\int_{\R^2}\Theta(\tau,X)e^{-ix\cdot k}\dif x\\
&=(abt)^{\frac{2+\alpha}{a}-1}\int_{\R^2}\Theta(\tau,X)e^{-iX\cdot K}\dif X
=(abt)^{\frac{2+\alpha}{a}-1}\hat{\Theta}(\tau,K),
\end{align*}
where 
$
K=(abt)^{1/a}k.
$
Thus,
$$
g(t,k)
=(abt)^{\frac{2+\alpha-s}{a}-1}G(\tau,K),
$$
where $g=|k|^s\hat{\theta}$ and $G=|K|^s\hat{\Theta}$. Analogously,  we compute
\begin{align*}
\check{g}(t,x)
&=\int_{\R^2}g(t,k)e^{ik\cdot x}\dif k
=(abt)^{\frac{2+\alpha-s}{a}-1}\int_{\R^2}G(\tau,K)e^{ik\cdot x}\dif k\\
&=(abt)^{\frac{\alpha-s}{a}-1}\int_{\R^2}G(\tau,K)e^{iK\cdot X}\dif K
=(abt)^{\frac{\alpha-s}{a}-1}\check{G}(\tau,X).
\end{align*}
Therefore,
\begin{align*}
\|\theta(t)\|_{\dot{W}^{s,p}}^p
&=\int_{\R^2}
|\check{g}(t,x)|^p\dif x
=(abt)^{p\left(\frac{\alpha-s}{a}-1\right)}
\int_{\R^2}
|\check{G}(\tau,X)|^p\dif x\\
&=(abt)^{p\left(\frac{\alpha+\frac{2}{p}-s}{a}-1\right)}
\int_{\R^2}
|\check{G}(\tau,X)|^p\dif X
=(abt)^{p\left(\frac{\alpha+\frac{2}{p}-s}{a}-1\right)}
\|\Theta(\tau)\|_{\dot{W}^{s,p}}^p.
\end{align*}
The proof for $f$ is analogous.
\end{proof}

\begin{defi}\label{defi:selfsimilar}
We say that $(\theta_0,f)$ is a \textit{self-similar solution} to the $\alpha$-SQG equation if the corresponding pair $(\bar{\Theta},F)$ given by the self-similar change of variables \eqref{eq:oOfF}
is a stationary solution of the self-similar $\alpha$-SQG equation, that is, 
\begin{equation}\label{ansatz:F}
\bar{V}\cdot\nabla\bar{\Theta}
-b\left((a-\alpha)+X\cdot\nabla \right)\bar{\Theta}
=F,
\end{equation}
where $\bar{V}$ is recovered from $\bar{\Theta}$ through the $\alpha$-Biot-Savart law \eqref{eq:BiotSavart}.
\end{defi}

\begin{Rem}
Observe that any $\bar{\Theta}$ becomes a solution by defining the force \textit{ad hoc} as in equation \eqref{ansatz:F}. This degree of freedom is the main advantage of considering the forced equation, enabling the consideration of functions for which the stability analysis becomes feasible, although it still requires considerable effort.
\end{Rem}

In view of Proposition \ref{prop:scalingHs},
the integrability conditions \eqref{forceintegrability} and \eqref{thetaintegrability}
for self-similar solutions require
to take $a$ in the regime
$$
0<a<\alpha+\frac{2}{p}-s,
$$
which is possible if and only if $s$ satisfies
\begin{equation}\label{eq:s}
s<\alpha+\frac{2}{p},
\end{equation}
as stated in Theorem \ref{thm:main:0}. 
From now on, we consider these parameters to be fixed in this regime
\begin{equation}\label{eq:a}
0<a<\varepsilon\leq\alpha+\frac{2}{p}-s,
\end{equation}
where $0<\varepsilon<1+\alpha$ is the parameter in Theorem \ref{thm:main}.

\begin{cor}\label{cor:selfsimilarscaling}
Let $(\theta_0,f)$ be a self-similar solution with $(\bar{\Theta},F)$ smooth and compactly supported. It holds that
\begin{subequations}
\begin{align}
\|\theta_0(t)\|_{\dot{W}^{s,p}}
&=(abt)^{\frac{\alpha+\frac{2}{p}-s}{a}-1}
\|\bar{\Theta}\|_{\dot{W}^{s,p}},\label{eq:scalingHs:theta0}\\
\int_0^t \|f(t')\|_{\dot{W}^{s,p}}\dif t'
&=\frac{(abt)^{\frac{\alpha+\frac{2}{p}-s}{a}-1}}{ab\left(\frac{\alpha+\frac{2}{p}-s}{a}-1\right)}
\|F\|_{\dot{W}^{s,p}}\label{eq:scalingHs:f0}.
\end{align}
\end{subequations}
In particular, $\theta_0|_{t=0}=0$.
\end{cor}
\begin{proof}
It follows by applying Definition \ref{defi:selfsimilar} and Proposition \ref{prop:scalingHs}.
\end{proof}

\subsection{Linear stability equation}\label{sec:stabilityequation}
In this section we derive the linear stability equation around a (self-similar) solution. 
Given a background steady temperature $\bar{\Theta}$, a perturbation
$$
\Theta_\epsilon
=\bar{\Theta}+\epsilon\tilde{\Theta},
$$
is a second solution to the (self-similar) $\alpha$-SQG equation \eqref{eq:SQG:SS} with the same initial data and force if and only if the deviation $\tilde{\Theta}$ solves
\begin{equation}\label{eq:SQG:lin}
(\partial_\tau-L_b)\tilde{\Theta}
+\epsilon (\tilde{V}\cdot\nabla\tilde{\Theta})=0,
\end{equation}
coupled with the initial condition
\begin{equation}\label{eq:devinitial}
\tilde{\Theta}|_{\tau=-\infty}=0.
\end{equation}
The operator $L_b$ is the linearization of the (self-similar) $\alpha$-SQG equation around $\bar{\Theta}$. This can be decomposed into
\begin{equation}\label{Lb}
L_b
=b(a-\alpha)
+T_b+K,
\end{equation}
where 
\begin{align*}
T_b\Theta&=(bX-\bar{V})\cdot\nabla\Theta,\\
K\Theta&=-V\cdot\nabla\bar{\Theta}.
\end{align*}
The velocities $(V,\bar{V})$ are recovered from $(\Theta,\bar{\Theta})$ respectively through the $\alpha$-Biot-Savart law \eqref{eq:BiotSavart}. 

\begin{Rem}
Notice that $L_b$, $T_b$ and $K$ depend on $a$, $\alpha$ and $\bar{\Theta}$. Since the parameter $a$ has already been fixed in \eqref{eq:a}, we omit this dependence. The dependence on $\alpha$ and $\bar{\Theta}$ will be made explicit in the notation when necessary.
\end{Rem}

As usual in Stability theory, as $\epsilon\to 0$ one is lead to study the linear equation
\begin{equation}\label{eq:SQG:L}
(\partial_\tau-L_b)\Theta^{\text{lin}}=0,
\end{equation}
and seek for solutions that grow exponentially in time, that is,
\begin{equation}\label{omegalin}
\Theta^{\text{lin}}(\tau,X)=\Re(e^{\lambda\tau}W(X)),
\end{equation}
with $\Re\lambda >0$.
For solutions of this form, equation \eqref{eq:SQG:L} is equivalent to the eigenvalue problem
\begin{equation}\label{eq:L}
L_bW=\lambda W.
\end{equation}
In other words, linear instability translates into the existence of eigenvalues $\lambda$ in the right half plane $\C_+$,
and eigenfunctions $W$ in some suitable Hilbert space, of the linear operator $L_b=L_{b,\alpha,\bar{\Theta}}$.

An obvious but crucial observation is that
$$
e^{\lambda\tau}
=t^{\frac{\lambda}{ab}}
\to 0,
$$
as $\tau\to -\infty$ (or equivalently $t\to 0$), and therefore
\begin{equation}\label{eq:Thetalininitial}
\Theta^{\text{lin}}|_{\tau=-\infty}=0,
\end{equation}
which leads to non-uniqueness at the linear level.

\subsubsection{Vortices}
We consider the special case of radially symmetric temperatures, called \textit{vortices}. In polar coordinates $X=Re^{i\phi}$, this corresponds to
\begin{equation}\label{eq:radiallysymmetric}
\bar{\Theta}(X)
=\bar{\Theta}(R).
\end{equation}
In Lemma \ref{lemma:BiotSavartbarV}, we show that the corresponding velocity $\bar{V}$, recovered from $\bar{\Theta}$ through the $\alpha$-Biot-Savart law \eqref{eq:BiotSavart}, has only an angular component $\bar{V}_\phi$. Moreover, we show that $\bar{V}_\phi$ is expressed in terms of the operator $V_{1,\alpha}$, defined below. We take the opportunity to introduce this operator more generally as $V_{n,\alpha}$, since it will appear in the eigenspace $U_n$ in the following subsection.

Given $n\in\Z$, we define the operator
\begin{equation}\label{eq:Vnalpha}
V_{n,\alpha}[f](R)
:=C_\alpha\int_0^\infty
I_{n,\alpha}\left(\frac{R}{S}\right)f(S)S^{1-\alpha}\dif S,
\end{equation}
where $C_\alpha$ is the constant in \eqref{eq:BiotSavart}, and $I_{n,\alpha}$ is the kernel
\begin{equation}\label{eq:In:intro}
I_{n,\alpha}(\sigma)
:=\frac{1}{\alpha}\int_{-\pi}^{\pi}\frac{\cos(n\beta)}{|\sigma-e^{i\beta}|^{\alpha}}\dif\beta,
\end{equation}
whenever these expressions make sense.

\begin{Rem}\label{Rem:Iforalpha=0}
A priori, \eqref{eq:In:intro} seems ill defined for $\alpha=0$. However, an integration by parts 
(see \eqref{eq:Kn}) shows that 
$I_{n,\alpha}$ can be expressed as
$$
I_{n,\alpha}(\sigma)
=\frac{\sigma}{n}K_{n,\alpha}(\sigma),
$$
where
$$
K_{n,\alpha}(\sigma)
:=\int_{-\pi}^{\pi}\frac{\sin(\beta)\sin(n\beta)}{|\sigma-e^{i\beta}|^{2+\alpha}}\dif\beta,
$$
which is well defined for $\alpha=0$ as well. Since this case corresponds to the 2D Euler equation, which has already been analyzed in \cite{CFMSpp}, we mostly deal with \eqref{eq:In:intro} for the sake of convenience. However, at certain points, it will be more convenient to use $K_{n,\alpha}$. 
\end{Rem}

\begin{lemma}\label{lemma:BiotSavartbarV}
It holds that
$$
\bar{V}(X)
=\bar{V}_\phi(R)e_\phi,
\quad
\text{with}
\quad
\bar{V}_\phi
=-
V_{1,\alpha}[\partial_R\bar{\Theta}],
$$
where $V_{1,\alpha}$ is defined in \eqref{eq:Vnalpha}.
\end{lemma}
\begin{proof}
By applying
$$
\nabla\bar{\Theta}(X)
=\partial_R\bar{\Theta}(R)e^{i\phi},
$$
we compute ($\beta=\phi-\varphi$)
\begin{align*}
\bar{V}(X)
=-\frac{C_\alpha}{\alpha}\int_{\R^2}\frac{\nabla^\perp\bar{\Theta}(Y)}{|X-Y|^{\alpha}}\dif Y
&=-\frac{C_\alpha}{\alpha}\int_{-\pi}^{\pi}\int_0^\infty
\frac{ie^{i\varphi}\partial_R\bar{\Theta}(S)}{|Re^{i\phi}-Se^{i\varphi}|^\alpha}S\dif S\dif\varphi \\
&=-ie^{i\phi}\frac{C_\alpha}{\alpha}\int_0^\infty\int_{-\pi}^{\pi}
\frac{e^{-i\beta}}{|R-Se^{i\beta}|^\alpha}\dif\beta\partial_R\bar{\Theta}(S) S\dif S\\
&=-ie^{i\phi}\frac{C_\alpha}{\alpha}\int_0^\infty\int_{-\pi}^{\pi}
\frac{\cos\beta}{|R-Se^{i\beta}|^\alpha}\dif\beta\partial_R\bar{\Theta}(s) S\dif S.
\end{align*}
This concludes the proof.
\end{proof}

\begin{cor}\label{cor:TKvortices}
It holds that
\begin{align*}
T_bW&=\left(bR\partial_R-\frac{\bar{V}_\phi}{R}\partial_\phi\right)W,\\
KW&=-V_R\partial_R\bar{\Theta}.
\end{align*}
\end{cor}
\begin{proof}
On the one hand,
$$
X\cdot\nabla W=R\partial_R W.
$$
On the other hand,
$$
\bar{V}\cdot\nabla W=\frac{\bar{V}_\phi}{R}\partial_\phi W,
\quad\quad
V\cdot\nabla\bar{\Theta}
=V_R\partial_R\bar{\Theta},
$$
where we applied \eqref{eq:radiallysymmetric} and Lemma \ref{lemma:BiotSavartbarV}.
\end{proof}

\subsubsection{Purely $n$-fold symmetric temperatures}
In order to ensure that $\bar{\Theta}$ has finite Hamiltonian for $\alpha=0$, we will work in the space $U_0$ of vortices in $L^2(\R^2)$ with zero-mean,
\begin{equation}\label{eq:U0}
U_0
:=\left\lbrace
\bar{\Theta}\in L^2
\,:\,
\bar{\Theta}(X)=\bar{\Theta}(R)\, ,\
\int_0^\infty\bar{\Theta}(R)R\dif R=0\right\rbrace.
\end{equation}
Although the zero-mean condition is not strictly necessary for $\alpha>0$ for having finite Hamiltonian, it appears to be more convenient for producing unstable vortices (see Remark \ref{rem:zeromean}).
In this context, given $0\neq n\in\Z$, it is natural to seek eigenfunctions in the space of \textit{purely $n$-fold symmetric} temperatures, 
\begin{equation}\label{eq:Un}
U_n:=\{W\in L^2\,:\,W(X)=W_n(R)e^{in\phi}\}.
\end{equation}
Notice that any element of $U_n$ has zero-mean.
Since $U_{-n}=U_n^*$, we will consider without loss of generality the case $n\in\N$.

\begin{lemma}\label{lemma:vr:eigenfunction}
For every $W\in U_n$,
$$
V_R(X)
=in\frac{V_{n,\alpha}[W_n](R)}{R}e^{in\phi},
$$
where $V_{n,\alpha}$ is defined in \eqref{eq:Vnalpha}.
\end{lemma}
\begin{proof}
By writing the $\alpha$-Biot-Savart law \eqref{eq:BiotSavart} in polar coordinates ($\beta=\phi-\varphi$)
\begin{align*}
V(X)
&=C_\alpha\int_0^\infty\int_{-\pi}^{\pi}
\frac{(Re^{i\phi}-Se^{i\varphi})i}{|Re^{i\phi}-Se^{i\varphi}|^{2+\alpha}}W(Se^{i\varphi})
\dif\varphi S\dif S\\
&=ie^{i\phi}
C_\alpha\int_0^\infty\int_{-\pi}^{\pi}
\frac{R-Se^{-i\beta}}{|R-Se^{i\beta}|^{2+\alpha}}W(Se^{i(\phi-\beta)})
\dif\beta S\dif S,
\end{align*}
we get
\begin{align*}
V_R
=\Re(Ve^{-i\phi})
&=-C_\alpha\Im\int_0^\infty\int_{-\pi}^{\pi}
\frac{R-Se^{-i\beta}}{|R-Se^{i\beta}|^{2+\alpha}}W(Se^{i(\phi-\beta)})
\dif\beta S\dif S\\
&=-C_\alpha\int_0^\infty\int_{-\pi}^{\pi}
\frac{\sin\beta}{|R-Se^{i\beta}|^{2+\alpha}}W(Se^{i(\phi-\beta)})\dif\beta S^2\dif S.
\end{align*}
We remark that this is the expression for the real operator $V_R$ (accordingly for $KW=-V_R\partial_R\bar{\Theta}$). Next, we consider $V_R$ acting on complex-valued $W$'s.
Hence, for $W(X)=W_n(R)e^{in\phi}$, we have
$$
V_R
=-e^{in\phi}C_\alpha\int_0^\infty \int_{-\pi}^{\pi}
\frac{\sin\beta e^{-in\beta}}{|R-Se^{i\beta}|^{2+\alpha}}\dif\beta W_n(S) S^2\dif S
=ie^{in\phi}C_\alpha\int_0^\infty
K_{n,\alpha}\left(\frac{R}{S}\right)
W_n(S)S^{-\alpha}\dif S,
$$
where
\begin{equation}\label{eq:Kn}
\begin{split}
K_{n,\alpha}(\sigma)
=\int_{-\pi}^{\pi}\frac{\sin\beta\sin(n\beta)}{|\sigma-e^{i\beta}|^{2+\alpha}}\dif\beta
&=-\frac{1}{\alpha\sigma}
\int_{-\pi}^{\pi}\partial_\beta\left(\frac{1}{|\sigma-e^{i\beta}|^{\alpha}}\right)\sin(n\beta)\dif\beta\\
&=\frac{n}{\alpha\sigma}
\int_{-\pi}^{\pi}\frac{\cos(n\beta)}{|\sigma-e^{i\beta}|^{\alpha}}\dif\beta
=\frac{n}{\sigma}I_{n,\alpha}(\sigma).
\end{split}
\end{equation}
This concludes the proof.
\end{proof}

\begin{cor}
For every $W\in U_n$, it holds that
\begin{align*}
T_bW
&=\left(bR\partial_R-in\frac{\bar{V}_\phi}{R}\right)W_n e^{in\phi},\\
KW
&=-in\frac{V_{n,\alpha}[W_n]}{R}\partial_R\bar{\Theta}e^{in\phi}.
\end{align*}
\end{cor}
\begin{proof}
It follows from Corollary \ref{cor:TKvortices}, definition \eqref{eq:Un} and Lemma \ref{lemma:vr:eigenfunction}.
\end{proof}

Therefore, the eigenvalue problem $L_bW=\lambda W$ for $W\in U_n$ can be rewritten as
\begin{equation}\label{eq:SQG:selfsimilarstability}
b(a-\alpha)W_n
+\left(bR\partial_R-in\frac{\bar{V}_\phi}{R}\right)W_n
-in\partial_R\bar{\Theta}\frac{V_{n,\alpha}[W_n]}{R}
=\lambda W_n,
\end{equation}
which we will refer to as the \textit{linear stability equation}. We recall that the case $b=0$ corresponds to the original coordinates, while $b>0$ corresponds to the self-similar coordinates.

\begin{Rem}
For $b=0$, equation \eqref{eq:SQG:selfsimilarstability} corresponds to the stability equation in the original coordinates. In particular, for $\alpha=0$, this represents the Rayleigh stability equation for the 2D Euler equation at the vorticity level.
\end{Rem}

\begin{defi}\label{def:selfsimilarunstable}
We say that the vortex $\bar{\Theta}$ is \textit{unstable} if, for some $n\in\N$, there exists $0\neq W\in U_n$ satisfying $L_{0,\bar{\Theta}}W=\lambda W$ with $\Re\lambda>0$. Similarly, we say that $\bar{\Theta}$ is \textit{self-similarly unstable} if, for some $n\in\N$ and $b>0$, there exists $0\neq W\in U_n$ satisfying $L_{b,\bar{\Theta}}W=\lambda W$ with $\Re\lambda>0$.
\end{defi}

\subsection{Non-linear stability equation}

Let us assume that such a (self-similarly) unstable vortex $\bar{\Theta}$ exists. We decompose the deviation into
$$
\tilde{\Theta}=\Theta^{\text{lin}}
+\epsilon\Theta^{\text{cor}},
$$
where 
$\Theta^{\text{lin}}(\tau,X)=\Re(e^{\lambda\tau}W(X))$, and
$\Theta^{\text{cor}}$ is the non-linear correction, to be determined. Namely, $\Theta^{\text{cor}}$ 
must satisfy the equation
\begin{equation}\label{eq:SQG:cor}
(\partial_\tau - L_b)\Theta^{\text{cor}}
+\underbrace{(V^{\text{lin}}+\epsilon V^{\text{cor}})\cdot\nabla(\Theta^{\text{lin}}+\epsilon\Theta^{\text{cor}})}_{\mathcal{F}}
=0,
\end{equation}
which we will refer to as the \textit{non-linear stability equation}, 
coupled with the initial condition
\begin{equation}\label{eq:Thetacorinitial}
\Theta^{\text{cor}}|_{\tau=-\infty}=0.
\end{equation}
If we interpret $\mathcal{F}$ as a forcing term, since the linear part decays as $e^{\Re\lambda\tau}$ and the contribution of the quadratic term is expected to be negligible for short times, one can naively expect to
gain a slightly faster exponential decay 
by exploiting the Duhamel formula. 
We recall that the case $b=0$ corresponds to the original coordinates, while $b>0$ corresponds to the self-similar coordinates.

\begin{defi}\label{def:non-linearlyunstable}
We say that the (self-similarly) unstable vortex $\bar{\Theta}$ is also \textit{(self-similarly) non-linearly unstable} if there exists $\Theta^{\text{cor}}$ solving \eqref{eq:SQG:cor} and \eqref{eq:Thetacorinitial} with
\begin{equation}\label{eq:expdecay}
\|\Theta^{\text{cor}}(\tau)\|_{L^2}
=o(e^{\Re\lambda\tau}),
\end{equation}
as $\tau\to -\infty$.
\end{defi}

\subsubsection{Space of $n$-fold symmetric temperatures}\label{sec:nfold}

In the previous sections we considered the linearization $L_b$ acting on the invariant subspace $U_n$ of purely $n$-fold symmetric temperatures. However, for the non-linear instability
we will need to take into account
the quadratic term, which introduces multiples of the original frequency $n$. For this reason, it is convenient to analyze $L_b$ within the larger space $L_n^2$ of temperatures $\Theta\in L^2$ which have zero mean and are $n$-fold symmetric 
\begin{equation}\label{eq:nfold}
\Theta(X)=\Theta(e^{\frac{2\pi i}{n}}X).
\end{equation} 
By writing the Fourier expansion of $\Theta$ in polar coordinates $X=Re^{i\theta}$,
$$
\Theta(X)=\sum_{k\in\Z}\Theta_k(R)e^{ik\theta},
\quad\quad
\Theta_k(R)
=\frac{1}{2\pi}\int_0^{2\pi}\Theta(Re^{i\theta})e^{-ik\theta}\dif\theta,
$$
it follows that
the $n$-fold symmetry \eqref{eq:nfold} is equivalent to the vanishing of the indices $k$ that are not multiples of $n$, and the zero mean condition for $k=0$
$$
\int_0^\infty\Theta_0(R)R\dif R=0.
$$
Therefore, we can decompose $L_n^2$ into the orthogonal direct sum
\begin{equation}\label{eq:directsum}
L_n^2=\bigoplus_{j\in\Z}U_{jn}
\end{equation}
of the invariant subspaces $U_0$ and $U_{jn}$ given in \eqref{eq:U0} and \eqref{eq:Un}, respectively.
More precisely, by the Plancherel identity
\begin{equation}\label{eq:L2n}
\|\Theta\|_{L^2}^2
=2\pi
\sum_{j\in\Z}
\int_0^\infty |\Theta_{jn}(R)|^2R\dif R,
\end{equation}
we consider sums of elements $\Theta_{jn}$ for which \eqref{eq:L2n} is finite.

\section{Sketch of the proof}\label{sec:sketch}

In this section, we prove Theorem \ref{thm:unstablevortex}, as well as the following stronger version of Theorem \ref{thm:main:0}.

\begin{thm}[Refined Sobolev non-uniqueness]\label{thm:main}
Let $0\leq\alpha\leq 1$ and $0<\varepsilon<1+\alpha$. 
There exist $T>0$ and a force
\begin{equation}\label{forceintegrability:improved}
f\in L^1([0,T],W^{s,p}\cap\dot{H}^{\frac{\alpha-2}{2}}),
\end{equation}
such that there are uncountably many solutions 
\begin{equation}\label{thetaintegrability:improved}
\theta_\epsilon\in L^\infty([0,T],W^{s,p}\cap\dot{H}^{\frac{\alpha-2}{2}}),
\end{equation}
to the $\alpha$-SQG equation starting from  $\theta^\circ=0$, for all $s\geq 0$ and $1\leq p\leq\infty$ satisfying $s\leq\alpha+\frac{2}{p}-\varepsilon$. 
\end{thm}

\begin{Rem}
The improvement is that, in Theorem \ref{thm:main}, the force and the solutions belong simultaneously to all the Sobolev spaces corresponding to the red region in Figure \ref{fig:figura_completa}.
The condition $\varepsilon < 1 + \alpha$ will appear frequently in the proofs. 
Moreover, it guarantees that $\theta \in L^2$, which is sufficient to properly define the weak formulation at the level of $\theta$. 
\end{Rem}

We take the opportunity to recall here Vishik's result to facilitate comparison. Notice that, for $\alpha=0$, the solutions can be extended globally in time and thus \eqref{forceintegrability:improved} and \eqref{thetaintegrability:improved} hold for all $T<\infty$. 

\begin{thm}[Vishik's non-uniqueness theorem]\label{thm:Vishik}
Let $2<p<\infty$.
There exists a force
\begin{equation}\label{force:Vishik}
f\in L^1([0,\infty),\dot{H}^{-1}\cap L^1\cap L^p),
\end{equation}
such that there are uncountably many solutions 
\begin{equation}\label{theta:Vishik}
\theta_\epsilon\in L^\infty([0,\infty),\dot{H}^{-1}\cap L^1\cap L^p),
\end{equation}
to the 2D Euler equation starting from $\theta^\circ=0$.
\end{thm}

We begin with a general overview of the strategy.
In short, the key observation is that a vortex $\bar{\Theta}$ that is self-similarly non-linearly unstable implies non-uniqueness. We first explain this idea succinctly and then outline the steps for constructing the vortex.

Firstly, we take the main solution $\theta_0$ as the self-similar vortex corresponding to $\bar{\Theta}$. According to Definition \ref{defi:selfsimilar}, this means that the pair $(\theta_0,f)$ is given by the self-similar change of variables
\begin{subequations}\label{eq:oOfF:0}
\begin{align}
\theta_0(t,x)
&=(abt)^{\frac{\alpha}{a}-1}\bar{\Theta}(X),\label{eq:oO:3}\\
f(t,x)
&=(abt)^{\frac{\alpha}{a}-2}F(X),\label{eq:fF:3}
\end{align}
\end{subequations}
in terms of $\bar{\Theta}$. The force $F$ is defined \textit{ad-hoc} by
\begin{equation}\label{ansatz:F:vortex}
F=-b\left((a-\alpha)+R\partial_R\right)\bar{\Theta}.
\end{equation}

Secondly, according to Definition \ref{def:selfsimilarunstable},
for some $n\in\N$ and $b>0$, there exists $0\neq W\in U_n$ satisfying $L_bW=\lambda W$ with $\Re\lambda>0$. We define
$$
\theta^{\text{lin}}(t,x)
=(abt)^{\frac{\alpha}{a}-1}\Theta^{\text{lin}}(\tau,X),
$$
where
$$
\Theta^{\text{lin}}
=\Re(e^{\lambda\tau}W).
$$

Thirdly, according to Definition \ref{def:non-linearlyunstable}, there exists a correcting term
$$
\theta^{\text{cor}}(t,x)
=(abt)^{\frac{\alpha}{a}-1}\Theta^{\text{cor}}(\tau,X),
$$
with $\Theta^{\text{cor}}$ satisfying \eqref{eq:SQG:cor} and
$$
\Theta^{\text{cor}}=o(e^{\lambda\tau}).
$$
Therefore, the temperature
$$
\theta_\epsilon
=\theta_0+\epsilon\theta^{\text{lin}}+\epsilon^2 \theta^{\text{cor}},
$$
yields a different solution to the $\alpha$-SQG equation for $\epsilon\neq 0$. 
In Section \ref{sec:proofmain}, we rigorously prove Theorem \ref{thm:main} as a consequence of the properties of $\bar{\Theta}$, $\Theta^{\text{lin}}$, and $\Theta^{\text{cor}}$, established in Theorems \ref{thm:Lb} and \ref{thm:non-linear}, along with the Sobolev scaling given in Proposition \ref{prop:scalingHs}. Similarly, in Section \ref{proof:thmunstablevortex}, we prove Theorem \ref{thm:unstablevortex} as a consequence of Theorems \ref{thm:L} and \ref{thm:non-linear}. 

As introduced in Section \ref{sec:novelties}, the proof of the existence of a self-similarly non-linearly unstable vortex $\bar{\Theta}$ is divided into three main steps: 1.~Eulerian, 2.~Self-similar, and 3.~Non-linear instability. Furthermore, Step 1.~is further divided into two intermediate steps: 1.1.~Construction of a piecewise constant unstable vortex and 1.2.~Regularization.

Next, we outline these steps, which allow us to establish Theorems \ref{thm:L}–\ref{thm:non-linear}.
The proofs of these theorems form the core of this work and will be presented rigorously in Sections \ref{sec:piecewise}–\ref{sec:non-linearinstability}.

\subsection*{Step 1.~Instability in the Eulerian coordinates}\label{Step1}

The first step involves constructing an unstable vortex that is smooth and compactly supported in the original coordinates. In other words, we must solve the stability equation \eqref{eq:SQG:selfsimilarstability} for  $b=0$. Since the construction of these vortices has its own physical and mathematical interest, and in order to maintain consistency in notation, we will adopt the convention of using lowercase letters until we return to the self-similar variables.

Let $\bar{\theta}$ be a vortex with corresponding velocity $\bar{v}$. By recalling Lemmas \ref{lemma:BiotSavartbarV} and \ref{lemma:vr:eigenfunction}, the stability equation \eqref{eq:SQG:selfsimilarstability} for $b=0$ and $w\in U_n$ reads as
\begin{equation}\label{eq:RSE:0}
\frac{1}{r}\int_0^\infty\left(I_{n,\alpha}\left(\frac{r}{s}\right)w_n(s)\partial_r\bar{\theta}(r)
-I_{1,\alpha}\left(\frac{r}{s}\right)w_n(r)\partial_r\bar{\theta}(s)\right)s^{1-\alpha}\dif s
=zw_n(r),
\end{equation}
where we have replaced $\lambda=-in C_\alpha z$, being $C_\alpha$ the constant in the $\alpha$-Biot-Savart law \eqref{eq:BiotSavart}.
Observe that $\Re\lambda>0$ translates to $\Im z>0$.
Therefore, \eqref{eq:RSE:0} implies that $w_n(r)$ must vanish wherever $\partial_r\bar{\theta}(r)=0$. Thus, we can write
$$
w_n=h\partial_r\bar{\theta},
$$
in terms of some profile $h$.
For that $h$, the equation \eqref{eq:RSE:0} reads as
\begin{equation}\label{eq:RSE}
\frac{1}{r}\int_0^\infty\left(I_{n,\alpha}\left(\frac{r}{s}\right)h(s)
-I_{1,\alpha}\left(\frac{r}{s}\right)h(r)\right)\partial_r\bar{\theta}(s)s^{1-\alpha}\dif s
=zh(r),
\quad\quad
r\in\mathrm{supp}(\partial_r\bar{\theta}).
\end{equation}
Thus, linear instability reduces to finding a vortex $\bar{\theta}$ that admits a nontrivial  eigenpair $(h,z)$ of \eqref{eq:RSE} with $\Im z>0$. We will show that this is indeed the
case and therefore land in our first main result
(recall Definition \ref{def:selfsimilarunstable}).

\begin{thm}\label{thm:L}
For every $n\geq 2$, there exists an unstable vortex
$\bar{\Theta}\in C_c^\infty\cap U_0$ such that the corresponding eigenfunction satisfies  $W\in C_c^\infty\cap U_n$. 
\end{thm}

Our proof of Theorem \ref{thm:L} is split into two steps:

\subsection*{Step 1.1.~Piecewise constant unstable vortex}\label{Step1.1}

In Section \ref{sec:piecewise} we construct an unstable vortex of the form
$$
\bar{\theta}
=c 1_{[0,r_1)} - 1_{[r_1,r_2)},
$$
in terms of some parameters $0<r_1<r_2<\infty$ and $c>0$, to be determined.
For the sake of simplicity we fix
$$
r_1=\sigma,
\quad\quad
r_2=1,
$$
for some $0<\sigma<1$, to be determined. 
Firstly, we choose $c$ making the mean of $\bar{\theta}$ equal zero, and thus $\bar{\theta}\in U_0$. This is equivalent to
\begin{equation}\label{eq:csigma:intro}
(1+c)\sigma^2=1.
\end{equation}
Secondly, since
$$
\mathrm{supp}(\partial_r\bar{\theta})=\{\sigma,1\},
$$
the linear stability equation \eqref{eq:RSE} turns into two conditions for the vector $h=(h(\sigma),h(1))$. This can be written as a linear system
\begin{equation}\label{eq:A}
A h=zh,
\end{equation}
where the matrix $A\in\R^{2\times 2}$ depends on the parameter $\sigma$ and the frequency $n$, namely
\begin{equation}\label{matrixA:intro}
A=
\left[
\begin{array}{cc}
\sigma^{-(2+\alpha)}J_{n,\alpha}(1)-\sigma^{-1}I_{1,\alpha}(\sigma) & \sigma^{-1}I_{n,\alpha}(\sigma)  \\[0.1cm]
-\sigma^{-1} I_{n,\alpha}(\sigma) & -J_{n,\alpha}(1) +\sigma^{-1} I_{1,\alpha}(\sigma) 
\end{array}
\right],
\end{equation}
with $J_{n,\alpha}=I_{1,\alpha}-I_{n,\alpha}$. 
By linear algebra, there exists an eigenvector $h\neq 0$ with corresponding eigenvalue $z$ if and only if $z$ is a root of the characteristic polynomial $\det(A-z)$. 

For the 2D Euler equation ($\alpha=0$) the coefficients of the characteristic polynomial can be computed explicitly. This allows us to obtain the existence of a non-trivial $h$ by an elementary explicit computation.
The case $\alpha>0$ is considerably more intricate since the coefficients of the characteristic polynomial are not explicit, being given in terms of the integral $I_{n,\alpha}$ evaluated at $\sigma$ and $1$. 
Furthermore, the situation is more delicate for SQG as $I_{n,1}$ diverges at $1$. Fortunately, in this case there is an extra cancellation in $J_{n,\alpha}$ that allows us to consider not only SQG but also the complete regime $0\leq\alpha<2$. In Section \ref{sec:In}, we explicitly compute certain values and derive the precise asymptotics of $I_{n,\alpha}$ in terms of the Gamma function. These results allow us 
to show that the discriminant $\Delta(\sigma)$ of the characteristic polynomial satisfies (see Proposition \ref{prop:discriminant<0})
$$
\Delta(1)=\Delta'(1)=0,
\quad\quad
\Delta''(1)<0,
$$
for all $0\leq\alpha<2$ and $n\geq 2$. Therefore, $\Delta(\sigma)<0$ for some $0<\sigma<1$, and thus $\Im z>0$. 

\begin{Rem}\label{rem:zeromean}
The zero-mean condition is imposed in the 2D Euler equations to ensure that the Hamiltonian remains finite. Although this condition is not strictly necessary for $\alpha > 0$, it can be shown that the discriminant for a general vortex satisfies $\Delta(1) \geq 0$ and vanishes if and only if the vortex has zero mean. Therefore, to complete our argument, it is convenient to work within the class $U_0$ for the full range of $\alpha$ values.   
\end{Rem}

\subsection*{Step 1.2.~Regularization}\label{Step1.2}
In Section \ref{sec:regularization} we prove that there exists a smooth vortex $\bar{\theta}^\varepsilon$, obtained by suitably regularizing $\bar{\theta}$ from \hyperref[Step1.1]{Step 1.1}, which is also unstable for some small $\varepsilon>0$. Similarly to \hyperref[Step1.1]{Step 1.1}, now we need to solve the linear stability equation \eqref{eq:RSE} in the intervals $B_\varepsilon(r_j)$ for $j=1,2$. To this aim, we first
rescale variables around $r_j$ writing
 $r=r_j+\varepsilon\alpha$ with $\alpha\in I=(-1,1)$.
Next, we  make the asymptotic expansions for eigenfunctions and eigenvalues. Namely, we write 
$$
h^\varepsilon(r)
=h(r_j)+\varepsilon g_j(\alpha),
\quad\quad
z^\varepsilon
=z+\varepsilon y,
$$
for some profiles $g=(g_1,g_2)\in L^2(I)^2$, and  a constant $y\in\C$, to be determined.
The expansion of the operator
$$A^\varepsilon
g
=\frac{1}{r}\int_0^\infty\left(I_{n,\alpha}\left(\frac{r}{s}\right)g(s)
-I_{1,\alpha}\left(\frac{r}{s}\right)g(r)\right)\partial_r\bar{\theta}^\varepsilon(s)s^{1-\alpha}\dif s,
$$
requires a more involved analysis compared to the one presented in \cite{CFMSpp}. We split $A^\varepsilon$ into
$$
A^\varepsilon
=A_0^\varepsilon
+A_1^\varepsilon.
$$
On the one hand, the regular part $A_0^\varepsilon$ can be expanded uniformly in $\alpha$, namely (see Lemma \ref{lemma:A0eps})
$$
A_0^\varepsilon
=A_0+\varepsilon B_0,
$$
for some linear operator $B_0=B_0^\varepsilon$ bounded in $L^2(I)^2$. On the other hand, the singular part $A_1^\varepsilon$ can be expanded for $\alpha<1$ as (see Lemma \ref{lemma:A1eps:0<a<1})
\begin{equation}\label{eq:A1eps:intro}
A_1^\varepsilon
=A_1+\frac{\varepsilon^{1-\alpha}}{1-\alpha} B_1,
\end{equation}
for another linear operator $B_1=B_1^\varepsilon$ bounded in $L^2(I)^2$. Therefore, for $\alpha<1$ we can follow \cite{CFMSpp}. We recall the argument   here for the reader's convenience. The linear stability equation
$$
(A^\varepsilon-z^\varepsilon)h^\varepsilon=0,
$$
is rewritten as 
\begin{equation}\label{eq:RSE:eps:0<a<1:intro}
(A-z)g
=(y-B_0)h+ \left(\varepsilon y-\frac{\varepsilon^{1-\alpha}}{1-\alpha}B\right)g,
\end{equation}
where
$$
A=A_0+A_1,
\quad\quad
B=(1-\alpha)\varepsilon^\alpha B_0+B_1.
$$
Now the first step consists in ``inverting'' the operator $(A-z)$.
Unfortunately, since $z,z^*$ are the eigenvalues of $A$, the kernel and the image of $(A-z)$ are given by $\mathrm{Ker}(A-z)=\mathrm{span}(h)$ and $\mathrm{Im}(A-z)=\mathrm{span}(h^*)$. We bypass this obstacle by choosing $y$ in such a way that the right hand side of \eqref{eq:RSE:eps:0<a<1:intro} becomes ``parallel'' to $h^*$ (see Lemma \ref{lemma:f:0<a<1}). This allows rewriting the equation \eqref{eq:RSE:eps:0<a<1:intro} appropriately to apply a fixed-point argument (see Proposition \ref{prop:fixedpoint}).

Notice that the expansion \eqref{eq:A1eps:intro} degenerates as $\alpha\to 1$, inverting the roles of the main term and the error. In fact, we prove that for $\alpha=1$ (see Lemma \ref{lemma:A1:a=1})
$$
A_1^\varepsilon
=(\log\varepsilon)A_1+B_1,
$$
where now
$$
(A_1g)_j
= \frac{2c_j}{r_j}
\left(g_j-\int_{-1}^{1}g_j\eta\dif\rho\right)
$$
and $B_1=B_1^\varepsilon$ is a linear operator bounded in $L^2(I)^2$. Here, $\eta$ is a suitably chosen mollifier (see Lemma \ref{lemma:mollifier}). 
By decomposing $$g=\mu+\frac{f}{\log\varepsilon},$$ where $\mu\in \mathbb{C}^2$ and $f\in L^2(I)^2$, the stability equation takes the following form (see Corollary \ref{cor:RSE:a=1})
\begin{equation*}
A_1f
=(y-B_0)h-(A_0-z)\mu
+\varepsilon(y-B_0)\mu
+\frac{1}{\log\varepsilon}((z-B_1)+\varepsilon y)f.
\end{equation*}

Thus, in order to invert $A_1$, we now need to ensure that the right-hand side is compatible with the zero-mean condition of $A_1$. By selecting $y$ and $\mu$ appropriately, we show that this condition holds (see Lemma \ref{lemma:inverting:a=1}). Finally, we apply a fixed-point argument (see Proposition \ref{prop:fixedpoint:SQG}).

Once $\varepsilon>0$ is fixed, Theorem \ref{thm:L} holds for
$$
\bar{\Theta}=\bar{\theta}^\varepsilon.
$$
From now on $\bar{\Theta}\in C_c^\infty\cap U_0$ is fixed and thus we will omit $\varepsilon$ for the sake of simplicity. Finally, we check that the eigenfunction satisfies $W\in C_c^\infty\cap U_n$ in Proposition \ref{prop:eigenfunction:b=0}.

\subsection*{Step 2.~Instability in self-similar coordinates}\label{Step2}

In order to prove that the unstable vortex $\bar{\Theta}$
is also self-similarly unstable, we follow Vishik's argument. This requires decomposing the operator $L_b$ acting on $L_n^2$, as
$$
L_b=A_b+C,
$$
where $(A_b)$ is a family of linear operators that generate contraction semigroups and possess certain continuity with respect to the parameter $b$, and $C$ is compact (See Proposition \ref{prop:SpectralAnalysis}). By classical operator theory, this implies that the spectrum $\sigma(L_b)$ satisfies that, for any $b\geq 0$ and $w>0$,
$$
\sigma(L_b)\cap\{\Re\lambda>w\}
$$
is finite and consists of isolated eigenvalues. Now, Step 1 provides an eigenvalue $\lambda_0$ with positive real part for $b=0$.
Then, using the continuity with respect to the parameter $b$, it is possible to show that there must also be eigenvalues $\lambda_b$ near $\lambda_0$ for sufficiently small $b>0$. 

For the 2D Euler equation ($\alpha=0$) this argument is applied to the decomposition
$$
A_b=b(a-\alpha) + T_b,
\quad\quad
C=K,
$$
where $T_b$ and $K$ are defined in \eqref{Lb}.
In this work, we check $K$ remains compact for all $0\leq \alpha<1$ (see Lemma \ref{lemma:Kcompact}) and thus Vishik's argument can be applied. For $\alpha=1$, the operator $K$ is bounded but not compact. We bypass this obstacle by decomposing $K$ into (see Proposition \ref{prop:K=S+C})
$$
K=S+C,
$$
where the operator $S$ is skew-adjoint, thus preserving the growth of the semigroup, and the operator $C$ is a commutator that can be shown to be compact (see Lemma \ref{lemma:commutator}). Therefore, we can apply Vishik's argument to
 the decomposition
$$
A_b=b(a-\alpha) + T_b +S,
\quad\quad
C=K-S.
$$
As a consequence, we obtain the following result for $0\leq\alpha\leq 1$ (recall Definition~\ref{def:selfsimilarunstable}). 

\begin{thm}\label{thm:Lb} For every $m\in \mathbb{N}$, there exist $b>0$ and $j\neq 0$ such that the vortex $\bar{\Theta}$ from Theorem \ref{thm:L} is also self-similarly unstable and the corresponding   eigenfunction satisfies $W\in C_c^{m}\cap U_{jn}$. 
\end{thm}

The properties of the eigenfunction are proved in Proposition \ref{prop:eigenfunction}. Indeed, this proposition states that $W\in C_c^{m+2}$, as we address the possible loss of derivatives in the next step.

\subsection*{Step 3.~Non-linear instability}\label{Step3}

The last step entails proving that $\bar{\Theta}$ is also non-linearly unstable.
To this end, it suffices to prove that there exists $\Theta^{\text{cor}}$ solving \eqref{eq:SQG:cor} and \eqref{eq:Thetacorinitial} with
\begin{equation}\label{eq:thetacorineq}
\|\Theta^{\text{cor}}(\tau)\|_{H_\omega^m}\leq e^{(1+\delta_0)\Re\lambda\tau},
\end{equation}
for all $-\infty <\tau\leq\bar{\tau}$, for some $\delta_0>0$ and $\bar{\tau}>-\infty$. Notice that we have replaced the $L^2$ space in \eqref{eq:expdecay} with the stronger weighted Hilbert space $H_\omega^m$ given by the norm 
$$
\|f\|_{H_\omega^m}
:=\|f\omega\|_{L^2}
+\sum_{0<|K|\leq m}\|\partial_X^K f\omega\|_{L^2},
$$
where $\omega$ is the standard radial weight 
$$
\omega(X)
=\langle X\rangle^2
=1+R^2.
$$
Moreover, we restrict $H_\omega^m$ to the space of $n$-fold symmetric temperatures, that is, $H_\omega^m\subset L_n^2$.
The purpose of introducing the weight $\omega$ is to control our solutions in general Sobolev spaces by taking $m$ big enough (see Lemma \ref{lemma:embeddingYmSobolev}). In the remainder of this section, we discuss the strategy for proving \eqref{eq:thetacorineq}. 
The proof starts
by making some a priori energy estimates. Namely, the Duhamel formula 
$$
\Theta^{\text{cor}}(\tau)
=
-\int_{-\infty}^\tau 
e^{(\tau-\tau')L_b}
\mathcal{F}(\tau')\dif\tau',
$$
where recall
$$
\mathcal{F}
=(V^{\text{lin}}+\epsilon V^{\text{cor}})
\cdot\nabla(\Theta^{\text{lin}}+\epsilon\Theta^{\text{cor}}),
$$
gives that
$$
\|\Theta^{\text{cor}}(\tau)\|_{L^2}\leq e^{2 \Re\lambda\tau},
$$ 
over the time interval for which
\begin{equation}\label{eq:aprioriestimates}
\|V^{\text{cor}}(\tau)\|_{L^\infty}\leq e^{\Re\lambda\tau},
\quad\quad
\|\nabla\Theta^{\text{cor}}(\tau)\|_{L^2}\leq e^{\Re\lambda\tau}.
\end{equation}
This requires choosing $m$ that controls $\|V^{\text{cor}}\|_{L^\infty}$ and $\|\nabla\Theta^{\text{cor}}\|_{L^2}$. Hence, at this point, it would suffice to take $m>1$.
Note that we need to prove that $\Theta^{\text{cor}}$ actually exists in the space $C((-\infty,\bar{\tau}];H_{\omega}^m)$. Following the approach in \cite{Vishikpp1,Vishikpp2,ABCDGMKpp,CFMSpp}, we initialize the (self-similar) $\alpha$-SQG equation at time $\tau=-k$, for which both the initial data 
$$
\bar{\Theta}+\epsilon\Theta^{\text{lin}}|_{\tau=-k},
$$
and the forcing term are sufficiently smooth to define a (local in time) solution
$$
\Theta_\epsilon
=\bar{\Theta}+\epsilon\Theta^{\text{lin}}
+\epsilon^2\Theta^{\text{cor}}.
$$
In other words, for every $k\in\N$, we consider the unique solution $\Theta_k^{\text{cor}}$ to \eqref{eq:SQG:cor} coupled with the initial condition (instead of \eqref{eq:Thetacorinitial})
$$
\Theta_k^{\text{cor}}|_{\tau=-k}
=0.
$$
Thus, we need $H_{\omega}^m$ to be a well-posedness function space for $\alpha$-SQG, which requires $m>1+\alpha$. Consequently, we obtain that \eqref{eq:thetacorineq} holds in this space, at least on a time interval $[-k,\tau_k]$.

To prove that the time of existence $\tau_k$ does not diverge to $-\infty$ when $k\to \infty$, a bootstrapping argument is used in \cite{Vishikpp1,Vishikpp2,ABCDGMKpp,CFMSpp}. As noted in the previous works, the bootstrapping argument does not work in Cartesian coordinates, but it does in polar coordinates, provided that the partial derivatives are estimated in the correct order. In this work, we realize that these inductive energy estimates can be extended to spaces with higher-order derivatives.

To this end, we need to express Cartesian derivatives in terms of polar coordinates. In fact, we use the following relation
$$
\partial_X^K f
=\sum_{0<|J|\leq|K|}p_J^K\frac{\partial_R^{j_1}\partial_\phi^{j_2}\Theta}{R^{|K|-j_1}},
$$
where $p_J^K=p_J^K(\cos\phi,\sin\phi)$ are polynomials (see Appendix \ref{sec:cartesianpolar}). These partial derivatives need to be estimated in the following order
$$
\begin{array}{cccc}
    \displaystyle\frac{\partial_\phi}{R^m} & \displaystyle\frac{\partial_R}{R^{m-1}} & &  \\[0.5cm]
    \displaystyle\frac{\partial_\phi^2}{R^m} & \displaystyle\frac{\partial_R\partial_\phi}{R^{m-1}} & \displaystyle\frac{\partial_R^2}{R^{m-2}} & 
    \\[0.5cm]
    \displaystyle\frac{\partial_\phi^3}{R^m} & \displaystyle\frac{\partial_R\partial_\phi^2}{R^{m-1}} & \displaystyle\frac{\partial_R^2\partial_\phi}{R^{m-2}} & \displaystyle\frac{\partial_R^3}{R^{m-3}}
    \\[0.5cm]
    \cdots
\end{array}
$$
This suggests to consider the space $Y^m\subset L_n^2$ endowed with the norm
$$
\|f\|_{Y^m}
:=\|f\omega\|_{L^2}
+\sum_{0<|J|\leq m}\left\|\frac{\partial_R^{j_1}\partial_\phi^{j_2}f}{R^{m-j_1}}\omega \right\|_{L^2},
$$
which turns out to be equivalent to the norm (see Lemma \ref{lemma:Marcos})
$$
\|f\|_{Z^m}
:=\|f\omega\|_{L^2}
+\sum_{0<|K|\leq m}\left\|\frac{\partial_X^K f}{R^{m-|K|}}\omega \right\|_{L^2}.
$$
Notice that, although $Y^1=H_\omega^1$, in general $Y^m\subsetneq H^m_\omega$ for $m>1+\alpha$ due to the weight at the origin (see Lemma \ref{lemma:YminweightedHm}).

In the previous works for the 2D Euler equation \cite{Vishikpp1,Vishikpp2,ABCDGMKpp,CFMSpp}, the energy estimates were conducted for the global smooth solution starting at time $\tau=-k$. However, for $\alpha>0$ the lack of global well-posedness makes it necessary to regularize the system in order to justify the energy estimates.
Furthermore, standard regularization using mollifiers does not fit well with the weighted energy space $Y^m$. Instead, we used the regularization that comes from linearizing the equation.
We consider the following linearized system. Starting from $\Theta_k^{(0)}=0$, we define recursively $\Theta_k^{(q)}$, for every $q\in\N$, as the solution to
$$
(\partial_\tau - L_b)\Theta_k^{(q)}
+\underbrace{(V^{\text{lin}}+\epsilon V_k^{(q-1)})\cdot\nabla(\Theta^{\text{lin}}+\epsilon\Theta_k^{(q)})}_{\mathcal{F}_k^{(q)}}
=0.
$$
This iteration defines a sequence of global solutions, which can be shown to lie in $Y^{m+1}$ (see Proposition \ref{prop:ThetakqYm}). Then, we apply the inductive energy estimates on $Y^m$ to gain extra decay (see Proposition \ref{DerivativesboundedImpliesHbounded}). With this, we deduce that the time of existence in our iteration scheme does not collapse (see Corollary \ref{cor:bartau}). Finally, by a standard argument, we verify that the sequence converges to a solution as $q \to \infty$ and later $k \to \infty$. We obtain the following result (recall Definition \ref{def:non-linearlyunstable}).

\begin{thm} \label{thm:non-linear} 
The vortex $\bar{\Theta}$ from Theorem \ref{thm:L} (respectively Theorem \ref{thm:Lb}) is also (self-similarly) non-linearly unstable. Moreover, there exist $\delta_0>0$ and $\bar{\tau}>-\infty$ such that the correcting term is $n$-fold symmetric and satisfies $\Theta^\text{cor}\in C((-\infty,\bar{\tau}],H_{\omega}^m\cap \dot{H}^{\frac{\alpha-2}{2}})$ with
$$
\|\Theta^{\text{cor}}(\tau)\|_{H_\omega^m}\leq e^{(1+\delta_0)\Re\lambda\tau}, 
$$
for all $-\infty<\tau\leq\bar{\tau}$.
\end{thm}

\subsection{Proof of Theorem \ref{thm:main}}\label{sec:proofmain}

Coming back to the original coordinates, 
Theorems \ref{thm:Lb} and \ref{thm:non-linear} imply that
the temperature
$$
\theta_\epsilon
=\theta_0
+\epsilon\theta^{\text{lin}}
+\epsilon^2\theta^{\text{cor}},
$$
where recall
$$
\theta_0(t,x)
=\frac{1}{abt}\bar{\Theta}(X),
\quad\quad
\theta^{\text{lin}}(t,x)
=\frac{1}{abt}\Theta^{\text{lin}}(\tau,X),
\quad\quad
\theta^{\text{cor}}(t,x)
=\frac{1}{abt}\Theta^{\text{cor}}(\tau,X),
$$ 
is a solution to the $\alpha$-SQG equation with the same initial datum $\theta^\circ=0$ and (radial) force
$$
f(t,x)
=-b(abt)^{\frac{\alpha}{a}-2}\left((a-\alpha)+R\partial_R\right)\bar{\Theta}(R).
$$  
By applying the regularity obtained in Theorems \ref{thm:Lb} and \ref{thm:non-linear}, the Sobolev scaling (Proposition \ref{prop:scalingHs}) 
and the embedding $H_{\omega}^m\hookrightarrow W^{s,p}$
 (Lemma \ref{lemma:embeddingYmSobolev}) applied to $s\leq\alpha+\frac{2}{p}-\varepsilon\leq 3\leq m-2$, we get
$$
\|\theta^{\text{lin}}(t)\|_{\dot{W}^{s,p}}
= (abt)^{\frac{\alpha+\frac{2}{p}-s}{a}-1}\|\Re(e^{\lambda\tau}W)\|_{\dot{W}^{s,p}},
\quad\quad
\|\theta^{\text{cor}}(t)\|_{\dot{W}^{s,p}}
=(abt)^{\frac{\alpha+\frac{2}{p}-s}{a}-1}o(e^{\Re\lambda\tau}).
$$
These equalities, in combination with \eqref{eq:a} and Corollary \ref{cor:selfsimilarscaling}, imply \eqref{forceintegrability:improved} and \eqref{thetaintegrability:improved} with $T=e^{ab\bar{\tau}}$.
By the reverse triangle inequality, we deduce that for different values of $\epsilon$, we obtain different solutions
$$
(abt)^{1-\frac{\alpha+1-r}{a}}e^{-\Re\lambda\tau}\|(\theta_\epsilon-\theta_{\bar{\epsilon}})(t)|\|_{L^2}
\geq
|\epsilon-\bar{\epsilon}|
(\|\Re(e^{i\Im\lambda\tau}W)\|_{L^2}
-o(1)).
$$
If $\Im\lambda\neq 0$, the right-hand side is positive in a sequence of times $\tau_k^{(q)}=\tau_0-\frac{2\pi k}{\Im\lambda}\to -\infty$ as $k\to\infty$. If $\Im\lambda=0$, we can assume from the beginning that $W$ is real-valued. Otherwise, it would suffice to take its imaginary part instead.
This allows concluding that
$\theta_\epsilon\neq\theta_{\bar{\epsilon}}$ whenever $\epsilon\neq\bar{\epsilon}$. 

\subsection{Proof of Theorem \ref{thm:unstablevortex}}\label{proof:thmunstablevortex}
We consider the $n$-fold symmetric temperature
$$\theta=\bar{\Theta}+\Theta^{\text{lin}}+\Theta^{\text{cor}}
\in C((-\infty,\bar{\tau}],H^m\cap \dot{H}^{\frac{\alpha-2}{2}}),$$
as obtained from Theorems \ref{thm:L} and \ref{thm:non-linear}.
Here we have $b=0$, and thus $f=0$.
It follows that $\theta$ satisfies the conditions in Theorem \ref{thm:unstablevortex}. 
By shifting the time interval, we can assume that $\bar{\tau}=0$.

\subsection{The Golovkin trick}\label{sec:Golovkin}

In a recent preprint \cite{DolceMescolinipp}, Dolce and Mescolini revived a clever trick from Golovkin \cite{Golovkin64} which allows proving non-uniqueness once self-similar instability is established at the linear level. As the authors point out, the trick appears to work for quadratic PDEs. Here, we revisit this trick in the context of the $\alpha$-SQG equation for the reader's convenience.

Once Steps (1)-(3) have been completed, we have a self-similarly unstable vortex $\bar{\Theta}\in C_c^\infty$ and the corresponding eigenvalue $\lambda\in\C_+$ and eigenfunction $W\in C_c^m$ solving $L_b W=\lambda W$, with which we define $\Theta^{\text{lin}}=\Re(e^{\lambda\tau}W)$. Then, we consider the two temperatures
\begin{equation}\label{Golovkin:solutions}
\Theta_\pm
=\bar{\Theta}\pm\Theta^{\text{lin}}.
\end{equation}
It is straightforward to check that $\Theta_+$ and $\Theta_-$ solve the (self-similar) $\alpha$-SQG equation \eqref{eq:SQG:SS}
with the same initial condition 
$$
\Theta_{\pm}|_{\tau=-\infty}=\bar{\Theta},
$$
and the Golovkin force
$$
G=F+V^{\text{lin}}\cdot\nabla\Theta^{\text{lin}},
$$
where $F$ is the Vishik force \eqref{ansatz:F:vortex}.
The smoothness of the solutions allows us to conclude Theorem \ref{thm:main} through the Sobolev scaling (Proposition \ref{prop:scalingHs}).
Vishik's and Golovkin's forces differ in several aspects. On the one hand, $F$ is radially symmetric and self-similar, while $G$ is $n$-fold symmetric and depends on $\tau$ due to the quadratic term. On the other hand, $F=0$ for $b=0$, while $G$ is not.

\section{Piecewise constant unstable vortex}\label{sec:piecewise}

In this section we construct a piecewise constant unstable vortex.
To this end, we need to solve the linear stability equation \eqref{eq:RSE}
\begin{equation}\label{eq:RSE:I}
\frac{1}{r}
\int_0^\infty
\left(I_{n,\alpha}\left(\frac{r}{s}\right)
h(s)-I_{1,\alpha}\left(\frac{r}{s}\right)
h(r)\right)\partial_r\bar{\theta}(s)s^{1-\alpha}\dif s
=zh(r)
\quad\quad
r\in\text{supp}(\partial_r\bar{\theta}),
\end{equation}
for some profile $h$ and eigenvalue $z$ with $\Im z>0$, to be determined.

\subsection{The kernel $I_{n,\alpha}$}\label{sec:In}
In this section we analyze the parametric integral
\begin{equation}\label{eq:In}
I_{n,\alpha}(\sigma)
=\frac{1}{\alpha}\int_{-\pi}^{\pi}\frac{\cos(n\beta)}{|\sigma-e^{i\beta}|^{\alpha}}\dif\beta
=\frac{\sigma}{n}\int_{-\pi}^{\pi}\frac{\sin(\beta)\sin(n\beta)}{|\sigma-e^{i\beta}|^{2+\alpha}}\dif\beta.
\end{equation}
See Remark \ref{Rem:Iforalpha=0} for the last equality.
The function \eqref{eq:In} is well defined and smooth for all $\sigma\geq 0$ with $\sigma\neq 1$. Notice that $I_{n,\alpha}(\sigma)\to 0$ as $\sigma\to 0,\infty$. However, the analysis of the behavior near $\sigma=1$ is rather cumbersome and takes a few technical lemmas. For our purposes, it will be enough to focus on the range 
$$
|1-\sigma|\leq\frac{1}{2}.
$$
The proofs of the following lemmas will be carried out for $\alpha>0$. However, in the statements, we will also include the case 
$\alpha=0$ since the formulas extend to this case. In fact, the kernel $I_{n,0}$ can be computed explicitly (see \cite[Appendix]{CFMSpp})
\begin{equation}\label{eq:In0}
I_{n,0}(\sigma)
=\frac{\pi}{n}\min\{\sigma,\sigma^{-1}\}^n.
\end{equation}

In the first lemma, we compute the limit of $I_{n,\alpha}(\sigma)$ as $\sigma \to 1$ explicitly. The proof is simpler for $n = 2$, but we include the general case to address any $n$-fold symmetry in Theorem \ref{thm:unstablevortex}. Continuity follows immediately from the dominated convergence theorem, but computing $I_{n,\alpha}(1)$ is more involved and requires the use of certain combinatorial identities, as well as properties of the Beta and Gamma functions.

\begin{lemma}\label{lemma:In}
Let $n\in\N$. For $0\leq\alpha<1$, the function
$I_{n,\alpha}(\sigma)$ is continuous at $\sigma=1$. Moreover,
$$
I_{n,\alpha}(1)
=\frac{D_\alpha}{1-\alpha}F_n(\alpha)>0,
$$
where
$$
D_\alpha=\frac{2\sqrt{\pi}}{2^\alpha}
\frac{\Gamma\left(\frac{3-\alpha}{2}\right)}{\Gamma\left(2-\frac{\alpha}{2}\right)},
$$
$F_1=1$ and, for $n\geq 2$,
$$
F_n(\alpha)
=\prod_{k=1}^{n-1}f_k(\alpha)
\quad\text{with}\quad
f_k(\alpha)
=\frac{2k+\alpha}{2k+(2-\alpha)}.
$$
For $1\leq\alpha<2$, it holds that
$$
\lim_{\sigma\to 1}I_{n,\alpha}(\sigma)=\infty.
$$
\end{lemma}
\begin{proof}
For $0<\alpha<1$, the integrand inside $I_{n,\alpha}(\sigma)$ is integrable at $\sigma=1$, namely, by elementary trigonometric identities,
\begin{equation}\label{eq:Incompute:1}
I_{n,\alpha}(1)
=\frac{2}{\alpha}\int_0^{\pi}
\frac{\cos(n\beta)}{|1-e^{i\beta}|^\alpha}\dif\beta
=\frac{2}{\alpha}\int_0^{\pi}\frac{\cos(n\beta)}{(2\sin(\beta/2))^\alpha}\dif\beta
=\frac{2^{2-\alpha}}{\alpha}
\int_0^{\pi/2}\frac{\cos(2n\beta)}{(\sin\beta)^\alpha}\dif\beta.
\end{equation}
Then, we apply that $\cos(2n\beta)=T_{2n}(\cos\beta)$, where
$$
T_{2n}(x)
=\sum_{j=0}^n\binom{2n}{2j}(x^2-1)^jx^{2(n-j)},
$$
is the Chebyshev polynomial of degree $2n$. Hence, by making the changes of variables $u=\sin\beta$ and $v=u^2$, we get
\begin{equation}\label{eq:Incompute:2}
\begin{split}
\eqref{eq:Incompute:1}
&=\frac{2^{2-\alpha}}{\alpha}
\sum_{j=0}^n\binom{2n}{2j}(-1)^j
\int_0^1\frac{u^{2j}(1-u^2)^{n-j}}{u^\alpha}\frac{\dif u}{\sqrt{1-u^2}}\\
&=\frac{2^{2-\alpha}}{\alpha}
\sum_{j=0}^n\binom{2n}{2j}(-1)^j
\int_0^1 v^{j-\frac{\alpha}{2}}(1-v)^{n-j-\frac{1}{2}}\frac{\dif v}{2\sqrt{v}}\\
&=\frac{2^{1-\alpha}}{\alpha}
\sum_{j=0}^n\binom{2n}{2j}(-1)^j
B\left(j+\frac{1-\alpha}{2},n-j+\frac{1}{2}\right),
\end{split}
\end{equation}
where $B$ is the Beta function. By applying the identity
$$
B(x+j,y-j)
=B(x,y)\prod_{k=1}^j
\frac{x+k-1}{y-k},
$$
we get,
\begin{equation}\label{eq:Incompute:3}
\eqref{eq:Incompute:2}
=\frac{2^{1-\alpha}}{\alpha}
B\left(\frac{1-\alpha}{2},n+\frac{1}{2}\right)
p_n(\alpha),
\end{equation}
where $p_n$ is the following polynomial of degree $n$
\begin{align*}
p_n(\alpha)
&=\sum_{j=0}^n\binom{2n}{2j}(-1)^j
\prod_{k=1}^j
\frac{\frac{1-\alpha}{2}+k-1}{n+\frac{1}{2}-k}\\
&=\sum_{j=0}^n\binom{2n}{2j}(-1)^j
\prod_{k=1}^j
\frac{2k-1-\alpha}{2(n-k)+1}\\
&=\sum_{j=0}^n\binom{n}{j}(-1)^j
\prod_{k=1}^j
\frac{2k-1-\alpha}{2k-1}.
\end{align*}
In the last equality, we have applied that
$$
\binom{2n}{2j}
=\prod_{i=0}^{2j-1}\frac{2n-i}{i+1}
=\prod_{k=0}^{j-1}
\frac{2(n-k)}{2k+1}
\frac{2(n-k)-1}{2(k+1)}
=\binom{n}{j}\prod_{k=1}^{j}
\frac{2(n-k)+1}{2k-1},
$$
where we have separated the cases $i=2k$ and $i=2k+1$ for $k=0,\ldots,j-1$ in the second equality. 

Next, we claim that we can express $p_n$ simply as
\begin{equation}\label{claim:pn}
p_n(\alpha)=c_n 
q_n(\alpha)
\quad\text{with}\quad
q_n(\alpha)=\prod_{m=0}^{n-1}(2m+\alpha),
\end{equation}
where $c_n$ is a constant, to be determined.
Since both $p_n$ and $q_n$ are polynomials of degree $n$, it is sufficient to check that the roots of $p_n$ are the same as $q_n$, namely
$\alpha_m=-2m$ for $m=0,\ldots,n-1$. 
Notice that we can write
\begin{equation}\label{eq:pn2m}
p_n(-2m)
=\sum_{j=0}^n\binom{n}{j}(-1)^j
\pi_m(j),
\end{equation}
where $\pi_0=1$ and, for any $m=1,\ldots,n$, 
$$
\pi_m(j)
=
\prod_{k=1}^j
\frac{2k-1+2m}{2k-1}
=\prod_{k=1}^m
\frac{2k-1+2j}{2k-1}
$$
is a polynomial of degree $m$. The last equality is trivial for $j=m$. The case $j>m$ follows by checking that some numerators and denominators cancel out. The case $j<m$ is symmetric. 

Since the first $(n-1)$ derivatives of the polynomial
$$
(1+x)^n=\sum_{j=0}^n\binom{n}{j}x^j
$$
vanish at $x=-1$, it follows that
$$
\sum_{j=0}^n\binom{n}{j}(-1)^j \pi(j)=0,
$$
for any polynomial $\pi$ of degree strictly less that $n$. In particular, $\eqref{eq:pn2m}=0$ for any $m=0,\ldots,n-1$. For $m=n$, since we can decompose
$$
\pi_{n}
=\tilde{\pi}_n
+\pi
\quad\text{with}\quad
\tilde{\pi}_n(j)=2^n \prod_{k=0}^{n-1}\frac{j-k}{2k+1},
$$
where $\pi$ is a polynomial of degree strictly less that $n$, we deduce that
$$
p_n(-2n)
=\sum_{j=0}^n\binom{n}{j}(-1)^j\tilde{\pi}_n(j)
=(-2)^n\tilde{\pi}_n(n)
=(-2)^n n!\prod_{k=0}^{n-1}\frac{1}{2k+1}.
$$
Thus, since 
$$
q_n(-2n)
=\prod_{k=0}^{n-1}(2k-2n)
=(-2)^nn!,
$$
we determine the constant in \eqref{claim:pn}
$$
c_n=\frac{p_n}{q_n}(-2n)
=\prod_{k=0}^{n-1}\frac{1}{2k+1}.
$$
Therefore, we have
\begin{equation}\label{eq:Incompute:4}
\eqref{eq:Incompute:3}
=\frac{2^{1-\alpha}}{\alpha}
B\left(\frac{1-\alpha}{2},n+\frac{1}{2}\right)
\prod_{k=0}^{n-1}\frac{2k+\alpha}{2k+1}.
\end{equation}
Next, by applying the identities
$$
B(x,y)
=\frac{\Gamma(x)\Gamma(y)}{\Gamma(x+y)},
\quad\quad
\Gamma(n+x)
=\Gamma(x)\prod_{k=0}^{n-1}(k+x),
$$
we get
$$
B\left(\frac{1-\alpha}{2},n+\frac{1}{2}\right)
=\frac{\Gamma\left(\frac{1-\alpha}{2}\right)\Gamma\left(n+\frac{1}{2}\right)}{\Gamma\left(n+1-\frac{\alpha}{2}\right)}
=\frac{\Gamma\left(\frac{1-\alpha}{2}\right)\Gamma\left(\frac{1}{2}\right)}{\Gamma\left(1-\frac{\alpha}{2}\right)}
\prod_{k=0}^{n-1}\frac{2k+1}{2k+(2-\alpha)}.
$$
This implies that
\begin{equation}\label{eq:Incompute:5}
\eqref{eq:Incompute:4}
=\frac{2^{1-\alpha}\sqrt{\pi}}{\alpha}
\frac{\Gamma\left(\frac{1-\alpha}{2}\right)}{\Gamma\left(1-\frac{\alpha}{2}\right)}
\prod_{k=0}^{n-1}\frac{2k+\alpha}{2k+(2-\alpha)}.
\end{equation}

For $1\leq\alpha<2$, by expanding in Taylor series
$$
I_{n,\alpha}(\sigma)
=\frac{1}{\alpha}\int_{-\pi}^{\pi}
\frac{1}{|\sigma-e^{i\beta}|^\alpha}\dif\beta
+\frac{1}{\alpha}\int_{-\pi}^{\pi}
\frac{O(\beta^2)}{|\sigma-e^{i\beta}|^\alpha}\dif\beta,
$$
it is clear that the second term is finite, while the first one diverges to infinity as $\sigma\to 1$.
\end{proof}

In the next two lemmas we study the behavior of $I_{n,\alpha}(\sigma)$ near $\sigma=1$ for $0\leq\alpha<1$ and $\alpha=1$, respectively.

\begin{lemma}\label{lemma:estimatesIn1}
Let $n\in\N$ and $0\leq\alpha<1$. There exists $C>0$ such that
$$
|I_{n,\alpha}'(\sigma)|
\leq\frac{C}{|1-\sigma|^\alpha},
$$
for all $0<|\sigma-1|\leq\frac{1}{2}$.
Therefore, 
$$
|I_{n,\alpha}(\sigma)-I_{n,\alpha}(1)|
\leq C\frac{|1-\sigma|^{1-\alpha}}{1-\alpha}.
$$
\end{lemma}

\begin{proof}
For $0<\alpha<1$ and $0<|\sigma-1|\leq\frac{1}{2}$, we can differentiate under the integral sign
$$
I_{n,\alpha}'(\sigma)
=-\int_{-\pi}^{\pi}\frac{\sigma-\cos\beta}{|\sigma-e^{i\beta}|^{2+\alpha}}\cos(n\beta)\dif\beta.
$$
This can be bounded by
\begin{align*}
|I_{n,\alpha}'
(\sigma)|
&\leq 2
\int_{0}^{\pi}\frac{1}{|\sigma-e^{i\beta}|^{1+\alpha}}\dif\beta
= 4\int_{0}^{\pi/2}
\frac{1}{\left((1-\sigma)^2+\sigma(2\sin\beta)^2\right)^{\frac{1+\alpha}{2}}}\dif\beta\\
&=
4\int_{0}^{\pi/2}
\frac{1-\cos\beta}{\left((1-\sigma)^2+\sigma(2\sin\beta)^2\right)^{\frac{1+\alpha}{2}}}\dif\beta
+4\int_{0}^{\pi/2}
\frac{\cos\beta}{\left((1-\sigma)^2+\sigma(2\sin\beta)^2\right)^{\frac{1+\alpha}{2}}}\dif\beta
=A+B.
\end{align*}
The first term $A$ can be bounded easily. For the bad term $B$, by making the changes of variables $u=\sin\beta$ and $w=\frac{2\sqrt{\sigma}}{|1-\sigma|}u$, we estimate
\begin{align*}
B&=4\int_{0}^{1}
\frac{\dif u}{\left((1-\sigma)^2+(2\sigma u)^2\right)^{\frac{1+\alpha}{2}}}
=\frac{2}{\sigma|1-\sigma|^\alpha}\int_{0}^{\frac{2\sqrt{\sigma}}{|1-\sigma|}}
\frac{\dif w}{\left(1+w^2\right)^{\frac{1+\alpha}{2}}}\\
&\leq\frac{4}{|1-\sigma|^\alpha}\int_{0}^{\infty}
\frac{\dif w}{\left(1+w^2\right)^{\frac{1+\alpha}{2}}}
\leq\frac{C}{\alpha|1-\sigma|^\alpha}.
\end{align*}
Notice that for $\alpha=0$ the function $I_{n,0}(\sigma)$ is Lipschitz at $\sigma=1$ (see \eqref{eq:In0}).
Finally, we apply the fundamental theorem of calculus.
\end{proof}

The following lemma was already proved in \cite[Lemma 4.12]{CCGSglobal}. Here, we present a more direct proof for the reader's convenience. 

\begin{lemma}\label{lemma:estimatesIn2}
Let $n\in\N$. There exists $R\in C^1$ such that
$$
I_{n,1}(\sigma)=-\frac{2}{\sqrt{\sigma}}\log|1-\sigma|+R(\sigma),
$$
for all $0<|\sigma-1|\leq\frac{1}{2}$.
\end{lemma}
\begin{proof}
We split
\begin{align*}
I_{n,1}(\sigma)
&=4\int_0^{\pi/2}
\frac{\cos(2n\beta)}{\sqrt{(1-\sigma)^2+\sigma(2\sin\beta)^2}}\dif\beta\\
&=4\int_0^{\pi/2}
\frac{\cos(2n\beta)-\cos\beta}{\sqrt{(1-\sigma)^2+\sigma(2\sin\beta)^2}}\dif\beta
+4\int_0^{\pi/2}
\frac{\cos\beta}{\sqrt{(1-\sigma)^2+\sigma(2\sin\beta)^2}}\dif\beta
=A+B.
\end{align*}
By the dominated convergence theorem, it is easy to check that $A\in C^1$. In fact, $A=-J_{2n,1}$ which is analyzed in more detail in Lemmas \ref{lemma:Jn} and \ref{lemma:Jndiff1}.
For the bad term $B$, by making the changes of variables $u=\sin\beta$ and $w=\frac{2\sqrt{\sigma}}{|1-\sigma|}u$, we compute
\begin{align*}
B=\frac{2}{\sqrt{\sigma}}\int_0^{\frac{2\sqrt{\sigma}}{|1-\sigma|}}
\frac{\dif w}{\sqrt{1+w^2}}
=\frac{2}{\sqrt{\sigma}}\mathrm{arcsinh}\left(\frac{2\sqrt{\sigma}}{|1-\sigma|}\right)
=-\frac{2}{\sqrt{\sigma}}\log|1-\sigma|+C,
\end{align*}
where $C$ is a smooth function.
In the last equality we have used the formula $$\mathrm{arcsinh}(x)=\log(x+\sqrt{x^2+1}),$$
which implies that
$$
C=\frac{2}{\sqrt{\sigma}}
\log(2\sqrt{\sigma}
+\sqrt{4\sigma+(1-\sigma)^2}).
$$
Finally, we take $R=A+C$.
\end{proof}

The rest of this section will be devoted to analyzing the operator
\begin{equation}\label{eq:Jn}
J_{n,\alpha}(\sigma)
=(I_{1,\alpha}-I_{n,\alpha})(\sigma)
=\frac{1}{\alpha}\int_{-\pi}^{\pi}\frac{\cos(\beta)-\cos(n\beta)}{|\sigma-e^{i\beta}|^{\alpha}}\dif\beta,
\end{equation}
which is now well defined at $\sigma=1$ for $\alpha<3$ thanks to the extra cancellation in the numerator.

\begin{lemma}\label{lemma:Jn}
Let $n\geq 2$. For $0\leq\alpha<3$, the function
$J_{n,\alpha}(\sigma)$ is continuous at $\sigma=1$. Moreover,
$$
J_{n,\alpha}(1)>0.
$$
For $3\leq\alpha<5$, it holds that
$$
\lim_{\sigma\to 1}J_{n,\alpha}(\sigma)=\infty.
$$
\end{lemma}
\begin{proof}
For $\alpha<3$, the integrand inside $J_{n,\alpha}$ is integrable at $\sigma=1$.
As we saw in Lemma \ref{lemma:In}, we can compute
$$
J_{n,\alpha}(1)
=D_\alpha\frac{1-F_n(\alpha)}{1-\alpha},
$$
which is indeed valid in the full range $0\leq\alpha<3$. On the one hand, $D_\alpha>0$.
On the other hand, for $0<\alpha<1$, we have 
$f_k(\alpha)<1$ for all $k=0,\ldots,n-1$, and thus
$$
F_n(\alpha)<1,
$$
while for $1<\alpha<3$, we have $f_k(\alpha)>1$ for all $k=0,\ldots,n-1$, and thus
$$
F_n(\alpha)>1.
$$
For $\alpha=1$, we have $f_k(1)=1$ for all $k=0,\ldots,n-1$, and thus
$$
F_n(1)=1.
$$
By applying the product rule, we deduce that
$$
F_n'(1)
=\sum_{k=1}^{n-1}f_k'(1)
=\sum_{k=1}^{n-1}\frac{4k+2}{(2k+2)^2}>0.
$$
Hence,
$$
J_{n,1}(1)
=D_1 F_n'(1)>0.
$$
For $3\leq\alpha<5$, by expanding the cosine functions in the integrands in Taylor series,
$$
J_{n,\alpha}(\sigma)
=\frac{n^2-1}{2\alpha}\int_{-\pi}^{\pi}
\frac{\beta^2}{|\sigma-e^{i\beta}|^\alpha}\dif\beta
+\frac{1}{\alpha}\int_{-\pi}^{\pi}
\frac{O(\beta^4)}{|\sigma-e^{i\beta}|^\alpha}\dif\beta,
$$
it is clear that the second term is finite, while the first one diverges to infinity as $\sigma\to 1$.
\end{proof}

\begin{lemma}\label{lemma:Jndiff1}
Let $n\in\N$. For $0\leq\alpha<2$, the function $J_{n,\alpha}(\sigma)$ is differentiable at $\sigma=1$. Moreover,
$$
J_{n,\alpha}'(1)
=-\frac{\alpha}{2}J_{n,\alpha}(1).
$$
For $0\leq\alpha<1$, the function $J_{n,\alpha}(\sigma)$ is twice differentiable at $\sigma=1$. Moreover,
$$
J_{n,\alpha}''(1)
=-\alpha J_{n,2+\alpha}(1)
+\frac{\alpha(2+\alpha)}{4}J_{n,\alpha}(1).
$$
For $1\leq\alpha<2$, it holds that
$$
\lim_{\sigma\to 1}J_{n,\alpha}''(\sigma)=-\infty.
$$
\end{lemma}
\begin{proof}
For any $0<|\sigma-1|\leq\frac{1}{2}$, we can differentiate under the integral sign
$$
J_{n,\alpha}'(\sigma)
=-\int_{-\pi}^{\pi}\frac{\sigma-\cos\beta}{|\sigma-e^{i\beta}|^{2+\alpha}}(\cos\beta-\cos(n\beta))\dif\beta.
$$
The integrand can be bounded by
$$
\left|\frac{\sigma-\cos\beta}{|\sigma-e^{i\beta}|^{2+\alpha}}(\cos\beta-\cos(n\beta))\right|
\lesssim
|\sigma-e^{i\beta}|^{1-\alpha}
\lesssim
\sin(\beta/2)^{1-\alpha},
$$
which is integrable for any $\alpha<2$. Thus, 
we can apply the dominated convergence theorem in the limit $\sigma\to 1$ to get
$$
J_{n,\alpha}'(1)
=-\int_{-\pi}^{\pi}\frac{1-\cos\beta}{|1-e^{i\beta}|^{2+\alpha}}(\cos\beta-\cos(n\beta))\dif\beta
=-\frac{\alpha}{2}J_{n,\alpha}(1),
$$
where in the last equality we have applied the identity
$
(1-\cos\beta)=\frac{1}{2}|1-e^{i\beta}|^2.
$
By differentiating under the integral sign again, we get
\begin{equation}\label{eq:Jndiff2}
J_{n,\alpha}''(\sigma)
=-\alpha J_{n,2+\alpha}(\sigma)
+(2+\alpha)\int_{-\pi}^{\pi}\frac{(\sigma-\cos\beta)^2}{|\sigma-e^{i\beta}|^{2+\alpha}}(\cos\beta-\cos(n\beta))\dif\beta.
\end{equation}
For $\alpha<1$, by applying Lemma \ref{lemma:Jn} and the dominated convergence theorem, we deduce that
$$
J_{n,\alpha}''(1)
=-\alpha J_{n,2+\alpha}(1)
+\frac{\alpha(2+\alpha)}{4}J_{n,\alpha}(1).
$$
For $1\leq\alpha<2$, the second integral in \eqref{eq:Jndiff2} is finite, while the first term $J_{n,2+\alpha}(\sigma)$ diverges to infinity as $\sigma\to 1$ by Lemma \ref{lemma:Jn}.
\end{proof}

\subsection{Ansatz}

Given some parameters $0<r_1<r_2<\infty$ and $c>0$ to be determined, we consider the piecewise constant profile 
\begin{equation}\label{ansatz:bartheta}
\bar{\theta}(r)
=
\left\lbrace
\begin{array}{rl}
	c, & 0<r\leq r_1, \\[0.1cm]
	-1, & r_1<r\leq r_2, \\[0.1cm]
	0, & r>r_2,
\end{array}
\right.
\end{equation}
where we have chosen $c$ satisfying
\begin{equation}\label{eq:thetazeromean}
(1+c)r_1^2=r_2^2.
\end{equation}
In this case $\bar{\theta}$ has zero mean and
$$
\partial_r\bar{\theta}
=-(1+c)\delta_{r_1} +\delta_{r_2}.
$$
For the sake of simplicity we fix
$$
r_1=\sigma,
\quad\quad
r_2=1,
$$
for some $\frac{1}{2}<\sigma<1$, to be determined. 
In the next section we will use again the letters $r_1$ and $r_2$ to make the notation more compact.
In any case, \eqref{eq:thetazeromean} reads as
\begin{equation}\label{eq:csigma}
(1+c)\sigma^2=1.
\end{equation}
Under our ansatz, it remains to solve the equation \eqref{eq:RSE:I} at the points $r=\sigma,1$. This imposes two conditions on the vector $h=(h(\sigma),h(1))$, 
represented by the following linear system:
$$A
h
=zh,$$
with
\begin{equation}\label{matrixA}
A=
\left[
\begin{array}{cc}
\sigma^{-(2+\alpha)}J_{n,\alpha}(1)-\sigma^{-1}I_{1,\alpha}(\sigma) & \sigma^{-1}I_{n,\alpha}(\sigma)  \\[0.1cm]
-\sigma^{-1} I_{n,\alpha}(\sigma) & -J_{n,\alpha}(1) +\sigma^{-1} I_{1,\alpha}(\sigma) 
\end{array}
\right],
\end{equation}
where we have applied \eqref{eq:csigma} and the identity
$$
I_{n,\alpha}(1/\sigma)=\sigma^\alpha I_{n,\alpha}(\sigma).$$ 
Recall that $I_{1,\alpha}(\sigma)$, $I_{n,\alpha}(\sigma)$ and $J_{n,\alpha}(1)$ are well defined for any $\alpha<2$.

In order to find an eigenvector  $h\neq 0$ we need to solve $\det (A-z)=0$, which is a quadratic equation in $z$. Since $A$ is real valued, one of its roots satisfies $\Im z>0$ if and only if its discriminant $\Delta$ is strictly negative. In the next lemma we compute $\Delta$ in terms of the parameter $\sigma$ and the functions $I_{n,\alpha}(\sigma)$ and $J_{n,\alpha}(1)$. 

\begin{lemma}\label{lemma:discriminant}
The discriminant $\Delta$ of the characteristic polynomial $\det(A-z)=0$ satisfies that 
$$
\sigma^2\Delta(\sigma)
=((\sigma+\sigma^{-(1+\alpha)})J_{n,\alpha}(1)-2I_{1,\alpha}(\sigma))^2
-(2I_{n,\alpha}(\sigma))^2.
$$
\end{lemma}
\begin{proof}
We compute
$$
\det (A-z)
=z^2-(\mathrm{tr}A)z+\det A,
$$
where
$$
\mathrm{tr}A
=(\sigma^{-(2+\alpha)}-1)J_{n,\alpha}(1),
$$
and
\begin{align*}
\det A
&=(\sigma^{-(2+\alpha)}J_{n,\alpha}(1)-\sigma^{-1}I_{1,\alpha}(\sigma))
(-J_{n,\alpha}(1) +\sigma^{-1} I_{1,\alpha}(\sigma))
+\sigma^{-2}I_{n,\alpha}(\sigma)^2\\
&=-\sigma^{-(2+\alpha)}
J_{n,\alpha}(1)^2
+(\sigma^{-(3+\alpha)}+\sigma^{-1})I_{1,\alpha}(\sigma)J_{n,\alpha}(1)
-\sigma^{-2}I_{1,\alpha}(\sigma)^2
+\sigma^{-2}I_{n,\alpha}(\sigma)^2.
\end{align*}
Therefore,
\begin{align*}
\Delta
&=(\mathrm{tr}A)^2-4\det A\\
&=(1+\sigma^{-(2+\alpha)})^2J_{n,\alpha}(1)^2
-4((\sigma^{-(3+\alpha)}+\sigma^{-1})I_{1,\alpha}(\sigma)J_{n,\alpha}(1)
-\sigma^{-2}(I_{1,\alpha}(\sigma)^2-I_{n,\alpha}(\sigma)^2))\\
&=((1+\sigma^{-(2+\alpha)})J_{n,\alpha}(1)-2\sigma^{-1}I_{1,\alpha}(\sigma))^2
-4\sigma^{-2}I_{n,\alpha}(\sigma)^2,
\end{align*}
as we wanted to prove.
\end{proof}

Notice that $\Delta=0$ for $n=1$. Thus, from now on we consider the case $n\geq 2$. 

\begin{prop}\label{prop:discriminant<0}
For every $n\geq 2$ and $0\leq\alpha<2$, there exists $\frac{1}{2}<\sigma<1$ such that $\Delta(\sigma)<0$. Therefore, there exist an eigenvalue $z\in\C$ with $\Im z>0$ and an eigenvector $h\neq 0$ solving $Az=hz$.
\end{prop}
\begin{proof}
Fixed $n$ and $\alpha$, we know from Lemma \ref{lemma:In} that there exist $\delta>0$ and $\frac{1}{2}<\sigma_0<1$ such that 
\begin{equation}\label{eq:In(1)pos}
I_{n,\alpha}(\sigma)>\delta,
\end{equation}
for all $\sigma_0\leq\sigma<1$. Thus, in these range of $\sigma$'s, it follows from Lemma \ref{lemma:discriminant},
that $\Delta(\sigma)<0$ if and only if
$$
|(\sigma+\sigma^{-(1+\alpha)})J_{n,\alpha}(1)-2I_{1,\alpha}(\sigma)|<2I_{n,\alpha}(\sigma),
$$
or, equivalently,
\begin{equation}\label{eq:Lnineq}
-4I_{n,\alpha}(\sigma)<L_{n,\alpha}(\sigma)<0,
\end{equation}
where we have abbreviated
$$
L_{n,\alpha}(\sigma)
=2J_{n,\alpha}(\sigma)-(\sigma+\sigma^{-(1+\alpha)})J_{n,\alpha}(1).
$$
By Lemma \ref{lemma:Jn}, we have
\begin{equation}\label{eq:Ln(1)=0}
L_{n,\alpha}(1)=0.
\end{equation}
In light of \eqref{eq:In(1)pos} and \eqref{eq:Ln(1)=0}, to conclude \eqref{eq:Lnineq} it is enough to check that $L_{n,\alpha}$ is strictly increasing on an interval $\sigma_1\leq\sigma<1$ for some $\sigma_1>\sigma_0$. We start by computing the first derivative
$$
L_{n,\alpha}'(\sigma)
=2J_{n,\alpha}'(\sigma)-(1-(1+\alpha)\sigma^{-(2+\alpha)})J_{n,\alpha}(1).
$$
By applying Lemma \ref{lemma:Jndiff1}, we deduce that
$$
L_{n,\alpha}'(1)
=2J_{n,\alpha}'(1)+\alpha J_{n,\alpha}(1)
=0.
$$
Next, we compute the second derivative
$$
L_{n,\alpha}''(\sigma)
=2J_{n,\alpha}''(\sigma)
-(1+\alpha)(2+\alpha)\sigma^{-(3+\alpha)}J_{n,\alpha}(1).
$$
Finally, we apply Lemmas \ref{lemma:Jn} and \ref{lemma:Jndiff1} again. For $0\leq\alpha<1$, we have
$$
L_{n,\alpha}''(1)
=2J_{n,\alpha}''(1)
-(1+\alpha)(2+\alpha)J_{n,\alpha}(1)
=-2\alpha J_{n,2+\alpha}(1)
-\frac{(2+\alpha)^2}{2}J_{n,\alpha}(1)<0.
$$
For $1\leq\alpha<2$, we have
$$
\lim_{\sigma\to 1}L_{n,\alpha}''(\sigma)=-\infty.
$$
This concludes the proof.
\end{proof}

\section{Regularization}\label{sec:regularization}

In this section we prove Theorem \ref{thm:L}. To this end, we show that there is a regularization $\bar{\theta}^\varepsilon$, for some small $\varepsilon>0$, of the piecewise constant vortex $\bar{\theta}$ that we constructed in the previous section, that is also unstable. 
As we explained in Section \ref{sec:sketch}, it is convenient to regularize $\bar{\theta}$ in such a way that $\bar{\theta}^\varepsilon$ also has zero mean. To this end, we take a mollifier $\eta\in C_c^\infty(I)$ with $I=(-1,1)$ satisfying the properties of the following lemma.

\begin{lemma}\label{lemma:mollifier}
There exists $\eta\in C_c^\infty(-1,1)$ such that
$$
\int_{-1}^{1}\eta(\rho)\dif\rho
=1,
\quad\quad
\int_{-1}^{1}\eta(\rho)\rho\dif\rho
=\int_{-1}^{1}\eta(\rho)\rho^2\dif\rho
=0.
$$
\end{lemma}
\begin{proof}
Let $\chi\in C_c^\infty(0,1)$ satisfying
$$
\int_0^1\chi\dif\rho=1.
$$
We consider the even function $\eta\in C_c^\infty(-1,1)$ defined on $(0,1)$ by 
$$
\eta(\rho)=\frac{1}{4}(3\chi(\rho)+\rho\partial_\rho\chi(\rho)).
$$
Notice that $\eta(0)=0$.
The mean of $\eta$ equals
$$
\int_{-1}^{1}\eta\dif\rho
=\frac{1}{2}\int_0^1(3\chi+\rho\partial_\rho\chi)\dif\rho=1.
$$
Since $\eta$ is even, we have
$$
\int_{-1}^1\eta\rho\dif\rho = 0.
$$
The second momentum of $\eta$ also vanishes
$$
\int_{-1}^{1}\eta\rho^2\dif\rho
=
\frac{1}{2}
\int_0^1\partial_\rho(\rho^3\chi)\dif\rho=0.
$$
This concludes the proof.
\end{proof}

Next, we use $\eta$ to regularize $\bar{\theta}$. We rewrite \eqref{ansatz:bartheta} as
$$
\bar{\theta}
=c-(1+c)1_{[r_1,\infty)}
+1_{[r_2,\infty)}.
$$
Let $0<\varepsilon<\frac{1}{3}\min\{r_1,r_2-r_1\}$, to be determined. This reformulation allows mollifying $\bar{\theta}$ without modifying its value at $r=0^+$, namely we define
\begin{equation}\label{eq:barthetaeps}
\bar{\theta}^\varepsilon
=c+(-(1+c)1_{[r_1,\infty)}
+1_{[r_2,\infty)})*\eta^\varepsilon,
\end{equation}
where $\eta^\varepsilon$ is the usual approximation of the identity
$$
\eta^\varepsilon(\rho)
=\frac{1}{\varepsilon}\eta\left(\frac{\rho}{\varepsilon}\right).
$$
By definition, $\bar{\theta}^\varepsilon$ is smooth and agrees with $\bar{\theta}$ outside $B_{\varepsilon}(\{r_1,r_2\})
=B_\varepsilon(r_1)\cup B_\varepsilon(r_2).$ Therefore, since
$$
\partial_r\bar{\theta}^\varepsilon=0
\quad\text{outside}\quad
B_\varepsilon(\{r_1,r_2\}), 
$$
it remains to find an eigenvalue $z^\varepsilon\in\C$ with $\Im z^\varepsilon>0$ and a profile $h^\varepsilon$ satisfying the linear stability equation \eqref{eq:RSE:I}
\begin{equation}\label{eq:RSE:eps}
\frac{1}{r}
\int_0^\infty
\left(I_{n,\alpha}\left(\frac{r}{s}\right)
h^\varepsilon(s)-I_{1,\alpha}\left(\frac{r}{s}\right)
h^\varepsilon(r)\right)\partial_r\bar{\theta}^\varepsilon(s)s^{1-\alpha}\dif s
=z^\varepsilon h^\varepsilon(r),
\quad\quad
r\in B_\varepsilon(\{r_1,r_2\}).
\end{equation}

\subsection{Rescaling}
In this section we zoom into each interval $B_\varepsilon(r_j)$ through the change of variables
$$
r=r_j+\varepsilon\rho,
$$
for $\rho\in I=(-1,1)$ and $j=1,2$. 
From now on we will denote 
$$
c_1=-(1+c),
\quad\quad
c_2=1,
$$
to make the notation more compact. In the following lemma, we use the conditions on the momenta of $\eta$.

\begin{lemma}[Rescaling of $\bar{\theta}^\varepsilon$]
It holds that
$$
\varepsilon\partial_r\bar{\theta}^\varepsilon(r)
=c_j\eta(\rho),
$$
for $r=r_j+\varepsilon\rho$. Moreover, $\bar{\theta}^\varepsilon$ has zero mean.
\end{lemma}
\begin{proof}
By differentiating \eqref{eq:barthetaeps}, we obtain the first claim,
$$
\varepsilon\partial_r\bar{\theta}^\varepsilon
=\varepsilon(c_1\delta_{r_1}
+c_2\delta_{r_1})*\eta^\varepsilon
=c_1
\eta\left(\frac{r-r_1}{\varepsilon}\right)
+c_2
\eta\left(\frac{r-r_2}{\varepsilon}\right).
$$
Hence, 
by applying Lemma \ref{lemma:mollifier}, we deduce that
$$
\int_0^\infty\bar{\theta}^\varepsilon r\dif r
=-\frac{1}{2}\int_0^\infty
\partial_r\bar{\theta}^\varepsilon r^2\dif r
=-\frac{1}{2}
\sum_{j=1,2}c_j
\int_{-1}^{1}
\eta(\rho)(r_j+\varepsilon\rho)^2\dif\rho
=-\frac{1}{2}
\sum_{j=1,2}c_j r_j^2=0,
$$
where the last equality follows from \eqref{eq:thetazeromean}. The second claim is also proved.
\end{proof}

We expand $h^\varepsilon$ in each interval $B_\varepsilon(r_j)$ as
$$
h^\varepsilon(r)
=h_j+\varepsilon g_j(\rho),
$$
for $r=r_j+\varepsilon\rho$, and for some profiles $g_j\in L^2(I)$ to be determined. Here $h=(h_1,h_2)$ is the eigenvector we found in Proposition \ref{prop:discriminant<0}. Typically, we will deal with $j=1,2$ and will speak of $g=(g_1,g_2) \in L^2(I)^2$ to deal with both equations simultaneously. Similarly, we expand
$$
z^\varepsilon
=z+\varepsilon y,
$$
for some $y\in\C$, to be determined. 
Here $z$ is the eigenvalue we found in Proposition \ref{prop:discriminant<0}.

For these $g\in L^2(I)^2$ and $y\in\C$, the linear stability equation \eqref{eq:RSE:eps} is rewritten as
\begin{equation}\label{eq:RSE:eps:1}
A^\varepsilon(h+\varepsilon g)
=(z+\varepsilon y)(h+\varepsilon g),
\end{equation}
where
$$
(A^\varepsilon f)_j(\rho)
=\sum_{k=1,2}
c_k\int_{-1}^{1}
\frac{(r_k+\varepsilon\varrho)^{1-\alpha}}{r_j+\varepsilon\rho}
\left(
I_{n,\alpha}\left(\frac{r_j+\varepsilon\rho}{r_k+\varepsilon\varrho}\right)f_k(\varrho)
-
I_{1,\alpha}\left(\frac{r_j+\varepsilon\rho}{r_k+\varepsilon\varrho}\right)f_j(\rho)\right)
\eta(\varrho)\dif\varrho.
$$
Recall that we have used the change of variables $r=r_j+\varepsilon\rho$ and $s=r_k+\varepsilon\varrho$.

\subsection{Expansion of $A^\varepsilon$}
We decompose
$$
A^\varepsilon=A^\varepsilon_0+A^\varepsilon_1,
$$
into
\begin{equation}\label{eq:A0eps1}
\begin{split}
(A_0^\varepsilon f)_j(\rho)
=&c_k\int_{-1}^{1}
\frac{(r_k+\varepsilon\varrho)^{1-\alpha}}{r_j+\varepsilon\rho}
\left(
I_{n,\alpha}\left(\frac{r_j+\varepsilon\rho}{r_k+\varepsilon\varrho}\right)f_k(\varrho)
-
I_{1,\alpha}\left(\frac{r_j+\varepsilon\rho}{r_k+\varepsilon\varrho}\right)f_j(\rho)\right)
\eta(\varrho)\dif\varrho\\
&-c_j\int_{-1}^{1}
\frac{(r_j+\varepsilon\varrho)^{1-\alpha}}{r_j+\varepsilon\rho}
J_{n,\alpha}\left(\frac{r_j+\varepsilon\rho}{r_j+\varepsilon\varrho}\right)f_j(\varrho)
\eta(\varrho)\dif\varrho
\end{split}
\end{equation}
with $k\neq j$, 
and
\begin{equation}\label{eq:A1eps}
(A_1^\varepsilon f)_j(\rho)
=c_j\int_{-1}^{1}
\frac{(r_j+\varepsilon\varrho)^{1-\alpha}}{r_j+\varepsilon\rho}
I_{1,\alpha}\left(\frac{r_j+\varepsilon\rho}{r_j+\varepsilon\varrho}\right)
(f_j(\varrho)
-
f_j(\rho))
\eta(\varrho)\dif\varrho.
\end{equation}
Notice that 
\begin{equation}\label{eq:A1epsnull}
A_1^\varepsilon\mu=0,
\end{equation}
for any constant vector $\mu\in\C^2$.
The operator $A_0^\varepsilon$ can be treated uniformly in $\alpha$ as done in the following lemma. As a corollary, we obtain a simplification of the linear stability equation.

\begin{lemma}[Expansion of $A_0^\varepsilon$]\label{lemma:A0eps}
Let $0\leq \alpha\leq 1$. It holds that
$$
A_0^\varepsilon
=A_0+\varepsilon B_0,
$$
where, for $k\neq j$,
\begin{equation}\label{eq:A0eps}
\begin{split}
(A_0 f)_j(\rho)
=&c_k
\frac{r_k^{1-\alpha}}{r_j}\int_{-1}^{1}
\left(
I_{n,\alpha}\left(\frac{r_j}{r_k}\right)f_k(\varrho)
-
I_{1,\alpha}\left(\frac{r_j}{r_k}\right)f_j(\rho)\right)
\eta(\varrho)\dif\varrho\\
&-c_j
\frac{r_j^{1-\alpha}}{r_j}
J_{n,\alpha}(1)\int_{-1}^{1}
f_j(\varrho)
\eta(\varrho)\dif\varrho,
\end{split}
\end{equation}
and $B_0=B_0^\varepsilon:L^2(I)^2\to L^2(I)^2$ is uniformly bounded in $\varepsilon$.
Moreover, for any $\mu\in\C^2$,
$$
A_0\mu=A\mu,
$$
where $A$ is the matrix in \eqref{matrixA}.
\end{lemma}
\begin{proof}
Recall that $k\neq j$ and that the kernels $I_{n,\alpha}(\sigma)$ are smooth away from $\sigma=1$, while $J_{n,\alpha}$ is $C^1$ by Lemma \ref{lemma:Jndiff1}.
Thus, all the functions in \eqref{eq:A0eps1} multiplying $f$ are $C^1$ in $\varepsilon$ and the first statement follows. The second statement is a direct consequence of the definition of $A_0$ and that of $A$ in \eqref{matrixA}, together with  $\int\eta=1$.
\end{proof}

\begin{cor}
For $0\leq\alpha\leq 1$,
the linear stability equation \eqref{eq:RSE:eps:1} can be rewritten as
\begin{equation}\label{eq:RSE:eps:2}
(A^\varepsilon-z)g
=(y-B_0)h+ \varepsilon yg.
\end{equation}
\end{cor}
\begin{proof}
Since $h\in\C^2$, by applying \eqref{eq:A1epsnull} and Lemma \ref{lemma:A0eps}, we deduce that
$$
A^\varepsilon h
=A_0^\varepsilon h
=A_0h+\varepsilon B_0 h.
$$
On the other hand,
$$
(z+\varepsilon y)(h+\varepsilon g)
=zh+\varepsilon(zg+yh)+\varepsilon^2 yg.
$$
Finally, we apply that $A_0h=Ah=zh$, thus canceling the zero-order term.
\end{proof}

The operator $A_1^\varepsilon$ is however more delicate, and needs to be analyzed separately for $\alpha<1$ or $\alpha=1$.

\subsection{Case $0\leq\alpha<1$}

\begin{lemma}[Expansion of $A_1^\varepsilon$]\label{lemma:A1eps:0<a<1}
Let $0\leq\alpha<1$. It holds that
$$
A_1^\varepsilon
=A_1+\frac{\varepsilon^{1-\alpha}}{1-\alpha} B_1,
$$
where
$$
(A_1 f)_j(\rho)
=\frac{c_j}{r_j^\alpha}
I_{1,\alpha}(1)
\left(
\int_{-1}^{1}
f_j(\varrho)\eta(\varrho)\dif\varrho
-
f_j(\rho)
\right),
$$
and $B_1=B_1^\varepsilon:L^2(I)^2\to L^2(I)^2$ is uniformly bounded in $\varepsilon$.
\end{lemma}
\begin{proof}
It follows by expanding around $\varepsilon=0$ the functions in \eqref{eq:A1eps} multiplying $f$.
Firstly, notice that the factor
$$
\frac{(r_j+\varepsilon\varrho)^{1-\alpha}}{r_j+\varepsilon\rho},
$$
is smooth in $\varepsilon$. Secondly, for the kernel $I_{1,\alpha}$, we apply Lemma \ref{lemma:estimatesIn1} for
$$\sigma=\frac{r_j+\varepsilon\rho}{r_j+\varepsilon\varrho}
\quad\Rightarrow\quad
1-\sigma=\frac{\varepsilon(\varrho-\rho)}{r_j+\varepsilon\varrho}.
$$
More precisely, we have
$$
I_{1,\alpha}\left(\frac{r_j+\varepsilon\rho}{r_j+\varepsilon\varrho}\right)
=I_{1,\alpha}(1)+O\left(\frac{(\varepsilon|\rho-\varrho|)^{1-\alpha}}{1-\alpha}\right),
$$
which implies that $B_1$ is bounded.
\end{proof}

\begin{cor} 
For $0\leq\alpha<1$, 
the linear stability equation \eqref{eq:RSE:eps:2} can be rewritten as
\begin{equation}\label{eq:RSE:eps:0<a<1}
(A-z)g
=(y-B_0)h+ \left(\varepsilon y-\frac{\varepsilon^{1-\alpha}}{1-\alpha}B\right)g,
\end{equation}
where
$$
A=A_0+A_1,
\quad\quad
B=(1-\alpha)\varepsilon^\alpha B_0+B_1.
$$
\end{cor}
\begin{proof}
It follows by applying Lemmas \ref{lemma:A0eps} and \ref{lemma:A1eps:0<a<1}.
\end{proof}

\subsubsection{Inverting the linear operator}
Notice that we can split 
$$A=A_0+A_1=D+C,$$ 
where
$$
(Df)_j(\rho)=-f_j(\rho)\sum_{k=1,2}
c_k\frac{r_k^{1-\alpha}}{r_j}
I_{1,\alpha}\left(\frac{r_j}{r_k}\right),
\quad\quad
(Cf)_j=\sum_{k=1,2}
c_k\frac{r_k^{1-\alpha}}{r_j}
I_{n,\alpha}\left(\frac{r_j}{r_k}\right)\int_{-1}^{1}f_k(\varrho)\eta(\varrho)\dif\varrho
.
$$
Notice that $D$ is a diagonal operator, and therefore $(D-z)$ is invertible, and that $C$ is constant valued, that is, $C:L^2(I)^2\to\C^2$. The next lemma takes advantage of this structure to reformulate the linear stability equation in order to find the corresponding eigenvalue with positive imaginary part. The proof is analogous to \cite[Lemma 4.4]{CFMSpp}. Since it is brief, we recall it here for the sake of completeness.

\begin{lemma}\label{lemma:f:0<a<1}
For $0\leq\alpha<1$,
the linear stability equation
\eqref{eq:RSE:eps:0<a<1}
is equivalent to
\begin{equation}\label{eq:gy}
\begin{split}
g&=f
+\gamma h^* + \delta h,\\
y&=\frac{\langle Cf,h^{*\perp}\rangle}{\langle h,h^{*\perp}\rangle},
\end{split}
\end{equation}
where $f=f(g,y)$
\begin{equation}\label{eq:f}
f=(D-z)^{-1}\left(-B_0 h+\left(\varepsilon y-\frac{\varepsilon^{1-\alpha}}{1-\alpha}B\right)g\right),
\end{equation}
with $\gamma=\gamma(f)$
$$
\gamma
=-\frac{1}{2iz_2|h|^2}\left(\frac{\langle h,h^*\rangle}{\langle h,h^{*\perp}\rangle}\langle Cf,h^{*\perp}\rangle
-\langle Cf,h^*\rangle\right),
$$
and $\delta\in\C$ is arbitrary.
\end{lemma}
\begin{proof}
Let $f$ be given as in \eqref{eq:f}. We write the unknown $g$ as
$$
g=f
+\mu,
$$
in terms of some $\mu$, to be determined. 
Then, $g$ solves \eqref{eq:RSE:eps:0<a<1} if and only if $\mu$ solves
\begin{equation}\label{eq:mu}
(A-z)\mu
=yh-Cf.
\end{equation}
Since $A=D+C$ with $D$ diagonal and $C$ constant valued, we deduce that $\mu$ must be a constant vector
Therefore, $A$ is acting on constant vectors in \eqref{eq:mu}, and it is thus given by \eqref{matrixA}.
Since $\mathrm{Ker}(A-z)=\mathrm{span}(h)$, $\mathrm{Im}(A-z)=\mathrm{span}(h^*)$, and $\{h,h^*\}$ is a basis of $\C^2$,
we take
$$
\mu=\gamma h^*+\delta h,
$$
in terms of some $\gamma,\delta\in\C$, to be determined. Notice that
\begin{equation}\label{eq:Amenoszrange}
(A-z)\mu
=\gamma(A-z)h^*
=\gamma(A-(z^*+2iz_2))h^*
=-2i z_2\gamma h^*.
\end{equation}
Hence, 
multiplying \eqref{eq:mu} by $h^{*\perp}=(-h_2^*,h_1^*)$, we get the compatibility condition for $y$
$$
0
=y\langle h,h^{*\perp}\rangle
-\langle Cf,h^{*\perp}\rangle.
$$
Notice that $\langle h,h^{*\perp}\rangle\neq 0$ since $\{h,h^*\}$ is a basis of $\C^2$. On the other hand, multiplying \eqref{eq:mu} by $h^*$, and using \eqref{eq:Amenoszrange},
we can determine $\gamma$ from the following equality
$$
-2iz_2\gamma
|h|^2
=y\langle h,h^*\rangle
-\langle Cf,h^*\rangle.
$$
The proof of the lemma is concluded. 
\end{proof}

Since $\delta$ is arbitrary, we will take $\delta=0$ for simplicity. 

\subsubsection{Fixed point argument}

For $\varepsilon=0$ the equation \eqref{eq:gy} is explicit for $(g,y)$ and thus we can find a solution $(g^0,y^0)$. For small $\varepsilon>0$ we will apply a fixed point argument. The proof is again analogous to \cite[Proposition 4.5]{CFMSpp}. 

\begin{prop}\label{prop:fixedpoint}
Let $0\leq\alpha<1$.
For every $M>\|(g^0,y^0)\|_{L^2(I^2)\times\C}$ there exists $\varepsilon_0>0$ satisfying that: for every $0\leq\varepsilon\leq\varepsilon_0$ there exists $(g^\varepsilon,y^\varepsilon)\in L^2(I)^2\times\C$ solving \eqref{eq:gy} with 
$\|(g^\varepsilon,y^\varepsilon)\|_{L^2(I)^2\times\C}\leq M$. In particular, $\Im z^\varepsilon>0$ for $\varepsilon<\frac{\Im z}{M}$.
\end{prop}
\begin{proof}
We fix $0\leq\alpha<1$.
Let us denote by $F$ the map on the right-hand side of \eqref{eq:gy}, that is, we rewrite this equation compactly as
$$
(g,y)=F(\varepsilon,g,y).
$$ 
For $\varepsilon=0$, the functional $F$ is constant, namely, $F(0,g,y)=(g^0,y^0)$.
Since all the operators involved in $F$ are bounded, depend continuously on $\varepsilon$ at zero, and $\|(g^0,y^0)\|_{L^2(I)^2\times\C}<M$, there exists $\varepsilon_0>0$ such that $F$ maps the ball $B_M$ of $L^2(I)^2\times\C$ into itself.
Moreover, we can choose $\varepsilon_0>0$ such that $F$ becomes a contraction on $B_M$. Then, we can apply the classical Banach fixed point theorem to find our required solution. 
\end{proof}

\subsection{Case $\alpha=1$}\label{sec:regularization:a=1}

\begin{lemma}[Expansion of $A_1^\varepsilon$]\label{lemma:A1:a=1}
For $\alpha=1$,
$$
A_1^\varepsilon
=(\log\varepsilon) A_1+ B_1,
$$
where
$$
(A_1f)_j(\rho)=\frac{2c_j}{r_j}\left(f_j(\rho)-\int_{-1}^{1}f_j(\varrho)\eta(\varrho)\dif\rho\right),
$$
and $B_1=B_1^\varepsilon:L^2(I)^2\to L^2(I)^2$ is uniformly bounded in $\varepsilon$.
\end{lemma}
\begin{proof}
By Lemma \ref{lemma:estimatesIn2}, we have
$$
I_{1,\alpha}\left(\frac{r_j+\varepsilon\rho}{r_j+\varepsilon\varrho}\right)
=-2\log(\varepsilon|\rho-\varrho|)+O(1).
$$
Since the convolution operator associated with $\log|\rho-\varrho|$ is bounded in $L^2$, the uniform bound for $B_1$ follows. 
\end{proof}

Due to the $(\log\varepsilon)$ scaling, it is natural to decompose
\begin{equation}\label{eq:g:a=1}
g=\mu+\frac{f}{\log\varepsilon},
\end{equation}
for some $\mu\in\C^2$ and $f\in L^2(I)^2$, to be determined.

\begin{cor}\label{cor:RSE:a=1}
For $\alpha=1$, 
the linear stability equation \eqref{eq:RSE:eps:1} can be rewritten in terms of the ansatz \eqref{eq:g:a=1} as
\begin{equation}\label{eq:RSE:a=1}
A_1f
=(y-B_0)h-(A_0-z)\mu
+\varepsilon(y-B_0)\mu
+\frac{1}{\log\varepsilon}((z-B_1)+\varepsilon y)f.
\end{equation}
\end{cor}
From the definition of $A_1$, a necessary condition for solving \eqref{eq:RSE:a=1} is that the right-hand side multiplied by $\eta$ has zero mean. 

\subsubsection{Inverting the linear operator}
Notice that we can split
$
A_1=D+C,
$
where
$$
(Df)_j(\rho)=\frac{8c_j}{r_j^\alpha}f_j(\rho),
\quad\quad
(Cf)_j=-\frac{8c_j}{r_j^\alpha}\int_{-1}^{1}f_j\eta\dif\rho.
$$
As before, $D$ is an invertible diagonal operator, and $C$ is constant valued. 

\begin{lemma}\label{lemma:inverting:a=1}
For $\alpha=1$, the linear stability equation
\eqref{eq:RSE:a=1}
is equivalent to 
\begin{equation}\label{eq:gy:SQG}
\begin{split}
f&=\bar{f}+\frac{D^{-1}}{\log\varepsilon}\left(((z-B_1)+\varepsilon y)f-\int_{-1}^{1}((z-B_1)+\varepsilon y)f\eta\dif\rho\right),\\
y&=\frac{\langle F,h^{*\perp}\rangle}{\langle h,h^{*\perp}\rangle},\\
\gamma
&=\frac{\langle F-yh,h^*\rangle}{2iz_2|h|^2},
\end{split}
\end{equation}
where
$F=F(f,y,\gamma)$ is given by
$$
F=B_0h
-\varepsilon(y-B_0)\mu
-\frac{1}{\log\varepsilon}\int_{-1}^{1}((z-B_1)+\varepsilon y)f\eta\dif\rho,
$$
$\mu=\gamma h^*+\delta h$, 
and $\delta,\bar{f}_j\in\C$ are arbitrary.
Moreover, $\int f\eta=\bar{f}$.
\end{lemma}
\begin{proof}
We need to solve \eqref{eq:RSE:a=1}, that is,
$$
(y-B_0)h-(A_0-z)\mu
+\varepsilon(y-B_0)\mu
+\frac{1}{\log\varepsilon}\int_{-1}^{1}((z-B_1)+\varepsilon y)f\eta\dif\rho
=0.
$$
Similarly to Lemma \ref{lemma:f:0<a<1}, we can write $\mu$ in the basis $\{h,h^*\}$
$$
\mu=\gamma h^* + \delta h,
$$
and thus 
$$
(A_0-z)\mu=-2iz_2\gamma h^*.
$$
Therefore, we need to solve
$$
yh+2iz_2\gamma h^*
=B_0h
-\varepsilon(y-B_0)(\gamma h^* + \delta h)
-\frac{1}{\log\varepsilon}\int_{-1}^{1}((z-B_1)+\varepsilon y)f\eta\dif\rho=F.
$$
This equation determines $\gamma$ and $y$ as in the statement of the lemma.
\end{proof}

\subsubsection{Fixed point argument}

Similarly to Proposition \ref{prop:fixedpoint}, we can apply a fixed-point argument to the SQG case. The proof is analogous, and we therefore omit the details. We denote the solution to \eqref{eq:gy:SQG} when $\varepsilon=0$ by $(f^0,y^0,\gamma^0)$.

\begin{prop}\label{prop:fixedpoint:SQG}
For every $M>\|(f^0,y^0,\gamma^0)\|_{L^2(I^2)\times\C}$ there exists $\varepsilon_0>0$ satisfying that: for every $0\leq\varepsilon\leq\varepsilon_0$ there exists $(f^\varepsilon,y^\varepsilon,\gamma^\varepsilon)\in L^2(I)^2\times\C$ solving \eqref{eq:gy:SQG} with 
$\|(f^\varepsilon,y^\varepsilon,\gamma^\varepsilon)\|_{L^2(I)^2\times\C}\leq M$. In particular, $\Im z^\varepsilon>0$ for $\varepsilon<\frac{\Im z}{M}$.
\end{prop}

Once $0<\varepsilon\leq\varepsilon_0$ is fixed, we take the vortex
$$
\bar{\Theta}=\bar{\theta}^\varepsilon.
$$
From now on the smooth vortex $\bar{\Theta}$ is fixed and thus we will omit $\varepsilon$ for the sake of simplicity. 

\subsection{Properties of the vortex}

We conclude this section by deriving some properties of the vortex $\bar{\Theta}$, the corresponding velocity field $\bar{V}(X)
=\bar{V}_\phi(R)e_\phi$, and the eigenfunction $W\in U_n$, that will be useful in the next sections. The next proposition will be applied in the proof of Proposition \ref{prop:eigenfunction}.

\begin{prop}\label{prop:asymptoticbarV}
There exists a constant $C$ and a smooth function $G$ such that
$$
\frac{\bar{V}_\phi(R)}{R}
= C+RG(R),
$$
for $0<R\leq\frac{r_1}{2}$. 
In particular,
$$\partial_R(C\log R+U(R))=\frac{\bar{V}_\phi(R)}{R^2},$$
where
$$
U(R)=\int_0^R G\dif S.
$$
\end{prop}
\begin{proof}
By Lemma \ref{lemma:BiotSavartbarV}, we have
\begin{equation}\label{eq:vortexK1}
\frac{\bar{V}_\phi(R)}{R}
=-C_\alpha
\int_0^\infty
K_{1,\alpha}\left(\frac{R}{S}\right)
\partial_R\bar{\Theta}(S)S^{-\alpha}\dif S,
\end{equation}
where we have applied Remark \ref{Rem:Iforalpha=0}, namely,
$$
\frac{1}{\sigma}I_{1,\alpha}(\sigma)
=\int_{-\pi}^{\pi}\frac{(\sin\beta)^2}{|\sigma-e^{i\beta}|^{2+\alpha}}\dif\beta
=K_{1,\alpha}(\sigma).
$$
Since $\partial_R\bar{\Theta}(S)=0$ outside $r_1-\varepsilon\leq S\leq r_2+\varepsilon$, we deduce that
$$
\lim_{R\to 0}\frac{\bar{V}_\phi(R)}{R}
=-C_\alpha
K_{1,\alpha}(0)
\int_0^\infty
\partial_R\bar{\Theta}(S)S^{-\alpha}\dif S=C,
$$
where $K_{1,\alpha}(0)=\pi$. Thus, the fundamental theorem of calculus implies that 
$$
\frac{\bar{V}_\phi(R)}{R}
-C
=-RC_\alpha\int_0^\infty
\int_0^1
K_{1,\alpha}'\left(\lambda\frac{R}{S}\right)\dif\lambda
\partial_R\bar{\Theta}(S)S^{-(1+\alpha)}\dif S
=RG(R).
$$
Finally, $G$ is smooth since $K_{1,\alpha}(\sigma)$ is smooth away from $\sigma=1$.
\end{proof}

The next lemma will be applied in Section \ref{sec:Hrad}.

\begin{lemma}\label{VbarDecay}
For every $i\in\N$,
$$
R^i\partial^i_R\left(\frac{\bar{V}_\phi}{R}\right)\in L^\infty.
$$
\end{lemma}
\begin{proof}
By construction, $\bar{\Theta}\in C_c^\infty$ and has zero mean. In particular, $\bar{\Theta}\in H^m$, thus $\bar{V}\in H^{m+1-\alpha}$ for all $ m\in\N$. Therefore, $\bar{V}\in C^\infty$. Hence, the statement holds in the annulus $\frac{1}{2}(r_1-\varepsilon)\leq R\leq \frac{3}{2}(r_2+\varepsilon)$. 
It remains to analyze the behavior near the origin and at infinity.
By changing variables $R=ST$ in \eqref{eq:vortexK1} we deduce that
$$
\frac{\bar{V}_\phi(R)}{R}
=-C_\alpha
\int_0^\infty
K_{1,\alpha}(T)\partial_R\bar{\Theta}\left(\frac{R}{T}\right)\left(\frac{R}{T}\right)^{1-\alpha}\frac{\dif T}{T}.
$$
The integrand vanishes outside 
$\frac{R}{r_2+\varepsilon}\leq T\leq\frac{R}{r_1-\varepsilon}$, and in $|T-1|\leq\frac{1}{2}$ outside the annulus $\frac{1}{2}(r_1-\varepsilon)\leq R\leq \frac{3}{2}(r_2+\varepsilon)$.
Thus, for any $R>0$ we can differentiate under the integral sign. By applying the Leibniz rule, we get
\begin{align*}
R^i\partial^i_R\left(\frac{\bar{V}_\phi}{R}\right)
&=\sum_{0\leq j\leq i}
\frac{i!}{j!}\binom{1-\alpha}{i-j}
\int_0^\infty K_{1,\alpha}(T)\partial_R^{j+1}\bar{\Theta}\left(\frac{R}{T}\right)\left(\frac{R}{T}\right)^{1-\alpha+j}
\frac{\dif T}{T}.
\end{align*}
Notice that all the  integrals in the above sum can be expressed as
$$
\int_{r_1-\varepsilon}^{r_2+\varepsilon}
K_{1,\alpha}\left(\frac{R}{S}\right)\partial_R^{j+1}\bar{\Theta}(S)S^{j-\alpha}\dif S
\dif S.
$$
The proof is concluded since $K_{1,\alpha}(\sigma)$ is bounded away from $\sigma=1$.
\end{proof}

The next proposition will be applied in the proof of Theorem \ref{thm:non-linear}.
Recall that the eigenfunction $W=W_n e^{in\phi}$ is given by $W_n=h\partial_R\bar{\Theta}$, where $h\in L^2$ is the profile we found in the previous section. 

\begin{prop}\label{prop:eigenfunction:b=0}
The eigenfunction $W\in U_n$ satisfies $W\in C_c^\infty$ with $\mathrm{supp}(W_n)\subset[r_1-\varepsilon,r_2+\varepsilon]$.
\end{prop}
\begin{proof}
Notice that $W_n=h\partial_r\bar{\Theta}=0$ outside the interval $B_\varepsilon([r_1,r_2])$. Let us check that it is smooth inside the interval $I_\varepsilon=B_{2\varepsilon}([r_1,r_2])$.
By construction, $W_n$ satisfies the eigenvalue equation \eqref{eq:SQG:selfsimilarstability} for $b=0$
\begin{equation}\label{eq:RSE:eigenfunction:b=0}
\left(\frac{\bar{V}_\phi(R)}{C_\alpha R}- z\right)W_{n}(R) +\partial_R\bar{\Theta}(R)\int_0^\infty I_{n,\alpha}(S)
W_{n}\left(\frac{R}{S}\right)\frac{\dif S}{S^{1+\alpha}}
=0,
\end{equation}
where $\lambda=-inC_\alpha z$.
Notice that the last integrand in \eqref{eq:RSE:eigenfunction:b=0} is supported on $\frac{R}{r_2+\varepsilon}\leq S\leq\frac{R}{r_1-\varepsilon}$. 
For $\alpha=1$, the expansion provided in Lemma \ref{lemma:estimatesIn2} implies that the integral operator in \eqref{eq:RSE:eigenfunction:b=0} gains a derivative in $L^2$. Since $\Im z>0$ we can invert the diagonal operator multiplying $W_n$ to conclude that $W_n \in H^1$. Using a bootstrapping argument, we further deduce that $W_n \in H^k$ for any $k \geq 1$. The cases $0 \leq \alpha < 1$ are even more favorable, as the kernel $I_{n,\alpha}$ becomes more regular as $\alpha$ decreases as stated in Lemma \ref{lemma:estimatesIn1}.
\end{proof}

\section{Self-similar instability}\label{sec:selfsimilarinstability}

In this section we prove Theorem \ref{thm:Lb}, and provide growth bounds of the semigroup generated by $L_b$. From now on we consider $n\geq 2$ fixed.

\subsection{The linear operators}

We split the linearization $L_b=L_{b,\alpha,\bar{\Theta}}$ of the self-similar $\alpha$-SQG equation around $\bar{\Theta}$ into
$$
L_b
=
(a-\alpha)b
+T_b+K,
$$
where $T_b$ is the transport operator
$$
T_b\Theta
=-\bar{V}_b\cdot\nabla\Theta
\quad\text{with}\quad
\bar{V}_b
=\bar{V}-b X,
$$
and $K$ is the operator
$$
K\Theta=-V\cdot\nabla\bar{\Theta}
\quad\text{with}\quad
V=-\nabla^\perp\Lambda^{\alpha-2}\Theta.
$$
Here, $\bar{\Theta}\in C_c^\infty$ is the vortex from Theorem \ref{thm:L}, and $\bar{V}$ is the corresponding velocity field.
The domain of $K$ is $D(K)=L_n^2$, and the domains of $L_b$ and $T_b$ are
$$
D(L_b)=D(T_b)
=\{\Theta\in L_n^2\,:\,\mathrm{div}(\bar{V}_b\Theta)\in L_n^2\}.
$$
Thus, the operators under consideration are closed and densely defined in $L_n^2$. We recall that in Section \ref{sec:stabilityequation}, we verified that $U_{jn}$ is invariant under $T_b$ and $K$ for every $j\in\Z$. 
The reason for considering the direct sum $$L_n^2=\bigoplus_{j\in\Z}U_{jn}$$ is that the quadratic term in the $\alpha$-SQG equation lacks invariance in $U_{jn}$, but it is invariant in $L_n^2$. The invariance in $L_n^2$ follows from the following identity (see \eqref{eq:nfold} and \eqref{eq:Vnfold})
$$
\mathrm{div}(V\Theta)(X)
=\mathrm{div}\left(e^{-\frac{2\pi i}{n}}(V\Theta)(e^{\frac{2\pi i}{n}} X)\right)
=\mathrm{div}(V\Theta)(e^{\frac{2\pi i}{n}} X).
$$
Specifically, in Section \ref{sec:non-linearinstability} we will need the growth bound of the semigroup generated by $L_b$
acting on the full space $L_n^2$.

As we mentioned in the introduction, one of the key points in Vishik's spectral analysis \cite{Vishikpp1,Vishikpp2} (as well as in \cite{ABCDGMKpp,CFMSpp}) for the 2D Euler equation ($\alpha=0$) is the fact that $K$ is compact. This still holds for $0\le \alpha<1$, but no longer for $\alpha=1$
(see Section \ref{sec:K}). 
In Section \ref{sec:selfsimilar:SQG}, we show how to overcome this difficulty by appropriately decomposing the operator $K$.

\subsection{Spectral analysis}\label{sec:Spectralanalysis}

In this section, we present the part of the analysis that is common to both cases: $0 \leq \alpha < 1$ and $\alpha = 1$.
For this purpose, we recall some classical results in Operator theory that will be useful during the analysis. We consider linear operators $A:D(A)\subset H\to H$ acting on some Hilbert space $H$, where $D(A)$ is the domain of $A$. For a fixed $H$, we denote by $\mathcal{L}$ and $\mathcal{K}$ the space of bounded and compact operators, respectively. In the next sections we will consider $H=L_n^2$.

Firstly, we recall the stability of strongly continuous semigroups under bounded perturbations, which can be found in 
\cite[Chapter III, Bounded Perturbation Theorem]{EngellNagel00}.

\begin{prop}\label{prop:stabilitySCS}
Let $A$ be a linear operator on $H$ generating a strongly continuous semigroup, and $B\in\mathcal{L}.$ Then, $A+B$ generates a strongly continuous semigroup.
\end{prop}

For the proof of Theorem \ref{thm:Lb}, we will decompose the operator $L_b$ into an operator $A_b$ that satisfies a series of properties detailed in the proposition below, and a compact operator $C$. The proof of this proposition can be found in Appendix \ref{sec:Operatortheory} and follows the same argument as in \cite{ABCDGMKpp,CFMSpp}. We recall it here for the convenience of the reader. We remark that although this proposition could be further generalized, we choose to present the version that is most useful for our purposes.
See Appendix \ref{sec:Operatortheory} and \cite{EngellNagel00} for more details on operator theory and definitions. In particular, see Definition \ref{defi:contraction} for the notion of contraction semigroup and \eqref{eq:growthbound} for the definition of the growth bound $\omega_0$.

\begin{prop}\label{prop:SpectralAnalysis}
Assume that the following conditions hold:
\begin{enumerate}
    \item\label{prop:SpectralAnalysis:1} Let $(A_b)$ be a family of linear operators on $H$ generating contraction semigroups. Suppose that for any fixed $\tau\geq 0$ and $W\in H$, the map 
    \begin{equation}\label{eq:mapb}
    b\mapsto e^{\tau A_b}W
    \end{equation}
    is continuous from $[0,\infty)$ to $H$.
    \item\label{prop:SpectralAnalysis:2} Let $C$ be a compact operator on $H$.
    \item\label{prop:SpectralAnalysis:3} Let $L_b=A_b+C$. Suppose that there exists $\lambda_0$ with $\Re\lambda_0>0$ and $W_0\in D(L_0)$ such that
    $$
    L_0W_0=\lambda_0 W_0.
    $$
\end{enumerate}
Then, for every $b_0>0$, there exist $\lambda_b=\omega_0(L_b)$ with $\Re\lambda_b>\frac{\Re\lambda_0}{2}$ and $W_b\in D(L_b)$ for some $0<b\leq b_0$ such that
$$
L_bW_b=\lambda_b W_b.
$$
Furthermore, for every $\delta>0$ there exists $C_\delta\geq 1$ such that
$$
\|e^{\tau L_b}\|_{\mathcal{L}}
\leq C_\delta e^{(\Re\lambda+\delta)\tau},
$$
for all $\tau\geq 0$.
\end{prop}
\begin{proof}
See Appendix \ref{sec:prop:SpectralAnalysis}.
\end{proof}

Therefore, to prove that the vortex $\bar{\Theta}$ is also self-similarly unstable, it will suffice to verify that the conditions \eqref{prop:SpectralAnalysis:1}-\eqref{prop:SpectralAnalysis:3} are satisfied for a suitable decomposition of $L_{b,\alpha,\bar{\Theta}}$.

\subsection{The transport operator $T_b$}

In this section, we show that the operator $T_b$ satisfies the condition \eqref{prop:SpectralAnalysis:1} in Proposition \ref{prop:SpectralAnalysis}. Although the proof is the same as in \cite{CFMSpp}, we recall it here for its brevity and for the convenience of the reader.

\begin{lemma}\label{lemma:Tbsemigroup}
The operator $T_b$ generates a contraction
semigroup $\{e^{\tau T_b}\}_{\tau\geq 0}$ with
\begin{equation}\label{eq:esTbbound}
\|e^{\tau T_b}\|_{\mathcal{L}}
	=e^{-b\tau},
\end{equation}
for all $\tau\geq 0$.
Furthermore, for any $\tau\geq 0$ and $\Theta\in L_n^2$, the map 
$$
b\mapsto e^{\tau T_b}\Theta
$$ 
is continuous from $[0,\infty)$ to $L_n^2$.
\end{lemma}

\begin{proof}
Notice that the velocity field $\bar{V}_b=\bar{V}-bX$ is Lipschitz.
Hence, the first statement follows from the Cauchy-Lipschitz theory applied to the transport equation.
In fact, it is well-known that the solution to the transport equation is given by
\begin{equation}\label{eq:semiequalflow}
e^{\tau T_b}
	\Theta
	=\Theta\circ\bar{X}_b(\tau,\cdot)^{-1},
	\end{equation}
	where $\bar{X}_b$ is the flow map
\begin{equation}\label{eq:flowmap}
\partial_\tau \bar{X}_b
	=\bar{V}_b(\bar{X}_b),
	\quad\quad
	\bar{X}_b|_{\tau=0}=\text{id}.
\end{equation}
For the second statement, by solving the ODE
$$
\partial_\tau J_{\bar{X}_b}=\mathrm{div}(\bar{V}_b)J_{\bar{X}_b}
\quad\text{with}\quad
\mathrm{div}(\bar{V}_b)=-2b,
$$
we deduce that the Jacobian of the flow map equals
\begin{equation}\label{JX}
J_{\bar{X}_b}=e^{-2b\tau}.
\end{equation}
Therefore,
$$
\int_{\R^2}|e^{\tau T_b}\Theta|^2\dif X
=e^{-2b\tau}
\int_{\R^2}|\Theta|^2\dif Y,
$$
where $X=\bar{X}_b(\tau,Y)$. Due to \eqref{eq:semiequalflow}, the second statement follows from the fact that the flow map $\bar{X}_b$ defined in \eqref{eq:flowmap} is continuous in $b$, which is a consequence of the continuity of $\bar{V}_b$ in $b$. 
More rigorously, we first prove the continuity in $b$ for $\Theta\in L_n^2\cap C_c^\infty$ by applying pointwise convergence followed by the dominated convergence theorem. Then, we extend the result to the entire domain by density.
\end{proof}

\subsection{The bounded operator $K$}\label{sec:K}

Analogously to the 2D Euler equation ($\alpha=0$) we check that the $\alpha$-Biot-Savart law \eqref{eq:BiotSavart} can be extended, for any $0\leq\alpha\leq 1$, within the subspace $L_n^2$. This is possible thanks to the following proposition.
By slight abuse of notation, we will continue to denote the $\alpha$-Biot-Savart operator as $V=-\nabla^\perp\Lambda^{\alpha-2}\Theta$.

\begin{lemma}\label{lemma:VDVL2}
There exists $C>0$ such that
$$
R^{\alpha-1}\|V\|_{L^2(B_R)}+\|V\|_{\dot{H}^{1-\alpha}(\R^2)}
\leq C\|\Theta\|_{L^2(\R^2)},
$$
for any $R>0$ and $\Theta\in L_n^2$, where $V=-\nabla^\perp\Lambda^{\alpha-2}\Theta$, for any $0\leq\alpha\leq 1$.
\end{lemma}
\begin{proof}
Let $\Theta\in L_n^2\cap C_c^\infty$.
Firstly, by applying the Plancherel identity
to the $\alpha$-Biot-Savart law
$$
\hat{V}
=-i\xi^\perp|\xi|^{\alpha-2}\hat{\Theta},
$$
we deduce that
$$
\|V\|_{\dot{H}^{1-\alpha}}
=\|\Theta\|_{L^2}.
$$ 
Secondly, by applying the $n$-fold symmetry \eqref{eq:nfold}, we deduce that
\begin{equation}\label{eq:Vnfold}
V(X)
=C_\alpha\int_{\R^2}\frac{(X-e^{-\frac{2\pi i k}{n}}Y)^\perp}{|X-e^{-\frac{2\pi i k}{n}}Y|^{2+\alpha}}\Theta(Y)\dif Y
=e^{-\frac{2\pi i k}{n}} V(e^{\frac{2\pi i k}{n}}X).
\end{equation}
Therefore, by integrating $V$ on the ball $B_R$, we get
$$
\int_{B_R}V(x)\dif x
=e^{-\frac{2\pi i k}{n}}\int_{B_R}V(x)\dif x,
$$
from which we deduce that
\begin{equation}\label{eq:intV=0}
\int_{B_R}V(x)\dif x
=\left(\frac{1}{n}\sum_{k=0}^{n-1}e^{-\frac{2\pi i k}{n}}\right)
\int_{B_R}V(x)\dif x
= 0.
\end{equation}
Thus, if we apply the fractional Poincaré inequality  (see e.g.~\cite{Triebel1983}), we deduce that
$$
R^{\alpha-1}\|V\|_{L^2(B_R)}
\leq C
\|V\|_{\dot{H}^{1-\alpha}(B_R)}
\leq C\|V\|_{\dot{H}^{1-\alpha}(\R^2)},
$$
for any $R>0$.
\end{proof}

We take the opportunity to derive a lemma that will be helpful later in Section \ref{sec:non-linearinstability}.

\begin{lemma}\label{propABCDGMKpp}
Suppose $m\geq4$. There exists $C>0$ such that
$$
\left\|\frac{V}{|X|}\right\|_{L^\infty}+\left\|\frac{\nabla\Theta}{|X|}\right\|_{L^\infty}
\leq C\|\Theta\|_{H^m}
$$
for any $\Theta\in L_n^2\cap H^m$, where $V=-\nabla^\perp\Lambda^{\alpha-2}\Theta$.
\end{lemma}
\begin{proof}
Let $\Theta\in L_n^2\cap C_c^\infty$. By the Morrey inequality we have that $DV,D^2\Theta\in L^\infty$. In particular, $V$ and $\nabla\Theta$ are continuous. Therefore, dividing both sides of \eqref{eq:intV=0} by $|B_R|$ and letting $R\to 0$, we deduce that $V(0)=0$. Consequently, by integrating $DV$ over the segment with endpoints $0$ and $X$ we get
$$|V(X)|\leq C |X|\|DV\|_{L^\infty}\leq C |X|\|\Theta\|_{H^m}.$$

Similarly, by the $n$-fold symmetry of $\Theta$ (recall \eqref{eq:nfold}), we have that
\[\nabla\Theta(X)=e^{\frac{2\pi i k}{n}}\nabla\Theta\left(e^{\frac{2\pi i k}{n}}X\right),\]
thus we conclude that $\nabla\Theta(0)=0$.
Hence, as before, $|\nabla\Theta(X)|\leq C |X|\|\Theta\|_{H^m}.$
\end{proof}

\subsection{Case $0\leq\alpha<1$}\label{sec:selfsimilar:aalphaSQG}

In this section, we prove that $\bar{\Theta}$ is self-similarly unstable for the cases $0\leq\alpha<1$. 
To this end, we will apply Proposition \ref{prop:SpectralAnalysis} to
$$
A_b
=b(a-\alpha)+T_b,
\quad\quad
C=K.
$$

\begin{lemma}\label{lemma:Aa<1}
The operators 
$A_b=b(a-\alpha)+T_b$
satisfy the condition \eqref{prop:SpectralAnalysis:1} in Proposition \ref{prop:SpectralAnalysis}.
\end{lemma}
\begin{proof}
By applying Lemma \ref{lemma:Tbsemigroup} and the fact that the identity is a bounded operator, Proposition \ref{prop:stabilitySCS} implies that $A_b$ generates a strongly continuous semigroup.
Then, for any $\Theta_0\in L_n^2\cap C_c^\infty$, there exists a unique global solution 
$\Theta=e^{\tau A_b}\Theta_0$
to
$$\partial_\tau\Theta=A_b\Theta,
\quad\quad
\Theta|_{\tau=0}=\Theta_0.$$
By applying the identity (recall $\mathrm{div}(\bar{V}_b)=-2b$)
$$
\int_{\R^2}\Theta T_b\Theta\dif X
=-\frac{1}{2}\int_{\R^2}\bar{V}_b\cdot\nabla|\Theta|^2\dif X
=-b\int_{\R^2}|\Theta|^2\dif X,
$$
the following energy estimate shows that $A_b=b(a-\alpha)+T_b$ is a contraction
$$
\partial_\tau\int_{\R^2}|\Theta|^2\dif X
=\int_{\R^2}\Theta A_b\Theta\dif X
= b(a-\alpha-1)\int_{\R^2}|\Theta|^2\dif X\leq 0.
$$
In the last inequality we use our choice of $0<a<\varepsilon<1+\alpha$ (recall \eqref{eq:a}). We remark that the same inequality is obtained in  the full domain by density.
The continuity in $b$ follows analogously to that of the operator $T_b$ in the proof of Lemma \ref{lemma:Tbsemigroup}.
\end{proof}

Secondly, we check that $C=K$ satisfies condition \eqref{prop:SpectralAnalysis:2} in Proposition \ref{prop:SpectralAnalysis}.

\begin{lemma}\label{lemma:Kcompact}
The operator $K$ is compact in $L_n^2$ for $0\leq \alpha<1$.
\end{lemma}
\begin{proof}
It follows by applying Lemma \ref{lemma:VDVL2} and the (fractional) Rellich-Kondrachov Theorem (see e.g.~\cite[Section 7]{DPV12}). 
Recall that $K\Theta=-V\cdot\nabla\bar{\Theta}$ with $\bar{\Theta}$ smooth and compactly supported.
\end{proof}

Finally, notice that condition \eqref{prop:SpectralAnalysis:3} in Proposition \ref{prop:SpectralAnalysis} was proved in Theorem \ref{thm:L}. Since we have verified conditions 
\eqref{prop:SpectralAnalysis:1}-\eqref{prop:SpectralAnalysis:3}, Proposition \ref{prop:SpectralAnalysis} implies that $\bar{\Theta}$ is self-similarly unstable.

\subsection{Case $\alpha=1$}\label{sec:selfsimilar:SQG}

In this section, we prove that $\bar{\Theta}$ is self-similarly unstable for SQG.
In this case we cannot apply Proposition \ref{prop:SpectralAnalysis} to $A_b=b(a-\alpha)+T_b$ since $K$ is not compact for $\alpha=1$. We bypass this obstacle by decomposing $K$ into a skew-adjoint operator $S$, which preserves the growth bound of the semigroup, along with a commutator $C$, which turns out to be compact. 

\begin{prop}\label{prop:K=S+C}
It holds that
$$
K=S+C,
$$
where $S$ is a skew-adjoint operator, and $C$ is the commutator
\begin{equation}\label{eq:commutator}
C\Theta
=\frac{1}{2}[\Lambda^{-1}\nabla^\perp,\nabla\bar{\Theta}]\Theta.
\end{equation}
\end{prop}
\begin{proof}
This follows by integrating by parts
\begin{align*}
\int_{\R^2}\Theta K\Theta\dif X
&=-\int_{\R^2}\Theta (\nabla^\perp\Lambda^{-1}\Theta)\cdot\nabla\bar{\Theta}\dif X\\
&=\int_{\R^2}\Lambda^{-1}\Theta(\nabla^\perp\Theta \cdot\nabla\bar{\Theta})\dif X\\
&=\int_{\R^2}\Theta\Lambda^{-1}(\nabla^\perp\Theta \cdot\nabla\bar{\Theta})\dif X.
\end{align*}
Therefore, 
$$
\int_{\R^2}\Theta K\Theta\dif X
=\int_{\R^2}\Theta C\Theta\dif X,
$$
where
$$
C\Theta
=\frac{1}{2}(\Lambda^{-1}(\nabla^\perp\Theta \cdot\nabla\bar{\Theta})
-(\Lambda^{-1}\nabla^\perp\Theta)\cdot\nabla\bar{\Theta}).
$$
By definition, we have
\begin{equation}\label{eq:Sskewadjoint}
\int_{\R^2}\Theta S\Theta\dif X
=\int_{\R^2}\Theta(K-C)\Theta\dif X=0.
\end{equation}
This property is equivalent to being skew-adjoint.
\end{proof}

We will apply Proposition \ref{prop:SpectralAnalysis} to
$$
A_b
=b(a-\alpha)+T_b+S,
\quad\quad
C=K-S.
$$

\begin{lemma}\label{lemma:commutator}
The commutator \eqref{eq:commutator} is compact in $L_n^2$.
\end{lemma}
\begin{proof} 
It is well known that $C$ is a smoothing operator in $L^2$ (see e.g.~\cite{Steinbook}). Therefore, it is compact when restricted to  $L^2(\Omega)$  for bounded subsets $\Omega\subset\R^2$. In order to extend the compactness to the whole plane, we will exploit the fact that $\bar{\Theta}$ is smooth and has compact support $B_{\bar{R}}$, with $\bar{R}=r_2+\varepsilon$. We split the commutator into
$$
C=C_1-C_2,
$$
where
$$
C_1\Theta
=\frac{1}{2}\Lambda^{-1}\nabla^\perp\cdot(\Theta\nabla\bar{\Theta}),
\quad\quad
C_2\Theta
=\frac{1}{2}(\Lambda^{-1}\nabla^\perp\Theta\cdot\nabla\bar{\Theta}).
$$
On the one hand, notice that $C_1$ is the Riesz transform of $H=\Theta\nabla\bar{\Theta}$
$$
C_1\Theta(X)
=\frac{1}{4\pi}
\int_{\R^2}
\frac{(X-Y)^\perp}{|X-Y|^3}\cdot H(Y)\dif Y.
$$
Therefore, for any $R>2\bar{R}$, we have
\begin{align*}
\int_{\R^2\setminus B_R}
|C_1\Theta|^2\dif X
&\leq\frac{1}{(4\pi)^2}\int_{|X|\geq R}
\left|
\int_{|Y|\leq\bar{R}}\frac{1}{|X-Y|^2}|H(Y)|
\dif Y
\right|^2\dif X\\
&\leq
\frac{1}{\pi^2}\int_{|X|\geq R}\frac{\dif X}{|X|^4}
\left|\int_{|Y|\leq\bar{R}}|H(Y)|\dif Y\right|^2
\leq\left(\frac{\bar{R}}{R}\right)^2\|H\|_{L^2}^2,
\end{align*}
where in the second line we applied that 
$$
|X-Y|\geq |X|-|Y|\geq \frac{1}{2}|X|.
$$ 
Observe that
$$
\|H\|_{L^2}
\leq\|\nabla\bar{\Theta}\|_{L^\infty}
\|\Theta\|_{L^2}.
$$
On the other hand, since $C_2\Theta$ vanishes outside the support of $\bar{\Theta}$, we have
$$
\int_{\R^2\setminus B_R}|C_2\Theta|^2\dif X=0.
$$
These estimates allow us to conclude that $C$ is compact on the whole plane.
Let $\Theta_n$ be a sequence that is uniformly bounded in $L^2$.
Since $C$ is compact when restricted to  $L^2(\Omega)$  for bounded subsets $\Omega\subset\R^2$, by a diagonal argument, we can extract a subsequence of $C\Theta_n$ that converges to some $D$ in $L^2$ on any ball $B_N$. By relabeling if necessary, we denote this subsequence by $C\Theta_n$, with a slight abuse of notation.
Let $\varepsilon>0$. Since $\Theta_n$ is uniformly bounded in $L^2$, the bounds we obtained for $C_1$ and $C_2$ ensure that we can take $N\in\N$ big enough such that
$$
\|C\Theta_n-D\|_{L^2(\R^2\setminus B_N)}
\leq
\|C\Theta_n\|_{L^2(\R^2\setminus B_N)}
+\|D\|_{L^2(\R^2\setminus B_N)}
\leq
\frac{\varepsilon}{2},
$$
uniformly in $n$. Once $N$ is fixed, we can take $n_0\in\N$ big enough such that
$$
\|C\Theta_n-D\|_{L^2(B_N)}
\leq\frac{\varepsilon}{2},
$$
for all $n\geq n_0$. Therefore, $C\Theta_n\to D$ in $L^2$. We have proved that $C$ is compact in $L^2$, and thus also in $L_n^2$.
\end{proof}

\begin{lemma}\label{lemma:Aa=1}
The operators 
$A_b=b(a-\alpha)+T_b+S$
satisfy the condition \eqref{prop:SpectralAnalysis:1} in Proposition \ref{prop:SpectralAnalysis}.
\end{lemma}
\begin{proof} 
By applying Lemma \ref{lemma:Aa<1} and that $S=K-C\in\mathcal{L}$, Proposition \ref{prop:stabilitySCS} implies that $A_b$ generates a strongly continuous semigroup.
Similarly to the proof of Lemma  \ref{lemma:Aa<1}, but now applying that $S$ is skewadjoint (recall \eqref{eq:Sskewadjoint}),
the following energy estimate on $\Theta=e^{\tau A_b}\Theta_0$ shows that $A_b=b(a-\alpha)+T_b+S$ is a contraction
$$
\partial_\tau\int_{\R^2}|\Theta|^2\dif X
=\int_{\R^2}\Theta A_b\Theta\dif X
=\int_{\R^2}\Theta (A_b-S)\Theta\dif X
= b(a-\alpha-1)\int_{\R^2}|\Theta|^2\dif X\leq 0.
$$
In the last inequality we use our choice of $0<a<\varepsilon<1+\alpha$ (recall \eqref{eq:a}). 
The continuity in $b$ follows analogously to that of the operator $T_b$ in the proof of Lemma \ref{lemma:Tbsemigroup}.
\end{proof}

Since condition \eqref{prop:SpectralAnalysis:1} has been verified in Lemma~\ref{lemma:Aa=1}, condition \eqref{prop:SpectralAnalysis:2} in Lemma~\ref{lemma:commutator}, and condition \eqref{prop:SpectralAnalysis:3} in Theorem \ref{thm:L}, Proposition \ref{prop:SpectralAnalysis} implies that $\bar{\Theta}$ is self-similarly unstable.

\subsection{The eigenfunction}\label{sec:eigenfunction}

In this section we prove that the eigenfunction 
$$
W_b\in\mathrm{Ker}(L_b-\lambda_b)
$$
associated with the eigenvalue $\lambda_b$ appearing in Proposition \ref{prop:SpectralAnalysis} is smooth and compactly supported. 

Given $m\geq 5$, we fix $b=b_m$ satisfying
\begin{equation}\label{eq:b<}
0<b
<\frac{\Re\lambda_b}{m+3}.
\end{equation}
Thus, from now on, we will denote $W$ and $\lambda$ instead of $W_b$ and $\lambda_b$ to alleviate the notation.

Since $U_{jn}$ is invariant under $L_b$ and 
$$
W=\sum_{j\in\Z}
W_{jn}(R)e^{ijn \theta},
$$
each (non-zero) projection $W_{jn}(R)e^{ijn \theta}\in U_{jn}$ is also an eigenfunction.  For $j=0$, since $L_b W_{0}=\lambda_b W_{0}$ reads as
$$
b(a+R\partial_R)W_{0}=\lambda_b W_{0},
$$
necessarily
$W_{0}=0$. Therefore,
$
0\neq W_{jn}(r)e^{ijn\theta}\in\mathrm{Ker}(L_b-\lambda)\cap U_{jn}
$
for some $j\neq 0$.
Moreover, since $\mathrm{Ker}(L_b-\lambda)$ is finite dimensional (see Lemma \ref{lemma:Fredholm}), $W_{jn}$ is null
for all but a finite number of (non-zero) $j$'s.

\begin{prop}\label{prop:eigenfunction}
Let $W\in\mathrm{Ker}(L_b-\lambda)\cap U_{jn}$ with $j\neq 0$. Then, 
$$
W\in C^\infty(\R^2\setminus\{0\})
\cap C_c^\gamma(\R^2),
$$
with
$\gamma =(\frac{\Re\lambda}{b}-(a-\alpha))\geq m+2$. Moreover, 
$$
W_{jn}(R)
=C_0
R^{\frac{\lambda-b(a-\alpha)+ijn C}{b}}e^{-\frac{ijn}{b}U(R)},
$$
for all $R\leq\frac{r_1}{2}$, where $C_0$ is a constant, and $C$ and $U(R)$ are given in Proposition \ref{prop:asymptoticbarV}.
\end{prop}
\begin{proof}
The self-similar stability equation \eqref{eq:SQG:selfsimilarstability} can be rewritten as
\begin{equation}\label{eq:RSE:eigenfunction}
\left((\lambda-b(a-\alpha))+ijn\frac{\bar{V}_\phi}{R}-bR\partial_R\right)W_{jn} +in C_\alpha\partial_R\bar{\Theta}\int_0^\infty I_{jn,\alpha}(S)
W_{jn}\left(\frac{R}{S}\right)\frac{\dif S}{S^{1+\alpha}}
=0.
\end{equation}
We remark that \eqref{eq:RSE:eigenfunction} is well defined since $W\in D(L_b)$ by construction.
On the interval $(0,r_1-\varepsilon)$, we have $\partial_R\bar{\Theta}=0$. Thus,  $W_{jn}$ solves an ODE that can be solved explicitly by applying the formula for 
$\bar{V}_\phi$ from Proposition \ref{prop:asymptoticbarV}. We obtain  that, on the interval $(0,r_1-\varepsilon)$, it holds that
$$
W_{jn}(R)
=C_0
R^{\frac{(\lambda-b(a-\alpha))+ijn C}{b}}e^{-\frac{ijn}{b}U(R)},
$$
for some new constant $C_0$. On the interval $(r_2+\varepsilon,\infty)$,  
since $\partial_R\bar{\Theta}=0$, we deduce that
$$
|W_{jn}(R)|
=C_1
R^{\frac{\Re\lambda-b(a-\alpha)}{b}}.
$$
In this case, since $W_{jn}\in L^2$ and by \eqref{eq:b<}, necessarily $C_1=0$. As a consequence, we have
$\mathrm{supp}(W_{jn})\subset[0,r_2+\varepsilon]$. 
We have checked that $W_{jn}$ is smooth outside the interval $B_\varepsilon([r_1,r_2])$. Let us check that it is smooth inside the interval $I_\varepsilon=B_{2\varepsilon}([r_1,r_2])$.
Notice that the last integrand in \eqref{eq:RSE:eigenfunction} is supported on $S\geq \frac{R}{r_2+\varepsilon}$.
Hence, it follows from \eqref{eq:RSE:eigenfunction} that $W_{jn}\in H^1(I_\varepsilon)$. By bootstrapping, the same formula allows to prove that $W_{jn}\in H^k(I_\varepsilon)$ for any $k\geq 1$.
\end{proof}

\section{Non-linear instability}
\label{sec:non-linearinstability}

In this section we prove Theorem \ref{thm:non-linear}. We emphasize that the proof works for $b>0$, which establishes the Sobolev non-uniqueness result (Theorem \ref{thm:main}), as well as for $b=0$, which demonstrates the existence of unstable vortices exhibiting non-uniqueness at $t=-\infty$ (Theorem \ref{thm:unstablevortex}).
In this section, $m\geq 5$ is fixed (except for the definition of $Y^m$ and the results related to its properties) and a corresponding $b$ is also fixed in the regime
\begin{equation}\label{eq:b<=}
0\leq b
<\frac{\Re\lambda}{m+3}.
\end{equation}

In order to prove Theorem \ref{thm:non-linear}, we construct a sequence of approximate solutions $\Theta_k^{(q)}$ to \eqref{eq:SQG:cor} that satisfy the decay condition \eqref{eq:expdecay} uniformly in both $k$ and $q$, and then we pass to the limit in a suitable Sobolev space. Indeed, we consider the weighted Hilbert space $H_\omega^m$ given by the norm 
$$
\|f\|_{H_\omega^m}
:=\|f\omega\|_{L^2}
+\sum_{0<|K|\leq m}\|\partial_X^K f\omega\|_{L^2},
$$
where $\omega$ is the standard radial weight 
$$
\omega(X)
=\langle X\rangle^2
=1+R^2.
$$

Firstly, for every $k\in\N$, we consider the unique solution $\Theta_k^{\text{cor}}$ to \eqref{eq:SQG:cor}
\begin{equation}\label{eq:SQG:cor:k}
(\partial_\tau - L_b)\Theta_k^{\text{cor}}
+\underbrace{(V^{\text{lin}}+\epsilon V_k^{\text{cor}})\cdot\nabla(\Theta^{\text{lin}}+\epsilon\Theta_k^{\text{cor}})}_{\mathcal{F}_k}
=0,
\end{equation}
coupled with the initial condition (replacing \eqref{eq:Thetacorinitial})
\begin{equation}\label{eq:Omegakcor0}
\Theta_k^{\text{cor}}|_{\tau=-k}=0.
\end{equation}
The local existence and uniqueness of this solution is guaranteed by \cite{chaewu}. 
We recall that for $b>0$, in the original system of coordinates, we have that
$$
\theta_{k}
=\theta_0+\epsilon\theta^{\text{lin}}
+\epsilon^2\theta_k^{\text{cor}}
$$
is a solution to the $\alpha$-SQG equation \eqref{eq:SQG} with the initial condition
\begin{equation}\label{eq:omegak}
\theta_{k}|_{t=t_k}=(\theta_0+\epsilon\theta^{\text{lin}})(t_k)
=(abt_k)^{\frac{\alpha}{a}-1}(\bar{\Theta}+\epsilon\Theta^{\text{lin}}(-k)),
\end{equation}
where $t_k=e^{-abk}$, and with the forcing
\begin{equation}\label{eq:ftk}
f=-(abt)^{\frac{\alpha}{a}-2}b((a-\alpha)+X\cdot\nabla)\bar{\Theta}.
\end{equation}
The corresponding $\Theta^{\text{cor}}_k$ solves \eqref{eq:SQG:cor:k} and \eqref{eq:Omegakcor0}. 

Secondly, for every $k\in\N$, we recover $\Theta_k^{\text{cor}}$ as the limit of the following iterative scheme. Starting from $\Theta_k^{(0)}=0$, we define $\Theta_k^{(q)}$ for any $q\in\N$ by
\begin{equation}\label{eq:SQG:cor:kq}
(\partial_\tau - L_b)\Theta_k^{(q)}
+\underbrace{(V^{\text{lin}}+\epsilon V_k^{(q-1)})\cdot\nabla(\Theta^{\text{lin}}+\epsilon\Theta_k^{(q)})}_{\mathcal{F}_k^{(q)}}
=0,
\end{equation}
coupled with the initial condition 
\begin{equation}\label{eq:Omegakqcor0}
\Theta_k^{(q)}|_{\tau=-k}=0.
\end{equation}

\begin{Rem}\label{rem:smoothbdryconditions}
On the one hand, since $\bar{\Theta} \in C_c^\infty$ (Theorem \ref{thm:L}) and $\Theta^{\text{lin}} \in C_c^{m+2}$ (Proposition \ref{prop:eigenfunction:b=0} for $b=0$, and Proposition \ref{prop:eigenfunction} for $b>0$), the initial datum satisfies $\bar{\Theta} + \epsilon \Theta^{\text{lin}}(-k) \in C_c^{m+2}$. On the other hand, for $b>0$, the force \eqref{eq:ftk} remains smooth and compactly supported for all $t \geq t_k$. Recall that $f=0$ for $b=0$.
Moreover, by construction, both the initial datum and the force have finite Hamiltonian.  
Therefore, for all $k, q$, the solution $\Theta_k^{(q)}$ to this linearized system belongs to $C([-k,T], H^{m+2} \cap\dot{H}^{\frac{\alpha-2}{2}})$ for all $T > -k$.
\end{Rem}

As we explained in the intro, the use of the Duhamel formula allows gaining extra exponential decay, but at the cost of having to control the bound under a stronger norm. In Section \ref{sec:energyspace} we introduce the energy space $Y^m\hookrightarrow H_\omega^m\cap L_n^2$. 
Since we need to work in $Y^m$ (instead of $H_\omega^m$) we must ensure that our approximate solution remains in this space globally in time.

\begin{prop}\label{prop:ThetakqYm}
Let $k,q\in\N$. For every $T>-k$, the unique solution $\Theta_k^{(q)}$ to \eqref{eq:SQG:cor:kq}, \eqref{eq:Omegakqcor0} satisfies
$$
\Theta_k^{(q)}\in L^\infty([-k,T];Y^{m+1})\cap
C([-k,T];Y^{m}).
$$
Moreover, $\Theta_k^{(q)}$ remains compactly supported.
\end{prop}
\begin{proof}
We will prove it in Section \ref{sec:LWPY} after introducing the space $Y^m$.
\end{proof}

Due to Proposition \ref{prop:ThetakqYm} and recalling \eqref{eq:Omegakqcor0}, given some fixed $\delta_0$ in the regime
$$0<\delta_0<c_{(m,0)}:=2^{-\frac{m(m+3)}{2}},$$  
we can define $-k<\tau_k^{(q)}\leq 0$ to be the largest non-positive time such that
\begin{equation}\label{kdependentbound}
\Vert\Theta^{(q)}_k(\tau)\Vert_{Y^m}\leq e^{(1+\delta_0)\Re\lambda\tau},
\end{equation}
for any $-k\leq\tau\leq\tau_k^{(q)}$. The purpose of the notation $c_{(m,0)}$ will become clear later in the text. Our task is to show that the times $\tau_k^{(q)}$ satisfying \eqref{kdependentbound} do not diverge to $-\infty$ as $k,q\to\infty$.
For this purpose, we will use a bootstrapping argument to improve the bound \eqref{kdependentbound}. This is the content of the following lemma, whose proof will take most of this section.

\begin{lemma}\label{lemma:non-lineark}
There exists $C>0$, independent of $k$ and $q$, such that 
\begin{equation}\label{eq:non-lineark}
\Vert\Theta^{(q)}_k(\tau)\Vert_{Y^m}\leq C e^{(1+c_{(m,0)})\Re\lambda\tau},
\end{equation}
for all $-k\leq\tau\leq\min_{p\leq q}\tau_k^{(p)}$.
\end{lemma}
\begin{proof}
We will prove it in Sections \ref{eq:baselineL2}-\ref{sec:Hcom}.
\end{proof}

As a corollary of this lemma, we obtain the required lower bound for $\tau_k^{(q)}$.

\begin{cor}\label{cor:bartau}
It holds that
\begin{equation}\label{eq:bartau}
\bar{\tau}
:=\inf_{k,q}\tau_k^{(q)} > -\infty.
\end{equation}
\end{cor}
\begin{proof}
We prove it by contradiction. Suppose that $\bar{\tau}=-\infty$. Then, there must exist a sequence satisfying
\begin{equation}\label{eq:bartaudiverges}
0>\tau_{k_j}^{(q_j)}\to -\infty,
\end{equation}
as $j\to\infty$. We can assume, without loss of generality, that
\begin{equation}\label{eq:taukqj}
\tau_{k_j}^{(q_j)}=\min_{p\leq q_j}\tau_{k_j}^{(p)}.
\end{equation}
Otherwise, it would suffice to replace $q_j$ by the $p$ that minimizes \eqref{eq:taukqj}.
By definition of $\tau_k^{(q)}$ in \eqref{kdependentbound}, we have
\begin{equation}\label{kdependentequality}
\left\Vert\Theta^{(q_j)}_{k_j}\left(\tau_{k_j}^{(q_j)}\right)\right\Vert_{Y^m}= e^{(1+\delta_0)\Re\lambda\tau_{k_j}^{(q_j)}}.
\end{equation}
Hence, by combining \eqref{eq:non-lineark} with \eqref{kdependentequality}, we get
$$
C e^{(c_{(m,0)}-\delta_0)\Re\lambda\tau_{k_j}^{(q_j)}}\geq 1,
$$
which contradicts \eqref{eq:bartaudiverges} due to $\delta_0<c_{(m,0)}$. 
\end{proof}

In the next section, we show how Theorem \ref{thm:non-linear} is proved once we know that \eqref{eq:bartau} holds. Therefore, in order to conclude it will remain to check Lemma \ref{lemma:non-lineark}. 

\subsection{Proof of Theorem \ref{thm:non-linear}}
\label{sec:Proofnon-linear}

Firstly, we take the limit $q\to\infty$, to recover the solutions $\Theta_k^{\text{cor}}$. By the embedding $Y^m\hookrightarrow H_\omega^m$ (see Lemmas \ref{lemma:YminweightedHm} and \ref{lemma:Marcos}) and \eqref{kdependentbound}, these solutions satisfy
$$
\|\Theta_k^{\text{cor}}(\tau)\|_{H_\omega^m}\lesssim
\|\Theta_k^{\text{cor}}(\tau)\|_{Y^m}
\leq\liminf_{q\to\infty}\|\Theta_k^{(q)}(\tau)\|_{Y^m}
\leq e^{(1+\delta_0)\Re\lambda\tau},
$$
for all $-k\leq\tau\leq\bar{\tau}$, 
with $\bar{\tau}>-\infty$ by Corollary \ref{cor:bartau}. 
Secondly, we take the limit $k\to\infty$.
With this regularity, the sequence $\Theta^{\text{cor}}_k$ converges to a solution $\Theta^{\text{cor}}$ of \eqref{eq:SQG:cor}, which satisfies
$$
\|\Theta^{\text{cor}}(\tau)\|_{H_\omega^m}
\leq
\liminf_{k\to\infty}
\|\Theta_k^{\text{cor}}(\tau)\|_{H_\omega^m}
\lesssim e^{(1+\delta_0)\Re\lambda\tau},
$$
for all $\tau\leq\bar{\tau}$. Notice that the Hamiltonian remains finite since the sequence of classical solutions $\Theta_k^{\text{cor}}$ satisfy the Hamiltonian identity, which provides a uniform bound in $k$ due to Remark \ref{rem:smoothbdryconditions}.

\subsection{Weighted energy space}\label{sec:energyspace}

As we explained in Section \ref{sec:sketch}, it will be necessary to switch between Cartesian and polar coordinates. The following lemmas relate the partial derivatives of any order in both coordinate systems. We present a proof in Appendix \ref{sec:cartesianpolar} for completeness.

\begin{lemma}[From Cartesian to Polar derivatives]\label{lemma:cartesianpolar}
For any multi-index $K=(k_1,k_2)$ with $|K|>0$, it holds that
$$
\partial_X^Kf
=\sum_{0<|J|\leq|K|}p_J^K\frac{\partial_R^{j_1}\partial_\phi^{j_2}f}{R^{|K|-j_1}},
$$
where $p_J^K=p_J^K(\cos\phi,\sin\phi)$ is a $|K|$-homogeneous polynomial.
\end{lemma}

\begin{lemma}[From Polar to Cartesian derivatives]\label{lemma:polarcartesian}
For any multi-index $J=(j_1,j_2)$ with $|J|>0$, it holds that
$$
\partial_R^{j_1}\partial_\phi^{j_2}f
=\sum_{\substack{0<|K|\leq|J| \\ j_1\leq|K|}}
q_K^J R^{|K|-j_1}\partial_X^K f,
$$
where $q_K^J=q_K^J(\cos\phi,\sin\phi)$ is a $|K|$-homogeneous polynomial.
\end{lemma}

Given $m\in\N$, we define the energy subspaces $Y^m$ and $Z^m$ of $L_n^2$ given by the norms
\begin{align*}
\|f\|_{Y^m}
&:=\|f\omega\|_{L^2}
+\sum_{0<|J|\leq m}\left\|\frac{\partial_R^{j_1}\partial_\phi^{j_2}f}{R^{m-j_1}}\omega \right\|_{L^2},\\
\|f\|_{Z^m}
&:=\|f\omega\|_{L^2}
+\sum_{0<|K|\leq m}\left\|\frac{\partial_X^K f}{R^{m-|K|}}\omega \right\|_{L^2},
\end{align*}
where $L^2=L^2(\R^2,\mathrm{d}X)$ in Cartesian coordinates.
In Lemmas \ref{lemma:YminweightedHm} and \ref{lemma:Marcos} we will see that indeed
$$
Y^m=Z^m\hookrightarrow H^m_\omega,
$$
as well as other embeddings that will be crucial during the energy estimates. 
Notice that $Y^1=Z^1=H_\omega^1$, while in general $Y^m=Z^m\subsetneq H_\omega^m$ for $m\geq 2$.

\subsubsection{Embeddings}

Note that the inequality
\begin{equation}\label{eq:ZminHm}
    \|f\|_{H^m}\lesssim\|f\|_{Z^m}
\end{equation}
follows by the definition of $Z^m$. 
Indeed, $\|f\|_{L^2}\leq\|f\omega\|_{L^2}\leq \|f\|_{Z^m}$ and, for any $|K|=m$, $\|\partial_X^K f\|_{L^2}\leq \|\partial_X^K f \, \omega\|_{L^2}\leq \|f\|_{Z^m}$.
 \begin{lemma}\label{lemma:YminweightedHm}
    For any $0\leq|K|\leq m$ and $f\in Z^m$,
\begin{equation}\label{eq:lemmapesos}
    \|\partial_X^K f \, \omega\|_{H^{m-|K|}}\lesssim\|f\|_{Z^m}.
\end{equation}
In particular, $$\|f\|_{H^m_\omega}\lesssim\|f\|_{Z^m}.$$
\end{lemma}
\begin{proof}
Note that the case $|K|=m$ follows directly by the definition of $Z^m$.
We start by proving the case $|K|=0$. Since $\|f \omega\|_{L^2}\leq \|f\|_{Z^m}$, 
it is enough to see that $\|\partial^L_X (f\omega)\|_{L^2}\lesssim \|f\|_{Z^m}$ for every $|L|=m$. Therefore, we let $|L|=m$ and apply the Leibniz rule
$$\partial^L_X (f\omega)=\sum_{(0,0)\leq J\leq L}{L \choose J}\partial^{L-J}_X f \, \partial^J_X\omega.$$
Note that, since $\omega(X)=1+X_1^2+X_2^2$, we have $\partial^J_X\omega=0$ for $J\neq(0,0),(1,0),(0,1),(2,0),(0,2)$. We deal with these cases separately. If $J=(0,0)$, then
$$\|\partial^{L-J}_X f \, \partial^J_X\omega\|_{L^2}=\|\partial^{L}_X f \, \omega\|_{L^2}\leq \|f\|_{Z^m}.$$
If $J=(1,0)$ or $J=(0,1)$
we have $\partial^J_X\omega=2X_1$ or $2X_2$, respectively. Therefore, since $|L-J|=m-1$,
$$\|\partial^{L-J}_X f \, \partial^J_X\omega\|_{L^2}=\left\|\frac{\partial^{L-J}_X f}{R} \frac{R\partial^J_X\omega}{\omega} \omega\right\|_{L^2}\leq 2\left\|\frac{\partial^{L-J}_X f}{R^{m-|L-J|}} \omega\right\|_{L^2}\lesssim \|f\|_{Z^m}.$$
Lastly, if $J=(2,0)$ or $J=(0,2)$ we have $\partial_X^J\omega=2$, thus, by \eqref{eq:ZminHm},
$$\|\partial^{L-J}_X f \, \partial^J_X\omega\|_{L^2}=2\|\partial^{L-J}_X f\|_{L^2} \lesssim \|f\|_{Z^m}.$$

We now proceed inductively on $|K|$. Let $1<|K|\leq m$ and suppose that \eqref{eq:lemmapesos} holds for every $K'$ with $|K'|=|K|-1$. Take $|J|=1$ with $J\leq K$, then
\begin{align*}
\|\partial^K_Xf \omega\|_{H^{m-|K|}}&=\|\partial^{J}_X(\partial^{K-J}_X f\omega)-\partial^{K-J}_Xf\partial_X^J\omega\|_{H^{m-|K|}}=\left\|\partial^{J}_X(\partial^{K-J}_X f\omega)-\partial^{K-J}_Xf\frac{\partial_X^J\omega}{\omega}\omega\right\|_{H^{m-|K|}}\\
&\lesssim \|\partial^{K-J}_X f\omega\|_{H^{m-(|K|-1)}}+\|\partial^{K-J}_X f \omega\|_{H^{m-|K|}}
\lesssim
\|\partial_X^{K-J}f\omega\|_{H^{m-|K-J|}}
\lesssim \|f\|_{Z^m},
\end{align*}
where $K-J$ plays the role of $K'$.
\end{proof}

\begin{lemma}\label{lemma:Marcos} 
It holds that
$$
\|f\|_{Z^m}
\simeq\|f\|_{Y^m}.
$$
Moreover, 
\begin{align*}
\|f\|_{Y^m}
&\simeq
\|f\|_{\bar{Y}^m}:=\|f\omega\|_{L^2}
+\sum_{0<|J|\leq r+j_2\leq m}\left\|\frac{\partial_R^{j_1}\partial_\phi^{j_2}f}{R^{m-r}}\omega \right\|_{L^2},\\
\|f\|_{Z^m}
&\simeq
\|f\|_{\bar{Z}^m}:=\|f\omega\|_{L^2}
+\sum_{0<|K|\leq r\leq m}\left\|\frac{\partial_X^K f}{R^{m-r}}\omega \right\|_{L^2}.
\end{align*}
\end{lemma}
\begin{proof}
To prove the equivalences of norms we prove the following sequence of inequalities:
$$\|f\|_{Z^m}
\lesssim\|f\|_{Y^m}
\leq\|f\|_{\bar{Y}^m}
\lesssim\|f\|_{\bar{Z}^m}
\lesssim\|f\|_{Z^m}.$$
To prove $\|f\|_{Z^m}\lesssim\|f\|_{Y^m}$ we use Lemma \ref{lemma:cartesianpolar} to write
$$
\frac{\partial_X^K f}{R^{m-|K|}}\omega
=\sum_{0<|J|\leq|K|}p_J^K\frac{\partial_R^{j_1}\partial_\phi^{j_2}f}{R^{m-j_1}}\omega
$$
for any $0<|K|\leq m$ and then apply the definition of $Y^m$.

The inequality $\|f\|_{Y^m}\leq\|f\|_{\bar{Y}^m}$ is trivial. For the inequality $\|f\|_{\bar{Y}^m}\lesssim\|f\|_{\bar{Z}^m}$ let $0<|J|\leq r+j_2\leq m$ and use Lemma \ref{lemma:polarcartesian} to write
$$\frac{\partial_R^{j_1}\partial_\phi^{j_2}f}{R^{m-r}}\omega
=\sum_{\substack{0<|K|\leq|J| \\ j_1\leq|K|}}
q_K^J \frac{\partial_X^K f}{R^{m-(r-j_1+|K|)}}\omega.$$
Then the inequality follows by using the definition of $\bar{Z}^m$ since $|K|\leq r-j_1+|K|\leq m$. 

Finally, $\|f\|_{\bar{Z}^m}\lesssim\|f\|_{Z^m}$ follows by decoupling
$$
\frac{\left|\partial_X^K f\right|}{R^{m-r}}\omega
\leq \frac{\left|\partial_X^K f\right|}{R^{m-|K|}}\omega 1_B
+\left|\partial_X^K f\right|\omega 1_{B^c}\leq \|f\|_{Z^m}+\|f\|_{H^m_\omega},
$$
for any $0<|K|\leq r\leq m$, and then applying Lemma \ref{lemma:YminweightedHm}.
\end{proof}
The next lemma relates $H_\omega^m$ with classical Sobolev spaces. The embedding is not optimal, but since $m$ is arbitrarily large, we present a simple version that is enough for our purposes.
\begin{lemma}\label{lemma:embeddingYmSobolev}
For all $1\leq p\leq\infty$,
$$
H_\omega^m\hookrightarrow W^{m-2,p}.
$$
\end{lemma}
\begin{proof}
Let $f\in H_\omega^m$. On the one hand, for any $0\leq|K|\leq m$,
$$
\|\partial^K f\|_{L^q}
\leq\|\partial^K f\omega\|_{L^2}\|\omega^{-1}\|_{L^{\frac{2q}{2-q}}},
$$
with $\omega^{-1}\in L^{\frac{2q}{2-q}}$ for $\frac{2}{3}<q\leq 2$. Therefore, $f\in W^{m,q}$ for all $1\leq q\leq 2$. On the other hand, since $f\in H^m$, the Morrey inequality implies that $f\in W^{m-2,\infty}$. Hence, by interpolation, we have $ f\in W^{m-2,p}$ for all $1\leq p\leq\infty$.
\end{proof}

In the next lemma,
we show that $Y^{m+1}(\Omega)$ is compact in $Y^m(\Omega)$ for bounded subsets $\Omega\subset\R^2$ and arbitrary $m\in\N$, where $Y^m(\Omega)$ is given as $Y^m$ but replacing $L^2(\R^2,\mathrm{d}X)$ by $L^2(\Omega,\mathrm{d}X)$.

\begin{lemma}\label{lemma:Ymcompact} 
The embedding $Y^{m+1}(\Omega)\hookrightarrow Y^m(\Omega)$ is compact for any bounded subset $\Omega\subset\R^2$.
\end{lemma}
\begin{proof}
Given $f\in Y^{m+1}(\Omega)$ and $0<|J|\leq m$, we denote
$$
f_J=\frac{\partial^J f}{R^{m-j_1}}=\frac{\partial_R^{j_1}\partial_\phi^{j_2}f}{R^{m-j_1}}.
$$
On the one hand,
$$
\partial_R f_J
=\frac{\partial^{J+(1,0)} f}{R^{(m+1)-(j_1+1)}}
-(m-j_1)\frac{\partial^J f}{R^{m+1-j_1}}.
$$
On the other hand,
$$
\frac{\partial_\phi f_J}{R}
=\frac{\partial^{J+(0,1)} f}{R^{(m+1)-j_1}}.
$$
Therefore, $f_J\in H^1(\Omega)$. The statement follows by applying that the embedding $H^1(\Omega)\hookrightarrow L^2(\Omega)$ is compact. The weight $\omega$ does not play any role here.
\end{proof}

Next, we check that the vortex and eigenfunction belong to these weighted energy spaces.

\begin{lemma}
It holds that
$$
\bar{\Theta}\in\bigcap_{m\in\N}Y^m.
$$
\end{lemma}
\begin{proof}
By construction, we have that $\bar{\Theta}\in C_c^\infty(\R^2)$ with $\bar{\Theta}(X)$ constant on $|X|\leq\frac{1}{4}$.
\end{proof}
\begin{lemma}\label{lemma:thetalinYm}
There exists $C>0$ such that
$$
\|\Theta^{\text{lin}}(\tau)\|_{Y^{m+2}}
\leq C e^{\Re\lambda\tau},
$$
for all $\tau\in\R$.
\end{lemma}
\begin{proof}
Recall that $\Theta^{\text{lin}}(\tau,X)=\Re(e^{\lambda_b\tau}W_b(X))$ (see Section \ref{sec:eigenfunction}). The statement follows by Proposition \ref{prop:eigenfunction:b=0} for $b=0$, and by Proposition \ref{prop:eigenfunction} for $b>0$.
\end{proof}

We conclude this section by controlling the $L^\infty$ norm of the velocity $V=-\nabla^\perp\Lambda^{\alpha-2}\Theta$, namely
\begin{equation}\label{eq:Vbounded}
\|V\|_{L^\infty}\lesssim\|\Theta\|_{H_\omega^m}\lesssim\|\Theta\|_{Y^m}.
\end{equation}
This follows from the following lemma in combination with Lemmas \ref{lemma:YminweightedHm}-\ref{lemma:embeddingYmSobolev}.

\begin{lemma}\label{lemma:Vbounded}
For $\alpha=0$ and $2<p\leq\infty$, there exists $C>0$ such that
$$
\|V\|_{L^\infty}
\leq C(\|\Theta\|_{L^1}+\|\Theta\|_{L^p}).
$$
For $0<\alpha\leq 1$ and $\epsilon>0$, there exists $C>0$ such that
$$
\|V\|_{L^\infty}\leq C\|\Theta\|_{H^{1+\epsilon}}.
$$
\end{lemma}
\begin{proof}
For $\alpha=1$, simply $V\in H^{1+\epsilon}\hookrightarrow L^\infty$ by the Morrey inequality. For $0\leq\alpha<1$, we bound
\begin{align*}
|V(X)|\leq C_\alpha\int_{\R^2}\frac{|\Theta(Y)|}{|X-Y|^{1+\alpha}}\dif Y.
\end{align*}
We split $\R^2$ into $|X-Y|>1$ and $|X-Y|<1$. 
Let $\alpha=0$. In the exterior domain,
$$
\int_{|X-Y|>1}\frac{|\Theta(Y)|}{|X-Y|}\dif Y
\leq\|\Theta\|_{L^1}.
$$
In the interior domain,
$$
\int_{|X-Y|<1}\frac{|\Theta(Y)|}{|X-Y|}\dif Y
\leq\|\Theta\|_{L^p}
\left(2\pi\int_1^\infty\frac{R\dif R}{R^{\frac{p}{p-1}}}\right)^{\frac{p-1}{p}}
=\left(2\pi\frac{p-1}{p-2}\right)^{\frac{p-1}{p}}
\|\Theta\|_{L^p}.$$
Now let $0<\alpha<1$. In the exterior domain,
$$
\int_{|X-Y|>1}\frac{|\Theta(Y)|}{|X-Y|^{1+\alpha}}\dif Y
\leq\|\Theta\|_{L^2}\left(2\pi\int_1^\infty\frac{R\dif R}{R^{2(1+\alpha)}}\right)^{\nicefrac{1}{2}}
=\sqrt{\frac{\pi}{\alpha}}\|\Theta\|_{L^2}.
$$
In the interior domain,
$$
\int_{|X-Y|<1}\frac{|\Theta(Y)|}{|X-Y|^{1+\alpha}}\dif Y
\leq 2\pi\|\Theta\|_{L^\infty}\int_0^1\frac{R\dif R}{R^{1+\alpha}}
=\frac{2\pi}{1-\alpha}\|\Theta\|_{L^\infty},
$$
with $\Theta\in H^{1+\epsilon}\hookrightarrow L^\infty$ for $\epsilon>0$.
\end{proof}

\subsection{Proof of Proposition \ref{prop:ThetakqYm}}\label{sec:LWPY}

We begin by writing equation \eqref{eq:SQG:cor:kq} in terms of the deviation in the original coordinates
\begin{equation}\label{eq:tildetheta:evolution}
\partial_t\tilde{\theta}
+v_\epsilon\cdot\nabla\tilde{\theta}
+\tilde{v}\cdot\nabla\theta_0
=0,
\end{equation}
coupled with the initial condition \eqref{eq:Omegakqcor0}
$$
\tilde{\theta}|_{t=t_k}
=\theta^{\text{lin}}(t_k),
$$
where we have abbreviated
$$
\tilde{\theta}
=\theta^{\text{lin}}+\epsilon\theta_k^{(q)},
\quad\quad
\tilde{v}
=v^{\text{lin}}+\epsilon v_k^{(q)},
$$
and also
$$
\theta_\epsilon
=\theta_0+\epsilon \theta^{\text{lin}}+\epsilon^2 \theta_k^{(q-1)},
\quad\quad
v_\epsilon
=v_0+\epsilon v^{\text{lin}}+\epsilon^2 v_k^{(q-1)}.
$$
Notice that $\tilde{\theta}$ and $\tilde{v}$ correspond to step $q$, while $\theta_\epsilon$ and $v_\epsilon$ refer to step $q-1$. When it is not clear from the context, we will use the superscript $q$ again. 
We remark that for $b=0$, we were already in the original variables and therefore do not need to change to lowercase letters. By a slight abuse of notation, we denote both cases using lowercase letters since the proof is exactly the same.

By Remark \ref{rem:smoothbdryconditions}, the solution to \eqref{eq:SQG:cor:kq} in physical variables satisfies
\begin{equation}\label{eq:tildetheta}
\tilde{\theta}
\in C([t_k,T];H^{m+2}),
\quad\quad
\partial_t\tilde{\theta}\in L^\infty([t_k,T];H^{m+1}),
\end{equation}
for all $T>t_k$. Moreover, this solution remains compactly supported.

Next, we perform energy estimates in the weighted energy space $Y^{m+1}=Z^{m+1}$.
To avoid the singularity at $r=0$, we will carry out the energy estimates uniformly with respect to a parameter $\varepsilon>0$ that desingularizes the weight $r^{-(m+1-|K|)}$.
Given a multi-index $1\leq|K|\leq m+1$, we deduce from \eqref{eq:tildetheta:evolution} that
$$
(\partial_t+v_\epsilon\cdot\nabla)
\left(\frac{\partial^K\tilde{\theta}}{(\varepsilon+r)^{m+1-|K|}}\omega\right)
+h=0,
$$
where $h=h_1+h_2$ with
$$
h_1=
\frac{\partial^K(v_\epsilon\cdot\nabla\tilde{\theta})}{(\varepsilon+r)^{m+1-|K|}}\omega
-v_\epsilon\cdot\nabla
\left(\frac{\partial^K\tilde{\theta}}{(\varepsilon+r)^{m+1-|K|}}\omega\right),
\quad\quad
h_2=\frac{\partial^K(\tilde{v}\cdot\nabla\theta_0)}{(\varepsilon+r)^{m+1-|K|}}\omega.
$$
Since we have \eqref{eq:tildetheta} and $\varepsilon>0$, we can integrate by parts to deduce that
$$
\partial_t\left\|\frac{\partial^K\tilde{\theta}}{(\varepsilon+r)^{m+1-|K|}}\omega\right\|_{L^2}
\leq\|h\|_{L^2},
$$
where we have applied that $v_\epsilon$ is divergence free.

Concerning $h_2$, 
since $\nabla\theta_0$ vanishes outside an annulus, we have
$$
\|h_2\|_{L^2}\leq C\|\tilde{\theta}\|_{H^{m+1}},
$$
where $C$ depends on $\theta_0$, $m$, $k$, $q$, and $T$, but not on $\varepsilon$.

Concerning $h_1$, by applying
$$
\nabla\left(\frac{\omega}{(\varepsilon+r)^{m+1-|K|}}\right)
=\left(\frac{2r}{(\varepsilon+r)^{m+1-|K|}}-\frac{m+1-|K|}{(\varepsilon+r)^{m+2-|K|}}\omega\right)e_r
$$
we split $h_1=h_{1,1}+h_{1,2}$ with
$$
h_{1,1}
=\frac{\partial^K(v_\epsilon\cdot\nabla\tilde{\theta})
-v_\epsilon\cdot\nabla\partial^K\tilde{\theta}}{(\varepsilon+r)^{m+1-|K|}},
\quad\quad
h_{1,2}
=
\left(\frac{2r}{(\varepsilon+r)^{m+1-|K|}}-\frac{m+1-|K|}{(\varepsilon+r)^{m+2-|K|}}\omega\right)(v_\epsilon)_r\partial^K\tilde{\theta}.
$$
The term $h_{1,2}$ can be bounded by
$$
\|h_{1,2}\|_{L^2}
\leq C\left\|\frac{v_\epsilon}{r}\right\|_{L^\infty}
\left\|\frac{\partial^K\tilde{\theta}}{(\varepsilon+r)^{m+1-|K|}}\omega\right\|_{L^2},
$$
where $C$ only depends on $m$.
The velocity term can be bounded by $\|\theta_\epsilon\|_{H^m}$ (Lemma \ref{propABCDGMKpp}).
Concerning $h_{1,1}$, 
by applying the Leibniz rule, we get
$$
h_{1,1}
=\sum_{(0,0)\neq L\leq K}\binom{K}{L}h_{K,L},
\quad\text{with}\quad
h_{K,L}
=\partial^Lv_\epsilon\cdot\frac{\nabla\partial^{K-L}\tilde{\theta}}{(\varepsilon+r)^{m+1-|K|}}.
$$
The terms with multi-indices $0<|L|\leq m$ are bounded by
$$
\|h_{K,L}\|_{L^2}
\leq\|\partial^Lv_\epsilon\|_{L^\infty}
\left\|\frac{\nabla\partial^{K-L}\tilde{\theta}}{(\varepsilon+r)^{m+1-|K|}}\right\|_{L^2}.
$$
The first term is bounded by $\|\theta_\epsilon\|_{H^{m+2}}$. 
The second term can be bounded, by splitting the integral into the regions $|x|\leq 1-\varepsilon$ and $|x|>1-\varepsilon$, as follows
\begin{align*}
\left\|\frac{\nabla\partial^{K-L}\tilde{\theta}}{(\varepsilon+r)^{m+1-|K|}}\right\|_{L^2}
&\leq
\left\|\frac{\nabla\partial^{K-L}\tilde{\theta}}{(\varepsilon+r)^{m+1-|K-L|}}\right\|_{L^2(B_{1-\varepsilon})}
+
\|\nabla\partial^{K-L}\tilde{\theta}\|_{L^2(\R^2\setminus B_{1-\varepsilon})}.
\end{align*}
If $|L|=m+1$, necessarily $L=K$ and thus
$$
\|h_{K,K}\|_{L^2}
\leq\|\partial^Kv_\epsilon\|_{L^2}
\|\nabla\tilde{\theta}\|_{L^\infty}
\leq\|\theta_\epsilon\|_{H^{m+1}}\|\tilde{\theta}\|_{H^m}.
$$

Summing over all multi-indices $K$, we deduce that the semi-norm
$$
f_\varepsilon=
\sum_{0<|K|\leq m+1}\left\|\frac{\partial_X^K\tilde{\theta}}{(\varepsilon+r)^{m+1-|K|}}\right\|_{L^2}
$$
satisfies the bound
$$
\partial_t f_\varepsilon
\leq
\|\theta_\epsilon\|_{H^{m+2}}f_\varepsilon
+P(\|\theta_\epsilon\|_{H^{m+2}},\|\tilde{\theta}\|_{H^{m+1}}),
$$
where $P$ is a polynomial.
By applying the Gr\"onwall inequality, we deduce that
\begin{equation}\label{eq:fvarepsilon}
\begin{split}
f_\varepsilon(t)
\leq & f_\varepsilon(t_k)\exp\left(\int_{t_k}^t\|\theta_\epsilon(s)\|_{H^{m+2}}\dif s\right)\\
&+\int_{t_k}^t P(\|\theta_\epsilon\|_{H^{m+2}},\|\tilde{\theta}\|_{H^{m+1}})(s)\exp\left(\int_{s}^t\|\theta_\epsilon(t')\|_{H^{m+2}}\dif t'\right)\dif s.
\end{split}
\end{equation}
Finally, by applying Lemma \ref{lemma:thetalinYm}
$$
f_\varepsilon(t_k)\leq f_0(t_k)\leq\|\theta^{\text{lin}}(t_k)\|_{Y^{m+1}},
$$
we deduce, by passing to the limit $\varepsilon\to 0$ in \eqref{eq:fvarepsilon}, that
$$
\tilde{\theta}\in L^\infty([t_k,T];Y^{m+1})
$$
for all $T>t_k$. Therefore, by applying the Aubin-Lions lemma (recall Lemma \ref{lemma:Ymcompact}) and that $\tilde{\theta}$ is compactly supported, we deduce that
$$
\tilde{\theta}\in C([t_k,T];Y^{m}),
$$
for all $T>t_k$.

\subsection{Baseline $L^2$ estimate}\label{eq:baselineL2}
In this section, we improve the exponential decay of $\|\Theta^{(q)}_k\|_{L^2}$. In the next sections, we will use this baseline estimate to address the $Y^m$-norm for the proof of Lemma~\ref{lemma:non-lineark}.

\begin{lemma}\label{lemma:irstinequality}
There exists $C>0$ such that
\begin{equation}\label{Lemmafirstinequality}
\Vert\Theta^{(q)}_k(\tau)\Vert_{L^2}\leq C e^{2\Re\lambda\tau},
\end{equation}
for all $-k\leq \tau\leq\min_{p\leq q}\tau_k^{(p)}$.
\end{lemma}
\begin{proof}
The solution to \eqref{eq:SQG:cor:kq} satisfies the Duhamel formula
$$
\Theta_k^{(q)}(\tau)
=-\int_{-k}^{\tau}e^{(\tau-\tau')L_b}\mathcal{F}_k^{(q)}(\tau')\dif \tau'.
$$
By applying \eqref{eq:ZminHm} and \eqref{eq:Vbounded}, we estimate
$$
\|\mathcal{F}_k^{(q)}\|_{L^2}
\leq
\|V^{\text{lin}}+\epsilon V_k^{(q-1)}\|_{L^\infty}
\|\nabla(\Theta^{\text{lin}}+\epsilon \Theta_k^{(q)})\|_{L^2}
\leq C\|\Theta^{\text{lin}}+\epsilon \Theta_k^{(q)}\|_{Y^m}^2.
$$
Thus, by applying Lemma \ref{lemma:thetalinYm} for 
$\Theta^{\text{lin}}$, and the bound \eqref{kdependentbound} for $\Theta_k^{(q)}$, we get
$$
\|\mathcal{F}_k^{(q)}(\tau)\|_{L^2}
\leq  Ce^{{2\Re\lambda\tau}},$$
for all $-k\leq \tau\leq\min_{p\leq q}\tau_k^{(p)}$.
Finally, by Proposition \ref{prop:SpectralAnalysis}, for some fixed $0<\delta\leq\frac{\Re\lambda}{2}$ and $C=C_\delta\geq 1$, we have
\begin{equation}
\Vert\Theta^{(q)}_k(\tau)\Vert_{L^2}\leq C e^{(\Re\lambda+\delta)\tau}\int_{-k}^{\tau} e^{{(\Re\lambda-\delta)s}}\dif s\leq C e^{{2\Re\lambda\tau}},
\end{equation}
for all $-k\leq \tau\leq\min_{p\leq q}\tau_k^{(p)}$.
\end{proof}

\subsection{Inductive energy estimates}\label{sec:energyestimates}

In this section we deal with the weighted norms appearing in the $Y^m$-norm for the proof of Lemma~\ref{lemma:non-lineark}. We start by rewriting \eqref{eq:SQG:cor:kq} as
\begin{equation}\label{eq:corrector}
\partial_\tau\Theta
-b(a-\alpha)\Theta+(\bar{V}-bX+\tilde{V})\cdot\nabla\Theta
+V\cdot\nabla\bar{\Theta}+\tilde{V}\cdot\nabla\Theta^{\text{lin}}=0,
\end{equation}
where we have abbreviated
\begin{equation}\label{eq:ThetaVkq}
\Theta=\Theta_k^{(q)},
\quad\quad
V=V_k^{(q)},
\end{equation}
and also
\begin{equation}\label{eq:Vkq-1}
\tilde{\Theta}=\epsilon\Theta^{\text{lin}}+\epsilon^2 \Theta_k^{(q-1)},
\quad\quad
\tilde{V}=\epsilon V^{\text{lin}}+\epsilon^2 V_k^{(q-1)}.
\end{equation}
Notice that $\Theta$ and $V$ correspond to step $q$, while $\tilde{\Theta}$ and $\tilde{V}$ refer to step $q-1$. When it is not clear from the context, we will use the superscript $q$ again.

We start by proving that it is possible to improve the bound of the weighted $L^2$ norm in $Y^m$ by using the baseline estimate in $L^2$.

\begin{lemma}\label{lemma:irstinequality:weight}
There exists $C>0$ such that
\begin{equation}\label{Lemmafirstinequality:weight}
\Vert\Theta(\tau)\omega\Vert_{L^2}\leq C e^{2\Re\lambda\tau},
\end{equation}
for all $-k\leq \tau\leq\min_{p\leq q}\tau_k^{(p)}$.
\end{lemma}
\begin{proof}
By applying that $X=Re_R$, $\bar{V}=\bar{V}_\phi e_\phi$ and $\nabla\omega=2Re_R$, we deduce that
\begin{align*}
(\bar{V}-bX+\tilde{V})\cdot\nabla\Theta\omega
-(\bar{V}-bX+\tilde{V})\cdot\nabla(\Theta\omega)
&=(\bar{V}-bX+\tilde{V})\cdot\nabla\omega\Theta
=(-bR+\tilde{V}_R)2R\Theta.
\end{align*}
Hence, multiplying \eqref{eq:corrector} by $\omega$, 
we deduce that
\begin{equation}\label{eq:corrector:weight}
\partial_\tau\Theta\omega
-b\left(a-\alpha+\frac{2R^2}{\omega}\right)\Theta\omega
+(\bar{V}-bX+\tilde{V})\cdot\nabla(\Theta\omega)+\mathcal{H}=0,
\end{equation}
where $\mathcal{H}=\bar{\mathcal{H}}+\mathcal{H}_1+\mathcal{H}_2$ with
$$
\bar{\mathcal{H}}
=V\cdot\nabla\bar{\Theta}\omega,
\quad\quad
\mathcal{H}_1
=\frac{\tilde{V}_R}{R}\frac{2R^2}{\omega}\Theta\omega,
\quad\quad
\mathcal{H}_2
=\tilde{V}\cdot\nabla\Theta^{\text{lin}}\omega.
$$
Multiplying \eqref{eq:corrector:weight} by 
$\Theta\omega$
and integrating over $\R^2$, we deduce that
\begin{align*}
\frac{1}{2}\partial_\tau\|\Theta\omega\|_{L^2}^2
&\leq b(2+a-(1+\alpha))\|\Theta\omega\|_{L^2}^2
+\left|\int_{\R^2}\mathcal{H}\Theta\omega\dif X\right|\\
&\leq 2b\|\Theta\omega\|_{L^2}^2
+\|\mathcal{H}\|_{L^2}\|\Theta\omega\|_{L^2},
\end{align*}
where we have applied Proposition \ref{prop:ThetakqYm}, that the velocities are divergence free, $\mathrm{div}X=2$, and that
$0<a<\varepsilon<1+\alpha$ by \eqref{eq:a}. Therefore, the Gr\"onwall inequality implies that
$$
\|\Theta\omega\|_{L^2}
\leq
\int_{-k}^{\tau}
e^{2b(\tau-\tau')}\|\mathcal{H}(\tau')\|_{L^2}\dif\tau'.
$$
For the first term, since $\bar{\Theta}$ is smooth and supported in a ball of radius $\bar{R}=r_2+\varepsilon$, by applying Lemmas \ref{lemma:VDVL2}
and \ref{lemma:irstinequality}, we get
$$
\|\bar{\mathcal{H}}\|_{L^2}
\leq\|\nabla\bar{\Theta}\omega\|_{L^\infty}\|V\|_{L^2(B_{\bar{R}})}
\lesssim \|\Theta\|_{L^2}
\lesssim e^{2\Re\lambda\tau}.
$$
For the quadratic term $\mathcal{H}_1$, by applying Lemmas \ref{propABCDGMKpp}, \ref{lemma:YminweightedHm} and \ref{lemma:Marcos}, and the bound \eqref{kdependentbound}, we get
$$
\|\mathcal{H}_1\|_{L^2}
\leq 2\left\|\frac{\tilde{V}_R}{R}\right\|_{L^\infty}\|\Theta\omega\|_{L^2}
\lesssim \|\Theta\|_{Y^m}^2
\lesssim e^{2\Re\lambda\tau}.
$$
Similarly for the other quadratic term $\mathcal{H}_2$, by applying also \eqref{eq:Vbounded} and Lemma \ref{lemma:thetalinYm}, we get
$$
\|\mathcal{H}_2\|_{L^2}
\lesssim\|\tilde{V}\|_{L^\infty}
\|\nabla\Theta^{\text{lin}}\omega\|_{L^2}
\lesssim
e^{2\Re\lambda\tau}.
$$
We conclude the proof by applying \eqref{eq:b<=}.
\end{proof}

We continue by improving the bound of the weighted partial derivatives in $Y^m$.
In general, we want to estimate terms of the form
$$
\left\|\frac{\partial^J\Theta}{R^{m-j_1}}\omega \right\|_{L^2},
$$
where $J=(j_1,j_2)$ is a multi-index and $\partial^J$ is the corresponding partial derivative in polar coordinates
$$
\partial^J
=\partial_R^{j_1}\partial_\phi^{j_2}.
$$
To simplify the notation, we will omit the subscript $(R,\phi)$ when working with partial derivatives in polar coordinates. The subscript $X$ will be used for partial derivatives in Cartesian coordinates to distinguish them from the previous case.

In the next proposition, we prove an energy estimate that holds for all multi-indices $J$.

\begin{prop}[Energy estimate]\label{prop:generalenergyestimate}
For any
multi-index $J$ with $|J|>0$, it holds that
$$
\left\|\frac{\partial^J\Theta(\tau)}{R^{m-j_1}}\omega \right\|_{L^2}^2
\leq 2\int_{-k}^{\tau}e^{2b(m+2)(\tau-\tau')}
\left|\int\mathcal{H}_{J}\frac{\partial^J\Theta}{R^{m-j_1}}\omega\dif X\right|(\tau')\dif\tau',
$$
for all $-k\leq\tau\leq\tau_k^{(q)}$, 
where we split
$$
\mathcal{H}_{J}
=\mathcal{H}_J^{\text{rad}}
+\mathcal{H}_{J}^{\text{com}}
+\mathcal{H}_J^{\text{quad}},
$$
into the radial term
$$
\mathcal{H}_J^{\text{rad}}
=\frac{\omega }{R^{m-j_1}}\sum_{0<i_1\leq j_1}\binom{j_1}{i_1}
\partial^{i_1}_R\left(\frac{\bar{V}_\phi}{R}\right)\partial^{j_1-i_1}_R\partial_\phi^{j_2+1}\Theta,
$$
the commutator term
$$
\mathcal{H}_{J}^{\text{com}}
=\frac{\partial^J(V\cdot\nabla\bar{\Theta})}{R^{m-j_1}}\omega ,
$$
and the quadratic term
$$
\mathcal{H}_J^{quad}
=\frac{\partial^J(\tilde{V}\cdot\nabla\Theta)}{R^{m-j_1}}\omega 
-\tilde{V}\cdot\nabla\left(\frac{\partial^J\Theta}{R^{m-j_1}}\omega \right)
+\frac{\partial^J(\tilde{V}\cdot\nabla\Theta^{\text{lin}})}{R^{m-j_1}}\omega.
$$
\end{prop}

\begin{proof}
Firstly, recall that
$$
X\cdot\nabla\Theta
=R\partial_R\Theta.
$$
By applying the Leibniz rule on the radial variable,
we get
\begin{equation}\label{eq:partialJ}
\partial^J(X\cdot\nabla\Theta)
=\sum_{i_1\leq j_1}\binom{j_1}{i_1}\partial_R^{i_1}R\partial_R^{j_1-i_1+1}\partial_\phi^{j_2}\Theta
=X\cdot\nabla\partial^J\Theta
+j_1\partial^J\Theta.
\end{equation}
The first term corresponds to $i_1=0$, and the second to $i_1=1$ (which equals zero when $j_1=0$). The remainder terms vanish because $\partial_R^{i_1}R=0$ for $i_1>1$.

The identity
\begin{equation}\label{eq:nablaR}
\nabla\left(\frac{1}{R^{m-j_1}}\omega\right)
=\left(\frac{2R}{R^{m-j_1}}-\frac{m-j_1}{R^{m-j_1+1}}\omega\right)e_R,
\end{equation}
combined with  \eqref{eq:partialJ}, implies that
$$
\frac{\partial^J(X\cdot\nabla\Theta)}{R^{m-j_1}}\omega
=X\cdot\nabla\left(\frac{\partial^J\Theta}{R^{m-j_1}}\omega\right)
+\left(m-\frac{2R^2}{\omega}\right)\frac{\partial^J\Theta}{R^{m-j_1}}\omega.
$$

Secondly, recall that
$$
\bar{V}\cdot\nabla\Theta
=\frac{\bar{V}_\phi}{R}\partial_\phi\Theta.
$$
By applying again the Leibniz rule on the radial variable, we get
\begin{align*}
\partial^J(\bar{V}\cdot\nabla\Theta)
=\sum_{i_1\leq j_1}
\binom{j_1}{i_1}\partial_R^{i_1}
\left(\frac{\bar{V}_\phi}{R}\right)
\partial_R^{j_1-i_1}\partial_\phi^{j_1+1}\Theta.
\end{align*}
Therefore, splitting into the term with $i_1=0$ and the terms with $i_1>0$, we get
$$
\frac{\partial^J(\bar{V}\cdot\nabla\Theta)}{R^{m-j_1}}\omega
=\bar{V}\cdot\nabla\left(\frac{\partial^J\Theta}{R^{m-j_1}}\omega\right)
+\mathcal{H}_J^{\text{rad}}.
$$

Therefore, we deduce from \eqref{eq:corrector} that the following equation holds
$$
\partial_\tau\left(\frac{\partial^J\Theta}{R^{m-j_1}}\omega\right)
-b\left(m-\frac{2R^2}{\omega}+a-\alpha\right)\frac{\partial^J\Theta}{R^{m-j_1}}\omega
+(\bar{V}-bX+\tilde{V})\cdot\nabla\left(\frac{\partial^J\Theta}{R^{m-j_1}}\omega\right)
+\mathcal{H}_J=0.
$$
Multiplying it by $\frac{\partial^J\Theta}{R^{m-j_1}}\omega$
and integrating over $\R^2$, we deduce that
$$
\frac{1}{2}\partial_\tau\left\|\frac{\partial^J\Theta}{R^{m-j_1}}\omega\right\|_{L^2}^2
\leq b(m+2+a-(1+\alpha))\left\|\frac{\partial^J\Theta}{R^{m-j_1}}\omega\right\|_{L^2}^2
+\left|\int_{\R^2}\mathcal{H}_J\frac{\partial^J\Theta}{R^{m-j_1}}\omega\dif X\right|,
$$
where we have applied Proposition \ref{prop:ThetakqYm}, that the velocities are divergence free and that $\mathrm{div}X=2$.
Finally, recall that $0<a<\varepsilon<1+\alpha$ by \eqref{eq:a}.
\end{proof}


In addition to the standard notation $I \le J$ for $i_1\le j_1$ and $i_2 \le j_2$, we define the following well-ordering of multi-indices:
$$I\prec J \ \ \Longleftrightarrow \ \ |I|<|J|, \ \text{ or } \ |I|=|J| \text{ and } i_1<j_1.$$
In fact, it is possible to number these multi-indices by
$$
\text{ind}(J):=\frac{|J|(|J|+1)}{2}+j_1,
$$
that is, $I\prec J$ if and only if $\text{ind}(I)<\text{ind}(J)$.
We define
$$
c_J:=2^{-\text{ind}(J)}.
$$

\begin{prop}[Inductive energy  estimate]\label{DerivativesboundedImpliesHbounded}
Let $J$ be a multi-index with $0<|J|\leq m$. Assume that for all $(0,0)\neq I\prec J$ there exists $C_I>0$ such that 
\begin{equation}\label{eq:Derivativesbounded}
\left\|\frac{\partial^I\Theta(\tau)}{R^{m-i_1}}\omega\right\|_{L^2}
\leq C_I e^{(1+c_I)\Re\lambda\tau},
\end{equation}
for all $-k\leq\tau\leq\min_{p\leq q}\tau_k^{(p)}$.
Then, there exists $D_J>0$ such that
\begin{equation}\label{eq:Hbounded}
\left|\int\mathcal{H}_{J}\frac{\partial^J\Theta}{R^{m-j_1}}\omega\dif X\right|(\tau)
\leq D_J e^{2(1+c_J)\Re\lambda\tau},
\end{equation}
for all $-k\leq\tau\leq\min_{p\leq q}\tau_k^{(p)}$. As a consequence, the bound \eqref{eq:Derivativesbounded} is also satisfied for $J$. In particular, no assumption is needed to obtain \eqref{eq:Derivativesbounded} for $J=(0,1)$.
\end{prop}
\begin{proof}
We prove the bound \eqref{eq:Hbounded} for $\mathcal{H}^{\text{rad}}$, $\mathcal{H}^{\text{com}}$ and $\mathcal{H}^{\text{quad}}$ in Sections \ref{sec:Hrad}, \ref{sec:Hcom} and \ref{sec:Hquad}, respectively. Assuming that we have already proven \eqref{eq:Hbounded}, we show that
\eqref{eq:Derivativesbounded} is also satisfied for $J$.
By applying Proposition \ref{prop:generalenergyestimate} and the bound \eqref{eq:Hbounded}, together with \eqref{eq:b<=},
we get
\begin{equation}\label{eq:Derivativesbounded:1}
\begin{split}
\left\|\frac{\partial^J\Theta}{R^{m-j_1}}\omega\right\|_{L^2}^2
&\leq
D_J\int_{-k}^{\tau}
e^{2b(m+2)(\tau-\tau')+2(1+c_J)\Re\lambda\tau'}\dif\tau'\\
&\leq
\frac{1}{2}\frac{D_J}{(1+c_J)\Re\lambda-b(m+2)}
e^{2(1+c_J)\Re\lambda\tau},
\end{split}
\end{equation}
for all $-k\leq\tau\leq\min_{p\leq q}\tau_k^{(p)}$.
\end{proof}

\subsection{Proof of Lemma \ref{lemma:non-lineark}}
We apply the inductive energy estimate (Proposition  \ref{DerivativesboundedImpliesHbounded}) on the multi-indices $J$.  
Notice that for the first multi-index $J=(0,1)$ there are no $(0,0)\neq I\prec J$, thus Proposition  \ref{DerivativesboundedImpliesHbounded} immediately implies \eqref{eq:Derivativesbounded} for $J=(0,1)$. 
We now make the induction hypothesis. 
Given a multi-index $J\succ (0,1)$, suppose that \eqref{eq:Derivativesbounded} is satisfied for any multi-index $I$ with $I\prec J$. Then, we can apply Proposition \ref{DerivativesboundedImpliesHbounded} to deduce that \eqref{eq:Derivativesbounded} is also satisfied for $J$. Once we have reached the last multi-index $J=(m,0)$, we have proven that \eqref{eq:Derivativesbounded} is satisfied for any multi-index $J$ with $0<|J|\leq m$. 
Since $c_J$ is decreasing with respect to the ordering $\prec$, we deduce \eqref{eq:non-lineark} and thus Lemma \ref{lemma:non-lineark} as well.
Recall that the $L^2$ part of the $Y^m$ norm was estimated in Lemma \ref{lemma:irstinequality:weight}.

\subsection{Radial term}\label{sec:Hrad}
In this section we prove the bound \eqref{eq:Hbounded} for
\begin{align*}
\mathcal{H}_{J}^{\text{rad}}
&=\sum_{0<i_1\leq j_1}\binom{j_1}{i_1}
R^{i_1}\partial^{i_1}_R\left(\frac{\bar{V}_\phi}{R}\right)\frac{\partial^{j_1-i_1}_R\partial_\phi^{j_2+1}\Theta}{R^{m-(j_1-i_1)}}\omega.
\end{align*}
Recall that $\Theta$ corresponds to step $q$ according to \eqref{eq:ThetaVkq}.
Notice that $\mathcal{H}_{J}^{\text{rad}}=0$ if $j_1=0$. Otherwise, since $(j_1-i_1,j_2+1)\prec J$ for every $0<i_1\leq j_1$ we can apply the assumption \eqref{eq:Derivativesbounded} which, combined with Lemma \ref{VbarDecay}, yields
\begin{align*}
\big\|\mathcal{H}_{J}^{\text{rad}}\big\|_{L^2}
&\lesssim \sum_{0<i_1\leq j_1}\left\|R^{i_1}\partial^{i_1}_R\left(\frac{\bar{V}_\phi}{R}\right)\right\|_{L^\infty}\left\|\frac{\partial^{j_1-i_1}_R\partial_\phi^{j_2+1}\Theta}{R^{m-(j_1-i_1)}}\omega \right\|_{L^2}\\
&\lesssim \sum_{0<i_1\leq j_1}e^{(1+c_{(j_1-i_1,j_2+1)})\Re\lambda\tau}\lesssim e^{(1+2c_{J})\Re\lambda\tau},
\end{align*}
where we used that $c_L\geq 2c_J$ whenever $L\prec J$. 
Finally, by applying \eqref{kdependentbound}, we get
$$
\left|\int\mathcal{H}_{J}^{\text{rad}}\frac{\partial^J\Theta}{R^{m-j_1}}\omega\dif X\right|
\leq
\|\mathcal{H}_{J}^{\text{rad}}\|_{L^2}
\|\Theta\|_{Y^m}
\lesssim e^{2(1+c_J)\Re\lambda\tau}.
$$

\subsection{Commutator term}\label{sec:Hcom}
In this section we prove the bound \eqref{eq:Hbounded} for
$\mathcal{H}_{J}^{\text{com}}$.
Recall that $\Theta$ and $V$ are both at step $q$ according to \eqref{eq:ThetaVkq}.
We have to estimate the integral
$$
A=\int \frac{\partial^J(V\cdot\nabla\bar{\Theta})}{R^{m-j_1}}\frac{\partial^J\Theta}{R^{m-j_1}}\omega^2 .
$$
By rewriting polar into Cartesian coordinates (Lemma \ref{lemma:polarcartesian}) we get
\begin{align*}
A
&=\int \frac{\omega^2}{R^{2(m-j_1)}}
\Bigg\lbrace
\sum_{\substack{0<|K|\leq|J| \\ j_1\leq|K|}}
q_K^J R^{|K|-j_1}\partial_X^K(V\cdot\nabla\bar{\Theta})\Bigg\rbrace
\Bigg\lbrace
\sum_{\substack{0<|L|\leq|J| \\ j_1\leq|L|}}
q_L^J R^{|L|-j_1}\partial_X^L\Theta\Bigg\rbrace\\
&=\sum_{\substack{0<|K|\leq|J| \\ j_1\leq|K|}}
\sum_{\substack{0<|L|\leq|J| \\ j_1\leq|L|}}
\int \frac{\omega^2 q_K^J q_L^J}{R^{2m-(|K|+|L|)}}
\partial_X^K(V\cdot\nabla\bar{\Theta})\partial_X^L\Theta.
\end{align*}
We remark that the weights $\frac{\omega^2}{R^{2m-(|K|+|L|)}}$ are negligible because $\nabla\bar{\Theta}$ is smooth and supported in an annulus.
The terms for which either $|K|$ or $|L|$ are less than $|J|$ can be bounded easily. 
By applying the Leibniz rule in the term $\partial_X^K(V\cdot\nabla\bar{\Theta})$, all the terms for which some derivatives hit $\nabla\bar{\Theta}$ can also be bounded easily. 
Taking into account all these observations we can decompose 
$$
A=B+G,
$$
into a bad term
$$
B
=\sum_{|K|=|L|=|J|}
\int \frac{\omega^2 q_K^J q_L^J}{R^{2m-(|K|+|L|)}} (\partial_X^KV\cdot\nabla\bar{\Theta})\partial_X^L\Theta,
$$
and a good term $G$ satisfying the bound
\begin{equation}\label{eq:G}
|G|\lesssim\|\Theta\|_{H^{|J|-1}}\|\Theta\|_{H^{|J|}}.
\end{equation}
We remark that we have used that the $L^2$-norm of the velocity in $B_{\bar{R}}$ with $\bar{R}=r_2+\varepsilon$ is bounded by $\|\Theta\|_{L^2}$ (Lemma \ref{lemma:VDVL2}). To estimate $B$, it is crucial that the function
$$
H_{K,L}^J:=\frac{\omega^2 q_K^J q_L^J}{R^{2m-(|K|+|L|)}}\nabla\bar{\Theta} \in C_c^\infty,
$$
and that $H_{L,K}^J=H_{K,L}^J$.
By writing 
$V=-R^\perp_\alpha\Theta$ (denote $R_\alpha=\nabla(-\Delta)^{\nicefrac{\alpha}{2}-1}$),
we can express $B$ by,
\begin{align*}
B
=&-\sum_{|K|=|L|=|J|}
\int R^\perp_\alpha(\partial_X^K\Theta)\cdot (H_{K,L}^J\partial_X^L\Theta)\\
=&\sum_{|K|=|L|=|J|}
\int (\partial_X^K\Theta) \cdot R^\perp_\alpha (H_{K,L}^J\partial_X^L\Theta)\\
=&\sum_{|K|=|L|=|J|}
\int (H_{K,L}^J\partial_X^K\Theta)\cdot R^\perp_\alpha(\partial_X^L\Theta)\\
&+\sum_{|K|=|L|=|J|}
\int \partial_X^K\Theta
\left\lbrace R^\perp_\alpha  (H_{K,L}^J\partial_X^L\Theta)
-H_{K,L}^J\cdot R^\perp_\alpha(\partial_X^L\Theta)
\right\rbrace.
\end{align*}
where, for the second equality, we have used that $R_\alpha^{\perp}$ is skew adjoint and in the last equality we have added and subtracted $\int (H_{K,L}^J\partial_X^K\Theta)\cdot R^\perp_\alpha(\partial_X^L\Theta)$. Notice that, by exchanging $K$ and $L$, it follows  that the first term equals $-B$. Therefore, $2B$ is indeed equal to the last term. It is easier to express it in terms of a commutator. That is, 
$$
B=\frac{1}{2}
\sum_{|K|=|L|=|J|}
\int \partial_X^K\Theta\left[R^\perp_\alpha,H_{K,L}^J\right]\partial_X^L\Theta.
$$
As a consequence, using  the smoothing properties of the commutator, we obtain that (see e.g.~\cite{Steinbook})
\begin{equation}\label{eq:B}
|B|
\leq
C\|\Theta\|_{H^{|J|-1}}\|\Theta\|_{H^{|J|}},
\end{equation}
where $C$ only depends on $m$ and $\bar{\Theta}$. Finally, recall that by \eqref{eq:ZminHm}, $Y^m\hookrightarrow H^m$. Therefore, we can use  \eqref{eq:G} \eqref{eq:B} to obtain the first inequality below. For the second inequality we use the hypothesis \eqref{eq:Derivativesbounded} (the weights are irrelevant because of the support of $\Theta$) together with \eqref{kdependentbound}. Thus,
$$
|A|\lesssim\|\Theta\|_{H^{|J|-1}}\|\Theta\|_{H^{|J|}}
\lesssim e^{(1+c_{(0,|J|-1)})\Re\lambda\tau}e^{\Re\lambda\tau}
\leq e^{2(1+c_J)\Re\lambda\tau},
$$
and the term $\mathcal{H}_{J}^{\text{com}}$ is dealt with. 

\subsection{Quadratic term}\label{sec:Hquad}
In this section we prove the bound \eqref{eq:Hbounded} for
$\mathcal{H}_{J}^{\text{quad}}$. It is a corollary of the following lemma. 

\begin{lemma}
It holds that
$$
\|\mathcal{H}_J^{\text{quad}}\|_{L^2}
\lesssim e^{2\Re\lambda\tau},
$$
for all $-k\leq\tau\leq\min_{p\leq q}\tau_k^{(p)}$.
\end{lemma}

We will prove this lemma in the next subsections.
We split 
$$
\mathcal{H}_J^{\text{quad}}
=\mathcal{H}_J^{\text{quad},1}
+\mathcal{H}_J^{\text{quad},2}
+\mathcal{H}_J^{\text{quad},3}
$$
where
\begin{align*}
\mathcal{H}_J^{\text{quad},1}
&=\frac{\partial^J(\tilde{V}\cdot\nabla\Theta)}{R^{m-j_1}}\omega
-\tilde{V}\cdot\nabla\left(\frac{\partial^J\Theta}{R^{m-j_1}}\omega\right),\\
\mathcal{H}_J^{\text{quad},2}
&=\frac{\partial^J(\tilde{V}\cdot\nabla\Theta^{\text{lin}})}{R^{m-j_1}}\omega
-\tilde{V}\cdot\nabla\left(\frac{\partial^J\Theta^{\text{lin}}}{R^{m-j_1}}\omega\right),\\
\mathcal{H}_J^{\text{quad},3}
&=\tilde{V}\cdot\nabla\left(\frac{\partial^J\Theta^{\text{lin}}}{R^{m-j_1}}\omega\right).
\end{align*}
Recall that $\Theta$ corresponds to step $q$, and $\tilde{V}$ corresponds to step $q-1$, according to \eqref{eq:ThetaVkq}.

\subsubsection{Quadratic term 1}
\begin{lemma} It holds that
\begin{equation}\label{eq:Hquad1}
\begin{split}
\mathcal{H}_J^{\text{quad},1}
=&\frac{\omega}{R^{m-j_1}}\sum_{(0,0)\neq I\leq J}\binom{J}{I}
\left(
\partial^{I}\tilde{V}_R\partial^{J-I}\partial_R\Theta
+\partial^{I}\left(\frac{\tilde{V}_\phi}{R}\right)\partial^{J-I}\partial_\phi\Theta\right)\\
&+\left((m-j_1)+\frac{2R^2}{\omega}\right)\frac{\tilde{V}_R}{R}\frac{\partial^J\Theta}{R^{m-j_1}}\omega.
\end{split}
\end{equation}
\end{lemma}
\begin{proof}
Notice that since
$$
\tilde{V}\cdot\nabla\Theta
=\tilde{V}_R\partial_R\Theta
+\frac{\tilde{V}_\phi}{R}\partial_\phi\Theta,
$$
by applying the Leibniz rule, we get
that $$
\partial^J(\tilde{V}\cdot\nabla\Theta)
=\sum_{I\leq J}\binom{J}{I}
\left(\partial^{I}\tilde{V}_R\partial^{J-I}\partial_R\Theta
+\partial^{I}\left(\frac{\tilde{V}_\phi}{R}\right)\partial^{J-I}\partial_\phi\Theta\right).
$$
We split this sum into the main term ($I=0$) and the reminder ($I>0$)
$$
\partial^J(\tilde{V}\cdot\nabla\Theta)
=\tilde{V}\cdot\nabla\partial^J\Theta
+\sum_{(0,0)\neq I\leq J}\binom{J}{I}
\left(\partial^{I}\tilde{V}_R\partial^{J-I}\partial_R\Theta
+\partial^{I}\left(\frac{\tilde{V}_\phi}{R}\right)\partial^{J-I}\partial_\phi\Theta\right).
$$
Multiplying both sides by $\omega$ and introducing the weight into the derivative of the main term  (and absorbing the extra terms in the remainder), we get 
$$
\frac{\omega}{R^{m-j_1}}(\tilde{V}\cdot\nabla\partial^J\Theta)
=\tilde{V}\cdot\nabla\left(\frac{\partial^J\Theta}{R^{m-j_1}}\omega\right)
+\left((m-j_1)+\frac{2R^2}{\omega}\right)
\frac{\tilde{V}_R}{R}\frac{\partial^J\Theta}{R^{m-j_1}}\omega,
$$
as desired. 
\end{proof}

\begin{lemma}
It holds that
$$
\|\mathcal{H}_J^{\text{quad},1}\|_{L^2}
\lesssim e^{2\Re\lambda\tau}.
$$
\end{lemma}
\begin{proof}
We want to bound the $L^2$ norm of the different terms appearing in \eqref{eq:Hquad1}.
By Lemma \ref{propABCDGMKpp} and \eqref{kdependentbound} (which will be used repeatedly below), the last term in \eqref{eq:Hquad1} can be bounded easily
$$
\left\|\frac{\tilde{V}_R}{R}\right\|_{L^\infty}
\left\|\frac{\partial^J\Theta}{R^{m-j_1}}\omega\right\|_{L^2}
\lesssim\|\Theta\|_{Y^m}^2
\lesssim e^{2\Re\lambda\tau}.
$$

Next, we proceed to bound the terms in \eqref{eq:Hquad1} of the form
\begin{equation}\label{eq:Hquad1:1}
\frac{\omega}{R^{m-j_1}}\partial^{I}\tilde{V}_R\partial^{J-I}\partial_R\Theta,
\end{equation}
with $(0,0)\neq I\leq J$. By applying the Leibniz rule on $\partial^I(\tilde{V}\cdot e_R)$, we realize that, in order to bound \eqref{eq:Hquad1:1}, it is enough to control terms of the form
\begin{equation}\label{eq:Hquad1:2}
\frac{\omega}{R^{m-j_1}}\partial^{\tilde{I}}\tilde{V}\partial^{J-I}\partial_R\Theta,
\end{equation}
for any $(0,0)\neq \tilde{I}\leq I$.
If $|\tilde{I}|=0$, we distinguish two cases: $i_1=0$ and $i_1>0$. If $i_1=0$ then
$$
\left\|\frac{\tilde{V}}{R}\right\|_{L^\infty}
\left\|\frac{\partial^{J-I}\partial_R\Theta}{R^{m-(j_1+1)}}\omega\right\|_{L^2}
\lesssim\|\Theta\|_{Y^m}^2
\lesssim e^{2\Re\lambda\tau}.
$$
If $i_1>0$ then we use \eqref{eq:Vbounded} in the first inequality  and Lemma \ref{lemma:Marcos} in the second to obtain
$$
\|\tilde{V}\|_{L^\infty}
\left\|\frac{\partial^{J-I}\partial_R\Theta}{R^{m-j_1}}\omega\right\|_{L^2}\lesssim\|\Theta\|_{\bar{Y}^m}^2
\lesssim\|\Theta\|_{Y^m}^2
\lesssim e^{2\Re\lambda\tau}.
$$
Here we have taken $r=j_1$ in the definition of $\bar{Y}_m$ since $|J|-|I|+1\leq r+(j_2-i_2)\leq m$.

For $|\tilde{I}|>0$, by rewriting polar into Cartesian coordinates (Lemma \ref{lemma:polarcartesian}), we expand
\begin{align*}
\eqref{eq:Hquad1:2}
&=\frac{\omega}{R^{m-j_1}}
\Bigg\lbrace\sum_{\substack{0<|K|\leq|\tilde{I}| \\ \tilde{i}_1\leq|K|}}
q_K^{\tilde{I}} R^{|K|-\tilde{i}_1}\partial_X^K\tilde{V}\Bigg\rbrace
\Bigg\lbrace\sum_{\substack{0<|L|\leq|J-I|+1 \\ j_1-i_1+1\leq|L|}}
q_L^{J-I+(1,0)} R^{|L|-(j_1-i_1+1)}\partial_X^L\Theta\Bigg\rbrace\\
&=\sum_{\substack{0<|K|\leq|\tilde{I}| \\ \tilde{i}_1\leq|K|}}
\sum_{\substack{0<|L|\leq|J-I|+1 \\ j_1-i_1+1\leq|L|}}
q_K^{\tilde{I}} q_L^{J-I+(1,0)}
\partial_X^K\tilde{V}
\frac{\partial_X^L\Theta}{R^{m-(i_1-\tilde{i}_1+|L+K|-1)}}\omega.
\end{align*}
If $0<|K|\leq m-2$, then by Lemma \ref{lemma:Marcos} and Sobolev embedding, we can bound the corresponding term by
$$
\|\partial_X^K\tilde{V}\|_{L^\infty}
\left\|\frac{\partial_X^L\Theta}{R^{m-(i_1-\tilde{i}_1+|L+K|-1)}}\omega\right\|_{L^2}\lesssim \|\partial_X^K\tilde{V}\|_{H^2}\|\Theta\|_{\bar{Z}^m}\lesssim \|\Theta\|_{Y^m}^2
\lesssim e^{2\Re\lambda\tau},
$$
because $|L|\leq i_1-\tilde{i}_1+|L+K|-1\leq|J|\leq m$.

If $m-1\leq|K|\leq m$, we instead bound the corresponding term by
$$
\|\partial_X^K\tilde{V}\|_{L^2}
\left\|\frac{\partial_X^L\Theta}{R^{m-(i_1-\tilde{i}_1+|L+K|-1)}}\omega\right\|_{L^\infty}\lesssim e^{\Re\lambda\tau}
\left\|\frac{\partial_X^L\Theta}{R^{m-(i_1-\tilde{i}_1+|L+K|-1)}}\omega\right\|_{L^\infty}.
$$
We estimate the remaining term
$$\left\|\frac{\partial_X^L\Theta}{R^{m-(i_1-\tilde{i}_1+|L+K|-1)}}\omega\right\|_{L^\infty}
$$
by considering separate cases. In the following arguments, we rely on the inequalities: $\tilde{I}\leq I\leq J$, $|K|\leq |\tilde{I}|$, $0<|L|\leq|J-I|+1$ and $|J|\leq m$. If $|K|=m$ then necessarily $\tilde{I}=I=J$ and $|L|=1$, so that $m-(i_1-\tilde{i}_1+|L+K|-1)=0$. Hence, in order to control such term  we have to estimate  $D_X\Theta\,\omega$ in $L^\infty$.
Suppose now that $|K|=m-1$, then we have that either $|L|=1$ or $|L|=2$. If $|L|=2$ then $|I|=m-1$ and thus $\tilde{I}=I$, so $m-(i_1-\tilde{i}_1+|L+K|-1)=0$ and therefore for this we have to estimate $D_X^2\Theta \,\omega$ in $L^\infty$. Finally, if $|L|=1$ then we deduce that $m-(i_1-\tilde{i}_1+|L+K|-1)$ is either $0$ or $1$, so we have to control $D_X\Theta \,\omega$ and $(D\Theta/R)\,\omega$ in $L^\infty$. In brief, we need to bound $D_X\Theta\,\omega$, $(D_X\Theta/R)\,\omega$ and $D_X^2\Theta\,\omega$ in $L^\infty$. 

By applying the Sobolev embedding and Lemmas \ref{lemma:YminweightedHm}-\ref{lemma:Marcos}, we get
$$
\|D_X\Theta \,\omega\|_{L^\infty}
\lesssim
\|D_X\Theta \,\omega\|_{H^2}
\lesssim
\|\Theta\|_{Z^3}\simeq\|\Theta\|_{Y^3}\lesssim e^{\Re\lambda\tau}
$$
and
$$
\|D_X^2\Theta\, \omega\|_{L^\infty}
\lesssim
\|D_X^2\Theta\,\omega\|_{H^2}
\lesssim
\|\Theta\|_{Z^4}\simeq\|\Theta\|_{Y^4}\lesssim e^{\Re\lambda\tau}.
$$
Finally, by Lemma \ref{propABCDGMKpp} and the Sobolev embedding,
$$
\left\|\frac{D_X\Theta}{R}\omega\right\|_{L^\infty}
\lesssim
\left\|\frac{D_X\Theta}{R}\right\|_{L^\infty(B)}+\left\|D_X\Theta \, \omega \right\|_{L^\infty(B^c)}
\lesssim
\|\Theta\|_{H^m}+\|D_X\Theta \,\omega\|_{H^2}
\lesssim
\|\Theta\|_{Y^m}\lesssim e^{\Re\lambda\tau}.
$$

We are left with the terms in \eqref{eq:Hquad1} of the form
$$\frac{\omega}{R^{m-j_1}}\partial^{I}\left(\frac{\tilde{V}_\phi}{R}\right)\partial^{J-I}\partial_\phi\Theta,$$
with $(0,0)\neq I\leq J$. 
These terms can be treated similarly to \eqref{eq:Hquad1:1} by applying the Leibniz rule to $\partial^I((\tilde{V}\cdot e_\phi)/R)$ and then using analogous bounds.
\end{proof}

\subsubsection{Quadratic term 2}

\begin{lemma}
It holds that
$$
\|\mathcal{H}_J^{\text{quad},2}\|_{L^2}
\lesssim e^{2\Re\lambda\tau}.
$$
\end{lemma}
\begin{proof}
Since the only difference between $\mathcal{H}_J^{\text{quad},2}$ and $\mathcal{H}_J^{\text{quad},1}$ is the replacement of $\Theta$ with $\Theta^{\text lin}$, the proof follows in an entirely analogous manner.
\end{proof}

\subsubsection{Quadratic term 3}

\begin{lemma}
It holds that
$$
\|\mathcal{H}_J^{\text{quad},3}\|_{L^2}
\lesssim e^{2\Re\lambda\tau}.
$$
\end{lemma}
\begin{proof}
We have 
$$
\tilde{V}\cdot\nabla\left(\frac{\partial^J\Theta^{\text{lin}}}{R^{m-j_1}}\omega\right)
=\tilde{V}_R\partial_R\left(\frac{\partial^J\Theta^{\text{lin}}}{R^{m-j_1}}\omega\right)
+\frac{\tilde{V}_R}{R}\partial_\phi\left(\frac{\partial^J\Theta^{\text{lin}}}{R^{m-j_1}}\omega\right).
$$
This can be directly bounded by the $Y^{m+1}$-norm of $\Theta^{\text{lin}}$ and the $L^\infty$-norm of $\tilde{V}/R$.
\end{proof}

\appendix

\section{}

\subsection{Preliminaries in Operator theory}\label{sec:Operatortheory}

In this section we prove Proposition \ref{prop:SpectralAnalysis}. 
We start by recalling some classical definitions and results in Operator theory. In general, we will consider a linear operator $A:D(A)\subset H\to H$ acting on some Hilbert space $H$, where $D(A)$ is the domain of $A$. For a fixed $H$, we will denote $\mathcal{L}$ and $\mathcal{K}$ by the space of bounded and compact operators on $H$ respectively. 
The \textit{spectrum} of $A$ is defined as  
$$
\sigma(A)=\{\lambda\in\C\,:\, (A-\lambda)\textit{ \rm is not invertible}\}.
$$
Let us suppose that $A$ is a  bounded operator. 
Then, $A$ is called a  \textit{Fredholm Operator} if both the kernel $\text{Ker}(A)$ and the cokernel $H/\text{Im}(A)$ are finite dimensional. In this case, the \textit{index} of $A$ is the integer
$$
\text{Ind}(A)
=\dim(\text{Ker}(A))
-\dim(H/\text{Im}(A)).
$$
We recall a classical result in Spectral theory: the stability of the index of Fredholm operators with respect to compact perturbations
(see e.g.~\cite[Theorem 5.26, Chapter IV]{Kato95}).

\begin{prop}\label{prop:StabilityFredholm}
	Let $A\in\mathcal{L}$ be a Fredholm operator and $K\in\mathcal{K}$. Then, $A+K\in\mathcal{L}$ is a Fredholm operator with $\text{Ind}(A+K)=\text{Ind}(A)$.
\end{prop}

Next, we recall several classical results in Semigroup theory. The first one gives a characterization of strongly continuous semigroups (see e.g.~\cite[Corollary 3.6, Chapter II]{EngellNagel00}).
\begin{prop}\label{prop:semigroupTFAE} 
Given $w\in\R$ and a linear operator $A$, the following are equivalent:
\begin{enumerate}[(i)]
    \item $A$ generates a strongly continuous semigroup satisfying $\|e^{\tau A}\|_{\mathcal{L}}\leq e^{w\tau}$ for all $\tau\geq 0$.
    \item $A$ is closed, densely defined, and for any $\lambda\in\C$ with $\Re\lambda>w$, it holds that
    $$
    \|(A-\lambda)^{-1}\|_{\mathcal{L}}\leq\frac{1}{\Re\lambda - w}.
    $$
\end{enumerate}
\end{prop}

\begin{defi}\label{defi:contraction} 
A linear operator $A$ generates a \textit{contraction semigroup} $e^{\tau A}$ if it satisfies the conditions in Proposition \ref{prop:semigroupTFAE} for $w=0$.
\end{defi}

Next, we recall that, given a strongly continuous semigroup $e^{\tau A}$ generated by an operator $A$,
its \textit{growth bound} is defined as
\begin{equation}\label{eq:growthbound}
\omega_0(A)
=\inf\{w\in\R\,:\,\text{there exists }C_w\geq 1\text{ such that }\|e^{\tau A}\|_{\mathcal{L}}
\leq C_w e^{w\tau}\text{ for all }\tau\geq 0\}.
\end{equation}
The second result from Semigroup theory that we need (see e.g.~\cite[Corollary 2.11, Chapter IV]{EngellNagel00}) relates \eqref{eq:growthbound} with the \textit{essential bound} 
$$
\omega_{\text{ess}}(A)
=\inf_{\tau >0}\frac{1}{\tau}\log\|e^{\tau A}\|_{\mathcal{L}/\mathcal{K}},
$$
where $\mathcal{L}/\mathcal{K}$ is the quotient space (the so-called Calkin algebra),
and the \textit{spectral bound}
$$
s(A)
=\sup\{\Re\lambda\,:\,\lambda\in\sigma(A)\}.
$$

\begin{prop}\label{prop:growthbound}
Given an operator $A$ generating a strongly continuous semigroup, it holds that
$$
\omega_0(A)
=\max\{\omega_{\text{ess}}(A),s(A)\}.
$$
Moreover,
$
\sigma(A)\cap \{\Re\lambda\geq w\}
$
is finite for any $w>\omega_{\text{ess}}(A)$.
\end{prop}

The third result in Semigroup theory that we need is the stability of the essential bound with
respect to compact perturbations (see e.g.~\cite[Proposition 2.12, Chapter IV]{EngellNagel00}).

\begin{prop}\label{prop:wessA+K}
Given an operator $A$ generating a strongly continuous semigroup, and $K\in\mathcal{K},$ it holds that
$$
\omega_{\text{ess}}(A+K)
=\omega_{\text{ess}}(A).
$$
\end{prop}

\subsubsection{Proof of Proposition \ref{prop:SpectralAnalysis}}\label{sec:prop:SpectralAnalysis}
Before embarking on the proof, we present four auxiliary lemmas.

\begin{lemma}\label{lemma:Tinv}
The resolvent map
\begin{equation}\label{eq:Tinv}
\begin{split}
(A_b-\lambda)^{-1}:
[0,\infty)\times\C_+& \rightarrow  \mathcal{L} \\
(b,\lambda) & \mapsto
\left(\Theta\mapsto-\int_0^\infty
e^{\tau(A_b-\lambda)}\Theta\dif\tau\right)
\end{split}
\end{equation}
is well defined with
\begin{equation}\label{eq:resolventbound}
\|(A_b-\lambda)^{-1}\|_{\mathcal{L}}\leq
\frac{1}{\Re\lambda}.
\end{equation}
For any $\Theta\in H$, the map
$(b,\lambda)\mapsto (A_b-\lambda)^{-1}\Theta$ is continuous.
\end{lemma}
\begin{proof}
The first statement follows directly from assumption \eqref{prop:SpectralAnalysis:1} and Proposition \ref{prop:semigroupTFAE}.
In fact, a simple integration by parts shows that \eqref{eq:Tinv} is the inverse of $(A_b-\lambda)$.
Alternatively, \eqref{eq:Tinv} is the Laplace transform of the semigroup $e^{\tau A_b}$. 

For the second statement, let $\Theta\in H$. Given $(b_0,\lambda_0)\in[0,\infty)\times\C_+$, we want to show that the following limit holds in $H$
$$
(A_b-\lambda)^{-1}\Theta\to
(A_{b_0}-\lambda_0)^{-1}\Theta,
$$
as $(b,\lambda)\to (b_0,\lambda_0)$. We have
\begin{align*}
\|((A_b-\lambda)^{-1}-
(A_{b_0}-\lambda_0)^{-1})\Theta\|_H
&\leq
\int_0^\infty
\|(e^{\tau(A_b-\lambda)}-e^{\tau(A_{b_0}-\lambda_0)})\Theta\|_H\dif\tau\\
&\leq 
\int_0^\infty
|e^{-\lambda\tau}-e^{-\lambda_0\tau}|
\|e^{\tau A_b}\Theta\|_H\dif\tau
&&=:I\\
&+\int_0^\infty e^{-\lambda_0\tau}
\|(e^{\tau A_b}-e^{\tau A_{b_0}})\Theta\|_H\dif\tau
&&=:J
\end{align*}
where we have added and subtracted $e^{\tau(A_b-\lambda_0)}\Theta$. By applying that the semigroup is contractive we show that
$$
I\leq 
\left|\frac{1}{\lambda}-\frac{1}{\lambda_0}\right|
\|\Theta\|_H
\to 0,
$$
as $\lambda\to\lambda_0$. For the second term we can apply the dominated convergence theorem and assumption \eqref{prop:SpectralAnalysis:1} to deduce that $J\to 0$
as $b\to b_0$.
\end{proof}

\begin{lemma}\label{lemma:Fredholm}
For any $\lambda\in\C_+$, the operator $(L_b-\lambda)$ is Fredholm with index zero.
Therefore, $\sigma(L_b)\cap\C_+$ consists of eigenvalues with finite multiplicity.
\end{lemma}
\begin{proof}
We split
\begin{equation}\label{eq:Llambdasplit}
L_b-\lambda
=(A_b-\lambda)+C.
\end{equation}
Since $\Re\lambda>0$, the first term is invertible by Lemma \ref{lemma:Tinv}. In particular, it is a Fredholm operator with index zero.
Since $C$ is compact, the operator \eqref{eq:Llambdasplit} is also Fredholm with index zero by Proposition \ref{prop:StabilityFredholm}. Thus, the second claim follows.
\end{proof}

\begin{lemma} 
The map
\begin{equation}\label{eq:C}
\begin{split}
	Q:
	[0,\infty)\times\C_+ & \rightarrow  \mathcal{K} \\
	(b,\lambda) & \mapsto
	(A_b-\lambda)^{-1}\circ C,
\end{split}
\end{equation}
is well defined and continuous. 
\end{lemma}
\begin{proof}
By applying Lemma \ref{lemma:Tinv} and that $C$ is compact, we deduce that $Q$ is well-defined.
The continuity of $Q$ in the operator norm follows from combining the fact that the resolvent map \eqref{eq:Tinv} is continuous for any fixed $\Theta$, and that $C$ can be approximated by finite rank operators. 
\end{proof}

\begin{lemma}\label{lemma:essentialbound}
The essential bound
satisfies 
$$
\omega_{\text{ess}}(L_b)\leq 0.$$
Therefore, $\sigma(L_b)\cap\{\Re\lambda>w\}$ is finite for any $w> 0$.
\end{lemma}
\begin{proof}
By applying assumption \eqref{prop:SpectralAnalysis:1}, that $C$ is compact, and Proposition \ref{prop:wessA+K}, we get that 
$$
\omega_{\text{ess}}(A_b+C)
=\omega_{\text{ess}}(A_b)
\leq 0.
$$
The second part of the statement follows from Proposition \ref{prop:growthbound}.
\end{proof}

With these preparations, we are ready to begin the proof Proposition \ref{prop:SpectralAnalysis}.
Since $L=L_0$, we know from Theorem \ref{thm:L} that there exists an eigenvalue $\lambda_0\in\sigma(L)\cap\C_+$ and an eigenfunction $0\neq W_0\in\text{Ker}(L-\lambda_0)$.
We claim that there exists $0<b<b_0$ satisfying
\begin{equation}\label{eq:PropLb}
\sigma(L_b)\cap\left\{\Re\lambda>\frac{\Re\lambda_0}{2}\right\}\neq\emptyset.
\end{equation}
We will prove by contradiction that necessarily \eqref{eq:PropLb} is true for some $0<b\leq b_0$. Without loss of generality we assume that $b_0\leq\frac{\Re\lambda_0}{4}$.

Let us suppose that \eqref{eq:PropLb} is false. Thus, $(L_b-\lambda)$ is invertible for all $\lambda \in  \C_+$ and $0<b\leq b_0$. We will first prove that in fact $(L_b-\lambda)^{-1}$ is continuous as a function of $(b, \lambda)$ restricted to a suitable compact domain.

By applying Lemma \ref{lemma:essentialbound}  to $b=0$, we can take a $\circlearrowleft$-oriented circle $\Gamma:\T\to\C_+\setminus\sigma(L)$ of radius strictly less than $b_0$ surrounding $\lambda_0$. By applying Lemma \ref{lemma:Tinv} coupled with \eqref{eq:C} and the decomposition
\begin{equation}\label{eq:composition} 
L_b-\lambda
=(A_b-\lambda )\circ(I+Q(b,\lambda)),
\end{equation}
our hypothesis implies that $I+Q$ maps continuously the compact set $[0,b_0]\times\Gamma(\T)$ into 
$$
\{A\in\mathcal{L}\,:\,A\text{ invertible}\}
=\bigcup_{N\in\N}
\{A\in\mathcal{L}\,:\,\|A^{-1}\|_{\mathcal{L}}< N\}.
$$
Since the last union forms an open cover, we deduce that there exists $N\in\N$ such that
$$
\|(I+Q(b,\lambda))^{-1}\|_{\mathcal{L}}\leq N,
$$
uniformly in $(b,\lambda)\in [0,b_0]\times\Gamma(\T)$, 
and thus, by applying \eqref{eq:resolventbound} and \eqref{eq:composition}, also
$$
\|(L_b-\lambda)^{-1}\|_{\mathcal{L}}
\leq \frac{N}{\Re\lambda}.
$$
By applying these bounds and the resolvent identity
\begin{align*}
(I+Q(b,\lambda))^{-1}-(I+C(b',\lambda'))^{-1}
=(I+Q(b,\lambda))^{-1}\circ(Q(b,\lambda)-C(b',\lambda'))\circ(I+C(b',\lambda'))^{-1},
\end{align*}
it follows that 
the map $(b,\lambda)\mapsto (I+Q(b,\lambda))^{-1}$ is continuous from $[0,b_0]\times\Gamma(\T)$ into $\mathcal{L}$. As a consequence, the same can be deduce for the map
$(b,\lambda)\mapsto (L_b-\lambda)^{-1}$.

Once continuity is obtained, we can consider the Riesz projection
$$
P_{b,\Gamma}=-\frac{1}{2\pi i}\int_\Gamma (L_b-\lambda)^{-1}\dif\lambda.
$$
and take limits under the integral sign by the dominated convergence theorem. Therefore, we have
$$
P_{b,\Gamma}(W_0)
\to P_{0,\Gamma}(W_0),
$$
as $b\to 0$. However, $P_{b,\Gamma}(W_0)=0$ for any $0<b\leq b_0$ by hypothesis, while $P_{0,\Gamma}(W_0)=W_0\neq 0$, which is a contradiction. 

Therefore, there exists $\lambda_b\in\sigma(L_b)\cap\{\Re\lambda>b_0\}$ for some $0<b\le b_0$. 
By Lemma \ref{lemma:Fredholm}, there exists $0\neq W_b\in\mathrm{Ker}(L_b-\lambda_b)$. This concludes the proof.

\subsection{Interchanging polar and Cartesian derivatives}\label{sec:cartesianpolar}

\begin{lemma}[From Cartesian to Polar derivatives]\label{lemma:cartesianpolarAppendix}

For any multi-index $K=(k_1,k_2)$ with $|K|>0$, it holds that
\begin{equation}\label{FromCartesiantoPolar}
\partial_X^Kf
=\sum_{0<|J|\leq|K|}p_J^K\frac{\partial_R^{j_1}\partial_\phi^{j_2}f}{R^{|K|-j_1}}
\end{equation}
where $p_J^K=p_J^K(\cos\phi,\sin\phi)$ is a $|K|$-homogeneous polynomial. 
Moreover, defining $p_{(0,0)}^{(0,0)}:=1$ and $p_J^K:=0$ if $j_1<0$, $j_2<0$ or $|J|>|K|$ we obtain the following recursive formulae for $p_J^K$:
\begin{align}
p_J^{(k_1+1,k_2)}&=-\sin\phi\,\partial_\phi p_J^K+\cos\phi\, p_{(j_1-1,j_2)}^K-\sin\phi \, p_{(j_1,j_2-1)}^K,\label{FromCartesiantoPolarj1plus1}\\
p_J^{(k_1,k_2+1)}&=\cos\phi\,\partial_\phi p_J^K+\sin\phi\, p_{(j_1-1,j_2)}^K+\cos\phi \, p_{(j_1,j_2-1)}^K,\label{FromCartesiantoPolarj2plus1}
\end{align}
for $(0,0)\leq J$ and $|J|\leq|K|+1$.
\end{lemma}
\begin{proof}
We will prove \eqref{FromCartesiantoPolar} by induction on $|K|$. If $|K|=1$ we have the following
\begin{align*}
\partial_x f&=\cos\phi \,\partial_{R} f-\frac{1}{R}\sin\phi \,\partial_{\phi} f,\\
\partial_y f&=\sin\phi \,\partial_{R} f+\frac{1}{R}\cos\phi \,\partial_{\phi} f.
\end{align*}
Now, suppose that the formula holds for every $|K|\leq n$. Then, given $|K|=n$,
\begin{align*}
&\partial_x^{k_1+1}\partial_y^{k_2}f=\partial_x\left(\partial_X^Kf\right)=\partial_x\left(\sum_{0<|J|\leq|K|}p_J^K\frac{\partial_R^{j_1}\partial_\phi^{j_2}f}{R^{|K|-j_1}}\right)\\
&=-\sum_{0<|J|\leq|K|}\sin\phi\,\partial_\phi p_J^K\frac{\partial_R^{j_1}\partial_\phi^{j_2}f}{R^{|K|+1-j_1}} +\sum_{1<|J|+1\leq|K|+1}p_J^K\left(\cos\phi\,\frac{\partial_R^{j_1+1}\partial_\phi^{j_2}f}{R^{|K|+1-(j_1+1)}} -\sin\phi\,\frac{\partial_R^{j_1}\partial_\phi^{j_2+1}f}{R^{|K|+1-j_1}}\right)
\end{align*}
from where \eqref{FromCartesiantoPolarj1plus1} is derived.
In the other case,
\begin{align*}
&\partial_x^{k_1}\partial_y^{k_2+1}f=\partial_y\left(\partial_X^Kf\right)=\partial_x\left(\sum_{0<|J|\leq|K|}p_J^K\frac{\partial_R^{j_1}\partial_\phi^{j_2}f}{R^{|K|-j_1}}\right)\\
&=\sum_{0<|J|\leq|K|}\cos\phi\,\partial_\phi p_J^K\frac{\partial_R^{j_1}\partial_\phi^{j_2}f}{R^{|K|+1-j_1}} +\sum_{1<|J|+1\leq|K|+1}p_J^K\left(\sin\phi\,\frac{\partial_R^{j_1+1}\partial_\phi^{j_2}f}{R^{|K|+1-(j_1+1)}} +\cos\phi\,\frac{\partial_R^{j_1}\partial_\phi^{j_2+1}f}{R^{|K|+1-j_1}}\right)
\end{align*}
from where \eqref{FromCartesiantoPolarj2plus1} is derived.
\end{proof}

\begin{lemma}[From Polar to Cartesian derivatives]\label{lemma:polarcartesianAppendix}
For any multi-index $J=(j_1,j_2)$ with $|J|>0$, it holds that
\begin{equation}\label{FromPolartoCartesian}
\partial_R^{j_1}\partial_\phi^{j_2}f
=\sum_{\substack{0<|K|\leq|J| \\ j_1\leq|K|}}
q_K^J R^{|K|-j_1}\partial_X^K f,
\end{equation}
where $q_K^J=q_K^J(\cos\phi,\sin\phi)$ is a $|K|$-homogeneous polynomial.
Moreover, defining $q_{(0,0)}^{(0,0)}:=1$ and $q_K^J:=0$ if $k_1<0$, $k_2<0$ or $|K|>|J|$ we obtain the following recursive formulae for $q_K^J$:
\begin{align}
q_K^{(j_1+1,j_2)}&=(|K|-j_1)q_K^J+\cos\phi\, q_{(k_1-1,k_2)}^J+\sin\phi \, q_{(k_1,k_2-1)}^J,\label{FromCartesiantoPolark1plus1}\\
q_K^{(j_1,j_2+1)}&=\partial_\phi q_K^J-\sin\phi\, q_{k_1-1,k_2}^J+\cos\phi \, q_{k_1,k_2-1}^J,\label{FromCartesiantoPolark2plus1}
\end{align}
for $(0,0)\leq K$ and $|K|\leq|J|+1$.
\end{lemma}
\begin{proof}
We will prove \eqref{FromPolartoCartesian} by induction on $|J|$. If $|J|=1$ we have the following formulae
\begin{align*}
\partial_R f&=\cos\phi \,\partial_{x} f+\sin\phi \,\partial_{y} f=\frac{x}{R}\cdot \nabla f,\\
\partial_\phi f&=-R\sin\phi \, \partial_{x} f+R\cos\phi \, \partial_{y} f=R\frac{x^\perp}{R} \cdot \nabla f.
\end{align*}
Now, suppose that the formula holds for every $|J|\leq n$. Then, given $|J|=n$,
\begin{align*}
\partial_R^{j_1+1}\partial_\phi^{j_2}f=&\partial_R\left(\partial_R^{j_1}\partial_\phi^{j_2}f\right)=\partial_R\left(\sum_{\substack{0<|K|\leq|J| \\ j_1\leq|K|}}
q_K^J R^{|K|-j_1}\partial_X^K f\right)\\
 =&\sum_{\substack{0<|K|\leq|J| \\ j_1+1\leq|K|}}
q_K^J (|K|-j_1) R^{|K|-(j_1+1)}\partial_X^K f\\
&+ \sum_{\substack{1<|K|+1\leq|J|+1 \\ j_1+1\leq|K|+1}}q_K^J R^{|K|+1-(j_1+1)}(\cos\phi\,\partial_x^{k_1+1}\partial_y^{k_2} f+\sin\phi\,\partial_x^{k_1}\partial_y^{k_2+1}f),
\end{align*}
from where one deduces \eqref{FromCartesiantoPolark1plus1}.
In the other case,
\begin{align*}
\partial_R^{j_1}\partial_\phi^{j_2+1}f=&\partial_\phi\left(\partial_R^{j_1}\partial_\phi^{j_2}f\right)=\partial_\phi\left(\sum_{\substack{0<|K|\leq|J| \\ j_1\leq|K|}}
q_K^J R^{|K|-j_1}\partial_X^K f\right)\\
 =&\sum_{\substack{0<|K|\leq|J| \\ j_1\leq|K|}}
\partial_\phi q_K^J R^{|K|-j_1}\partial_X^K f \\
&+ \sum_{\substack{1<|K|+1\leq|J|+1 \\ j_1+1\leq|K|+1}}q_K^J R^{|K|+1-j_1}(-\sin\phi\,\partial_x^{k_1+1}\partial_y^{k_2} f+\cos\phi\,\partial_x^{k_1}\partial_y^{k_2+1}f)
\end{align*}
from where we deduce \eqref{FromCartesiantoPolark2plus1}.
\end{proof}

\subsection*{Acknowledgments}

A.C.~acknowledges financial support from Grant PID2020-114703GB-I00 funded by MCIN/AEI/10.13039/501100011033 and by a 2023
Leonardo Grant for Researchers and Cultural Creators, BBVA Foundation. The BBVA Foundation
accepts no responsibility for the opinions, statements, and contents included in the project and/or
the results thereof, which are entirely the responsibility of the authors.
A.C.~and D.F.~ acknowledge financial support from the Severo Ochoa Programme
for Centers of Excellence Grant CEX2019-000904-S funded by MCIN/AEI/10.13039/501100011033. D.F.~acknowledges the financial support of QUAMAP and the ERC Advanced Grant 834728. D.F., F.M.~and M.S.~acknowledge support from PID2021-124195NB-C32. D.F.~and F.M.~acknowledge support from CM through the Line of Excellence for University Teaching Staff between CM and UAM.
F.M.~acknowledges support from the 
Max Planck Institute for Mathematics in the Sciences. 
M.S.~acknowledges support from grant PID2022-136589NB-I00 funded by MCIN/AEI/10.13039/501100011033 and by ERDF ({\it A way of making Europe}); and from ``Conselleria d'Innovaci\'o, Universitats, Ci\`encia i Societat Digital'' programme ``Subvenciones para la contratación de personal investigador en fase postdoctoral'' (APOSTD 2022),  ref. CIAPOS/2021/28. The four authors acknowledge support from Grants RED2022-134784-T and RED2018-102650-T funded by MCIN/AEI/10.13039/501100011033.

\bibliographystyle{abbrv}
\bibliography{Unstable_gSQG_vortex}

\begin{flushleft}
	\quad\\
	\'Angel Castro\\
	\textsc{Instituto de Ciencias Matemáticas, CSIC-UAM-UC3M-UCM\\
		28049 Madrid, Spain}\\
	\textit{E-mail address:} angel\_castro@icmat.es 
\end{flushleft}

\begin{flushleft}
	\quad\\
	Daniel Faraco\\
	\textsc{Departamento de Matem\'aticas, Universidad Aut\'onoma de Madrid, Instituto de
Ciencias Matem\'aticas, CSIC-UAM-UC3M-UCM\\
		28049 Madrid, Spain}\\
	\textit{E-mail address:} 
daniel.faraco@uam.es 
\end{flushleft}

\begin{flushleft}
	\quad\\
	Francisco Mengual\\
	\textsc{Max Planck Institute for Mathematics in the Sciences\\
		04103 Leipzig, Germany}\\
	\textit{E-mail address:} fmengual@mis.mpg.de
\end{flushleft}

\begin{flushleft}
	\quad\\
	Marcos Solera\\
	\textsc{Departament d’An\`alisi Matem\`atica, Universitat de Val\`encia\\
		46100 Burjassot, Spain}\\
	\textit{E-mail address:} marcos.solera@uv.es
\end{flushleft}

\end{document}